\documentclass{article}

\PassOptionsToPackage{numbers, sort&compress}{natbib}

\usepackage[preprint]{neurips_2026}

\usepackage[utf8]{inputenc} 
\usepackage[T1]{fontenc}    
\usepackage{hyperref}       
\usepackage{url}            
\usepackage{booktabs}       
\usepackage{amsfonts}       
\usepackage{nicefrac}       
\usepackage{microtype}      
\usepackage{xcolor}         

\usepackage{amsmath,amssymb,amsthm,mathtools}
\usepackage{bm}
\usepackage{algorithm}
\usepackage{float}
\usepackage{xcolor}
\usepackage{booktabs}
\usepackage[noend]{algpseudocode}
\usepackage{hyperref}
\usepackage{makecell}
\usepackage{cleveref}
\usepackage{multirow}
\usepackage{enumitem}
\usepackage{pgfplots}
\usepackage{longtable}
\usepackage{etoolbox}
\usepackage{rotating}
\usepackage{array}
\usetikzlibrary{shapes.geometric}

\newtheorem{definition}{Definition}
\newtheorem{theorem}{Theorem}
\newtheorem{lemma}{Lemma}
\newtheorem{proposition}{Proposition}
\newtheorem{corollary}{Corollary}
\theoremstyle{definition}
\definecolor{algoblue}{RGB}{0, 0, 150}

\newcommand{\R}{\mathbb{R}}

\newcolumntype{D}{>{$\,\displaystyle}l<{$}}

\title{A Parameter-Free First-Order Algorithm for Non-Convex Optimization with $\tilde{\mathcal{O}}(\epsilon^{-5/3})$ Global Rate}

\author{%
  Sichao Xiong\thanks{Corresponding author} \\
  Department of Computer Science\\
  University of Oxford\\
  Oxford, OX1 3QG \\
  \texttt{sichao.xiong@cs.ox.ac.uk} \\
  \And
  Sadok Jerad \\
  Mathematical Institue \\
  University of Oxford\\
  Oxford, OX2 6GG \\
  \texttt{sadok.jerad@maths.ox.ac.uk} \\
  \And
    Coralia Cartis \\
  Mathematical Institue \\
  University of Oxford\\
  Oxford, OX2 6GG \\
  \texttt{coralia.cartis@maths.ox.ac.uk} \\
}

\begin{document}

\maketitle

\begin{abstract}

We introduce \textsc{PF-AGD}, the first parameter-free, deterministic, accelerated first-order method to achieve  $\mathcal{O}(\epsilon^{-5/3}\log(1/\epsilon))$ oracle complexity bound when minimizing sufficiently smooth, non-convex functions; this is the best-known bound for first-order methods on smooth non-convex objectives. Unlike existing methods possessing this rate that require \emph{a priori} knowledge of smoothness constants, we use an adaptive backtracking scheme and a gradient-based restart mechanism to estimate local curvature. This yields a practical algorithm that matches best-known theoretical rates. Empirically, \textsc{PF-AGD} outperforms the practical variant of \textsc{AGD-Until-Guilty} (\citet{pmlr-v70-carmon17a}), as well as other parameter-free variants, and is a viable alternative to nonlinear conjugate gradient methods.
\end{abstract}

\section{Introduction}
Optimization algorithms are the workhorses of modern machine learning. A central question in optimization theory is the characterization of an algorithm's \emph{evaluation complexity}: for an objective function $f:\mathbb{R}^d\to\mathbb{R}$, the number of iterations and oracle calls required to achieve $\epsilon$-stationarity, i.e., a point $x$ such that $\|\nabla f(x)\|\le \epsilon$ (see \citet{nesterov2018lectures} and \citet{cartis2022evaluation} for an overview).

\paragraph{Problem Formulation.} We study unconstrained optimization problems of the form $\min_{x \in \mathbb{R}^d} f(x)$, where $f:\mathbb{R}^d\to\mathbb{R}$ is a potentially non-convex function bounded below. We assume that $f$ is $L_1$-smooth (i.e., $\nabla f$ is $L_1$-Lipschitz continuous) and has $L_3$-Lipschitz third-order derivatives, but $L_1$ and $L_3$ are unknown; in this sense, our method is parameter-free. We consider \emph{deterministic first-order} algorithms that access $f$ and $\nabla f$ through an oracle and measure complexity by the number of $\nabla f$ oracle calls to compute an $\epsilon$-stationary point. While our theoretical framework is established under standard Lipschitz smoothness assumptions, these global assumptions can be significantly relaxed; the method only requires smoothness along the actual path taken by the iterates. This enables application to problems where global smoothness fails but local smoothness holds along the iterates' trajectory.

\paragraph{Related Work.} We focus on non-convex optimization; for an extensive treatment of the convex counterpart, we refer the reader to \cite{Aspremont21}. For non-convex $L_1$-smooth objectives, steepest descent requires $\mathcal{O}(\epsilon^{-2})$ iterations to reach an $\epsilon$-stationary point \cite{alma990195539480107026}. A seminal development in non-convex optimization was the \textsc{AGD-Until-Guilty} framework \citep{pmlr-v70-carmon17a}, which pioneered the application of accelerated gradient techniques to non-convex settings. Although $f$ may not be strongly convex, \citet{pmlr-v70-carmon17a} optimize a quadratically regularized objective assuming that it is strongly convex; if the anticipated progress bounds do not hold, the algorithm issues a certificate of non-convexity and exploits negative curvature to guarantee sufficient decrease. Under $L_2$-Lipschitz Hessians and $L_3$-Lipschitz third derivatives, \textsc{AGD-Until-Guilty} variants achieve the first rates surpassing $\mathcal{O}(\epsilon^{-2})$, namely, of $\tilde{\mathcal{O}}(\epsilon^{-7/4})$ and $\tilde{\mathcal{O}}(\epsilon^{-5/3})$, respectively.

A key limitation of \textsc{AGD-Until-Guilty} is that it requires problem parameters ($L_1$ and $L_3$) to calibrate step sizes and momentum. In practice, these quantities are rarely known and difficult to estimate; moreover, the theoretically specified algorithm is impractical to run as the global constants are usually overly pessimistic. Indeed, \citet{pmlr-v70-carmon17a} report experiments using a practical variant of their algorithm where na\"ive backtracking invalidates the acceleration guarantees derived for fixed constants. Figure~\ref{fig:vanilla_AGD} illustrates this disparity: for problems where global $L_1$ and $L_3$ can be estimated, the theoretical (vanilla) implementation is up to $40\times$ more costly and fails to exploit negative curvature.

\begin{figure}[t]
  \centering
  \includegraphics[width=0.78\linewidth]{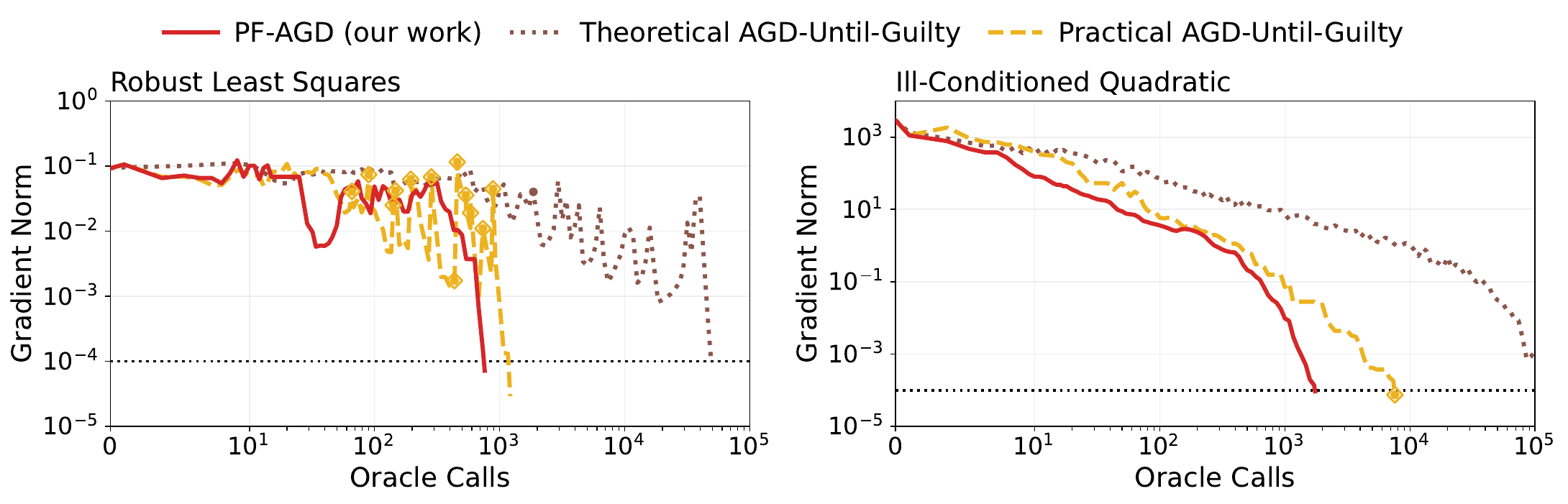}
  \caption{Theoretical and Practical \textsc{AGD-Until-Guilty} vs. \textsc{PF-AGD} with $\epsilon=10^{-4}$.}
  \label{fig:vanilla_AGD}
\end{figure}

\begin{table}[t]
\centering
\caption{Comparison of worst-case oracle complexity bounds for {\bf accelerated and quasi-Newton first-order} methods for \textbf{non-convex} optimization under different Lipschitz smoothness assumptions. ``Needs Constants?'' denotes which problem parameters must be known \textit{a priori} to achieve the stated bounds. The notation $\tilde{\mathcal{O}}(\cdot)$ suppresses poly-logarithmic factors in $\epsilon$ and the problem parameters when appropriate. For the regime of interest ($\epsilon\in (0,1)$ and sufficiently large $d$), the dimension-independent 
$\tilde{O}(\epsilon^{-5/3})$  is  best and best-known bound amongst the methods in Table \ref{table:complexity} and all first-order methods for non-convex optimization, including steepest descent methods which satisfy the tight bound of $\mathcal{O}(\epsilon^{-2})$~\cite{alma990195539480107026, cartis2022evaluation}.}

\begin{tabular}{@{}l
c@{\hspace{2pt}}c
c@{\hspace{2pt}}c@{\hspace{2pt}}c
c
l
c@{}}
\toprule
\multirow{2}{*}{Reference}
& \multicolumn{2}{c}{Oracle}
& \multicolumn{3}{c}{Lipschitz}
& \multirow{2}{*}{\makecell[c]{Conv. Rate to\\$\|\nabla f(x)\|\le\epsilon$}}
& \multirow{2}{*}{\makecell[c]{Needs\\Constants?}}
\\[-0.25ex]
\cmidrule(lr){2-3}\cmidrule(lr){4-6}
& $f(x)$ & $\nabla f(x)$
& $L_1$  & $L_2$         & $L_3$
& & \\
\midrule
\textnormal{[\citenum{pmlr-v70-carmon17a}, Theorem 1]} & \checkmark & \checkmark        & \checkmark & \checkmark & --        & $\tilde{\mathcal{O}}(\epsilon^{-7/4})$ & $L_1,L_2,\epsilon$ \\

\textnormal{[\citenum{doi:10.1137/22M1540934}, Theorem 5.7]} & \checkmark & \checkmark        & \checkmark & \checkmark & --        & $\mathcal{O}(\epsilon^{-7/4})$ & -- \\
\textnormal{[\citenum{Mokhtari25}, Theorem 4.1]} & -- & \checkmark & \checkmark & \checkmark & -- & $\mathcal{O}(d^{1/4}\epsilon^{-13/8})$ & $L_1,L_2,f^*,\epsilon$ \\
\textnormal{[\citenum{Marumo25}, Theorem 1]} & \checkmark & \checkmark & \checkmark & \checkmark & -- & $\tilde{\mathcal{O}}(d^{1/4}\epsilon^{-13/8})$ & -- \\
\textnormal{[\citenum{Grapiglia25}, Corollary 1]} & \checkmark & \checkmark & -- & \checkmark & -- & $\mathcal{O}(d^{1/2}\epsilon^{-3/2})$ & $\epsilon$ \\
\textnormal{[\citenum{pmlr-v70-carmon17a}, Theorem 2]} & \checkmark & \checkmark         & \checkmark & --        & \checkmark        & $\tilde{\mathcal{O}}(\epsilon^{-5/3})$ & $L_1,L_3,\epsilon$ \\
\textnormal{[\textbf{This work}, Theorem~\ref{Thm1}]}   & \checkmark & \checkmark        & \checkmark & --        & \checkmark        & $\tilde{\mathcal{O}}(\epsilon^{-5/3})$ & $\epsilon$ \\
\bottomrule
\end{tabular}\label{table:complexity}
\end{table}

Beyond \citet{pmlr-v70-carmon17a}, a number of approaches have sought to remove the polylogarithmic dependence in the $\tilde{\mathcal{O}}(\epsilon^{-7/4})$ bound
and achieve parameter-free convergence under $L_2$-Lipschitz Hessian assumptions \cite{doi:10.1137/22M1540934}. Quasi-Newton methods \citep{Mokhtari25, Marumo25} exploit Hessian-Lipschitz continuity to improve upon the $\mathcal{O}(\epsilon^{-7/4})$ bound, albeit with explicit dimensional dependence, while finite-difference schemes \citep{Grapiglia25, Cartis12} further reduce the rate to $\mathcal{O}(\epsilon^{-3/2})$ at the cost of stronger dependence on the problem dimension in the bound. All bounds explicitly depending on dimension $d$ will worsen with increasing problem size, while the  dimension-independent 
bounds are unaffected by such changes, thus justifying our focus here. 
While nonlinear Conjugate Gradient (CG) remains the empirical state-of-the-art for non-convex landscapes, it has historically lacked rigorous complexity bounds. Recent work by \citet{Royer22} provided the first such guarantees for CG; however, a significant theoretical gap remains as the global rate bound can be as large as that for Steepest Descent (SD) methods (namely, $\mathcal{O}(\epsilon^{-2})$). 

Despite these advances, no prior method attains the optimal $\tilde{\mathcal{O}}(\epsilon^{-5/3})$ rate without \emph{a priori} knowledge of the third-order smoothness constant $L_3$. Table~\ref{table:complexity} summarizes these developments and contextualizes our work within the broader landscape.

\paragraph{Contributions.}
A central obstacle to achieving a parameter-free algorithm is that Nesterov acceleration relies on a tight coupling between the fixed step sizes and momentum. We circumvent this limitation with \textsc{PF-AGD}, a deterministic first-order algorithm, and prove that it preserves the optimal rate without prior knowledge of problem constants. Concretely, our contributions are:

\begin{itemize}
    \item We show that optimal accelerated rates for non-convex optimization do not require prior knowledge of global smoothness constants, giving the first \emph{deterministic parameter-free} algorithm that attains an $\tilde{\mathcal{O}}(\epsilon^{-5/3})$ gradient complexity. Notably, the bound is independent of the problem dimension and is suitable for problems where $d$ may be large.

    \item Unlike existing methods that require problem constants $L_1$ and $L_3$, we develop a novel backtracking and restart scheme that estimates local curvature using only gradient information.

    \item Our method is a practical alternative to nonlinear CG and is implemented faithfully to theory. Benchmarks on standard test functions and ML tasks also show that \textsc{PF-AGD} consistently outperforms \textsc{AGD-Until-Guilty} (\citet{pmlr-v70-carmon17a}) and~\citet{doi:10.1137/22M1540934}.
\end{itemize}

\section{PF-AGD Algorithm}\label{section: algo}
Conceptually, \textsc{PF-AGD} preserves acceleration by behaving like accelerated gradient descent with adaptive $L_1$ estimation when the function is locally strongly convex, and switches to negative curvature exploitation when this assumption fails. The algorithm comprises two loops: an \textbf{inner loop} (\S~\ref{subsec:inner_loop}), indexed by $t$, which performs accelerated gradient steps while estimating $L_1$ and detecting non-convexity; an \textbf{outer loop} (\S~\ref{subsec:outer_loop}), indexed by $k$, which exploits negative curvature and refines the estimate $M_k$ of the third-order Lipschitz constant $L_3$. A violation in the inner loop triggers early termination, returning a witness pair $(u,v)$ used to update $M_k$ and produce a new iterate. The proofs justifying our design choices are deferred to Appendix~\ref{app:proofs}.

\paragraph{Adaptive Backtracking for $L_1$.} Following \citet{cavalcanti2025adaptive}, at each inner iteration $k$ we maintain an estimate $L_1^{(k)} > 0$ of the gradient Lipschitz constant and take a gradient step of size $\alpha_k = 1/L_1^{(k)}$ in direction $d_k=-\nabla f(x_k)$. Standard Armijo backtracking accepts a step if
\begin{equation}\label{eq:armijo}
    f(x_k + \alpha_k d_k) 
    \;\le\; 
    f(x_k) + c\cdot\alpha_k \langle \nabla f(x_k), d_k \rangle, 
    \qquad c \in (0,1),
\end{equation}
and otherwise shrinks $\alpha_k \leftarrow \rho\,\alpha_k$ with fixed $\rho \in (0,1)$, but this fails to preserve the $\mathcal{O}(\sqrt{Q}\log(1/\varepsilon))$ convergence rate of accelerated methods. To recover this rate, we replace $\rho$ with an adaptive factor $\hat\rho$ based on how badly the condition is violated. The violation map is
\begin{equation}
    v(\alpha_k) 
    \;:=\; 
    \tfrac{f(x_k + \alpha_k d_k) - f(x_k)}
         {c\cdot \alpha_k \langle \nabla f(x_k), d_k \rangle}, \quad \hat{\rho}\bigl(v(\alpha_k)\bigr) 
    \;:=\; 
    \max\bigl\{
        \varepsilon_{\min},\; 
        \rho{\,\tfrac{1-c}{1 - c\cdot v(\alpha_k)}}
    \bigr\},
\end{equation}
so condition \eqref{eq:armijo} can be written as $v(\alpha_k) \ge 1$, and $\varepsilon_{\min} \ll 1$ prevents numerical instability. Intuitively, the shrinkage factor reduces step sizes more aggressively for larger violations. This bounds the Lipschitz estimate ($L_1^{(k)} \le \max \{L_1^{(0)}, \tfrac{L_1}{2(1-c)\rho} \}$), preventing step sizes from collapsing to zero.

\subsection{The Inner Loop}\label{subsec:inner_loop}
We maintain three iterates: \(y_t\) (gradient step), \(x_t\) (extrapolated point), and \(w^{\min}_t\), the best auxiliary descent point selected from $\{\zeta_0,\zeta_1,\ldots,\zeta_t\}$, where $\zeta_t$ is obtained by a gradient step from $y_t$. We write $x_0^t := (x_0, x_1, \ldots, x_t)$ for the sequence of iterates through index $t$; similarly for $y_0^t$. The inner loop, \textsc{Modified-AGD}, initially assumes that $f$ is pathwise $\sigma$-strongly convex, i.e., that
\begin{equation}
    f(u) \ge f(v) + \langle \nabla f(v),\, u - v \rangle + \tfrac{\sigma}{2}\|u - v\|^2
    \label{eq:pathwise-sc}
\end{equation}
holds for all $(u,v) \in \{(y_s, x_s),(w,x_s),\, (w, y_s),\, (y_s, w)\}$ and $s = 0,\ldots,t-1$, where $w \in \mathbb{R}^d$ is the vector returned by \textsc{Certify-Progress}. Unlike global strong convexity, \eqref{eq:pathwise-sc} only requires convexity-like behavior along the trajectory of iterates. This relaxation makes it possible for \textsc{Find-Witness} to detect violations directly from observed iterates. Under this assumption, the algorithm runs accelerated steps via \textsc{AGD-Step} while \textsc{Certify-Progress} monitors for violations of \eqref{eq:pathwise-sc}. If a violation is detected, a witness pair $(u,v)$ is returned; if $f(y_t)>f(y_0)$, it returns \textsc{Restart}; otherwise, it returns \textsc{Null}. In effect, the witness pair provides a concrete certificate that the function violates local convexity, which the outer loop can then exploit.

\paragraph{\textsc{AGD-Step}.} Algorithm~\ref{alg:agd-step} estimates $L_1$ via Adaptive Backtracking (Algorithm~\ref{alg:ABLS}), updates the momentum parameter $\omega$, and takes a Nesterov step to produce $(x_t, y_t)$. It then performs an additional line search on $\zeta_t := y_t - \frac{1}{L^{(t)}_1}\nabla f(y_t)$ to ensure the descent condition
\begin{equation}
    f(w^{\min}_t) \;\leq\; f(\zeta_t) \;\leq\; f(y_t) - \tfrac{1}{2L_1^{(t)}}\|\nabla f(y_t)\|^2
    \label{eq:descent}
\end{equation}
holds; if violated, $L^{(t)}_1$ is multiplied by $\gamma > 1$ until the condition is satisfied (guaranteed whenever $L_1^{(t)} \geq L_1$). The condition is sufficient to give us the following progress bounds. 

\paragraph{Lyapunov Analysis.}
Our inner-loop progress bounds use the following Lyapunov potential
\begin{equation}
    V^{(w)}_t(y_k, x_k) := f(y_k) - f(w) + \tfrac{\sigma}{2}\bigl\|z_t(y_k,x_k) - w\bigr\|^2,
    \quad z_t(y,x) := x + \sqrt{Q_{t-1}}\,(x - y),
    \label{eq:lyapunov}
\end{equation}
where $w \in \mathbb{R}^d$ is a reference point and $Q_t := L_1^{(t)}/\sigma$. In \textsc{Modified-AGD}, the reference point is chosen adaptively from the iterates,
either $y_0$ in line~1 of \textsc{Certify-Progress} or $w^{\min}_t$ in line~6.
For a general reference point satisfying assumption~\eqref{eq:pathwise-sc}, the recursion acquires an additive slack:
\begin{equation}
    V^{(w)}_{k+1}(y_k, x_k)
    \;\leq\;
    \tfrac{Q_k^{3/2}}{Q_{k-1}^{3/2}}\,V^{(w)}_k(y_k, x_k)
    \;+\;
    \tfrac{Q_k^{3/2}}{\sigma\, Q_{k-1}^{3/2}}\,\bigl\|\nabla f(w)\bigr\|^2.
    \label{eq:lyapunov-slack}
\end{equation}

\begin{figure}[t]
\noindent
\begin{minipage}[t]{0.49\textwidth}
\begin{algorithm}[H]
\caption{\small Adaptive Line Search \cite{cavalcanti2025adaptive}}
\label{alg:ABLS}
\begin{algorithmic}[1]
\small
  \Require previous step size $\alpha_k > 0$, violation map $v : \mathbb{R}_+ \to \mathbb{R}$, adaptive factor $\hat{\rho} : \mathbb{R} \to (0,1)$
  \While{$v(\alpha_k) < 1$}
    \State $\alpha_k \leftarrow \hat{\rho}\bigl(v(\alpha_k)\bigr)\,\alpha_k$
  \EndWhile
  \State \textbf{output} $L^{(k)}_1 \gets 1/\alpha_k$
\end{algorithmic}
\end{algorithm}
\vspace{-13pt}
\begin{algorithm}[H]
\caption{
  \small\textcolor{algoblue}{\textsc{Modified-AGD}}%
  $(f,y_0,\varepsilon,L_1^{(0)},\sigma,\gamma)$}
\label{alg:agd-until-guilty}
\begin{algorithmic}[1]
\small
\State \textbf{Global} sequences $\{L^{(t)}_1\}_{t\ge 0}$
       with $L^{(0)}_1>\sigma>0$
\State $x_0 \gets y_0$,\;
       $w^{\min}_0 \gets y_0$,\;
       $Q_0 \gets L_1^{(0)}/\sigma$,\;
       $m \gets 0$
\For{$t = 1, 2, \ldots$}
  \State \parbox[t]{\linewidth}{%
    $(x_t, y_t, \zeta_t, Q_t^{\mathrm{agd}}) \gets $\\[1pt]
    \hspace{1.5em}$\textcolor{algoblue}{\textsc{AGD-Step}}(f,x_{t-1},y_{t-1},L_1^{(t-1)},\sigma,\gamma)$
  }\label{line:agd-step}
  \If{$Q_t^{\mathrm{agd}} > Q_{t-1}$} $m \gets m + 1$ \EndIf
  \State $z \gets x_t + {\scriptstyle \sqrt{Q^{\mathrm{agd}}_t}}\,(x_t - y_t)$
  \State $w^{\min}_t \gets
         \arg\min\bigl\{\,f(w^{\min}_{t-1}),\;f(\zeta_t)\bigr\}$
  \State \parbox[t]{\linewidth}{%
    $w_t \gets \textcolor{algoblue}{\textsc{Certify-Progress}}$\\[1pt]
    \hspace{1.5em}$(f,y_0,y_t,\sigma,Q_t^{\mathrm{agd}},t,m,w^{\min}_t)$, $Q_t \gets Q_t^{\mathrm{agd}}$
  }
  \If{$w_t = \textsc{Restart}$}
    \State \parbox[t]{\linewidth}{%
      $(x_t,\,y_t,\,\zeta_t,\,Q_t) \gets
       \textcolor{algoblue}{\textsc{Restart-Handler}}$\\[1pt]
      \hspace{1.5em}$(f,y_{t-1},z,Q_t^{\mathrm{agd}},L_1^{(t)},\sigma,\gamma)$%
    }\label{line:restart}
    \If{$Q_t > Q_t^{\mathrm{agd}}$} $m \gets m + 1$ \EndIf
    \State $w^{\min}_t \gets
           \arg\min\bigl\{\,f(w^{\min}_t),\;f(\zeta_t)\bigr\}$
  \EndIf
  \If{$w_t \notin \{\textsc{Null},\,\textsc{Restart}\}$}
    \State \parbox[t]{\linewidth}{
      $(u,v)\! \gets \!\textcolor{algoblue}{\textsc{Find-Witness}}(f,x_0^{t},y_0^{t},w_t,\sigma)$
    }
    \State \textbf{return} $(x_0^{t},y_0^{t},u,v)$
  \EndIf
  \If{$\|\nabla f(y_t)\|\le \varepsilon$} \textbf{return} $(x_0^{t},y_0^{t},\textsc{Null})$ \EndIf
\EndFor
\end{algorithmic}
\end{algorithm}
\vspace{-13pt}
\begin{algorithm}[H]
\caption{
  \small\textcolor{algoblue}{\textsc{Find-Witness}}$(f,x_0^{t},y_0^{t},w_t,\sigma)$}
\label{alg:find-witness}
\begin{algorithmic}[1]
\small
\For{$j=0,1,\dots,t-1$}
\For{$(u,v) \in \left\{\substack{
(y_j,x_j), (w_t,x_j),\\
(y_j,w_t), (w_t,y_j)
}\right\}$}
\If{eq.\ \eqref{eq:pathwise-sc} does not hold for $(u,v)$}
\State \textbf{return} $(u,v)$
\EndIf
\EndFor
\EndFor
\State (by Corollary~\ref{cor:modified_cor1_main} this line is never reached)
\end{algorithmic}
\end{algorithm}
\end{minipage}%
\hfill
\begin{minipage}[t]{0.49\textwidth}
\begin{algorithm}[H]
\caption{
  \small\textcolor{algoblue}{\textsc{AGD-Step}}
  $(f,x_{t-1},y_{t-1},L_1,\sigma,\gamma)$}
\label{alg:agd-step}
\begin{algorithmic}[1]
\small
\Repeat
  \State Adaptive backtracking for $L_1$
         via Algorithm~\ref{alg:ABLS}
  \State $Q\!\gets\!\frac{L_1}{\sigma}$,
         $\omega\! \gets\! \tfrac{\sqrt{Q}-1}{\sqrt{Q}+1}$;
         $y_t\! \gets\! x_{t-1}\! -\! \tfrac{\nabla f(x_{t-1})}{L_1}$
  \State $x_t\!\gets\!y_t + \omega\,(y_t - y_{t-1})$;
         $\zeta_t\!\gets\!y_t\!-\!\tfrac{1}{L_1}\nabla f(y_t)$
  \If{$f(\zeta_t) > f(y_t) - \tfrac{1}{2L_1}\,\|\nabla f(y_t)\|^{2}$}
    \State $L_1 \gets \gamma\, L_1$
  \EndIf
\Until{$f(\zeta_t) \le f(y_t) - \tfrac{1}{2L_1}\,\|\nabla f(y_t)\|^{2}$}
\State \parbox[t]{\linewidth}{
  $L_1^{(t)} \gets L_1$,\quad $Q_t \gets L_1/\sigma$
}
\State \Return $(x_t, y_t, \zeta_t, Q_t)$
\end{algorithmic}
\end{algorithm}
\vspace{-16pt}
\begin{algorithm}[H]
\caption{\mbox{}\\
  \small\textcolor{algoblue}{\textsc{Certify-Progress}}%
  $(f,y_0,y_t,\sigma,Q_t,t,m,w^{\min}_t)$}
\label{alg:certify-progress}
\begin{algorithmic}[1]
\small
\If{$f(y_t)\! > \!f(y_0)
    \!+\!\tfrac{2Q_t^{2}}{\sigma}\!\|\nabla f(y_0)\|^{2}$} \Return $y_0$
\EndIf
\If{$f(y_t) > f(y_0)$} \Return \textsc{Restart} \EndIf
\State $\psi(w^{\min}_t) \gets
    f(y_0) - f(w^{\min}_t)
    + \tfrac{\sigma}{2}\,\|w^{\min}_t - y_0\|^{2}$
\If{$\tfrac{\|\nabla f(y_t)\|^{2}}{2L_1^{(t)}}
    >
    (3Q_t)^{m}Q_t^{3/2}\psi(w^{\min}_t)
    e^{-t/\sqrt{Q_t}}$} \Return $w^{\min}_t$
\EndIf
\State \Return \textsc{Null}
\end{algorithmic}
\end{algorithm}
\vspace{-16pt}
\begin{algorithm}[H]
\caption{\mbox{}\\
  \small\textcolor{algoblue}{\textsc{Restart-Handler}}%
  $(f,y_{t-1},z,Q_t^{\mathrm{agd}},L_1,\sigma,\gamma)$}
\label{alg:restart-handler}
\begin{algorithmic}[1]
\small
\State $y_t \gets y_{t-1} - \tfrac{1}{L_1}\,\nabla f(y_{t-1})$
\While{$f(y_t) > f(y_{t-1})-\frac{1}{2L_1}\|\nabla f(y_{t-1})\|^{2}$}
  \State $L_1 \gets \gamma L_1;\quad
         y_t \gets y_{t-1} - \frac{1}{L_1}\,\nabla f(y_{t-1})$
\EndWhile
\State \parbox[t]{\linewidth}{%
  $x_t \gets (z + {\scriptstyle \sqrt{Q^{\mathrm{agd}}_t}}\, y_t)/(1 + {\scriptstyle \sqrt{Q^{\mathrm{agd}}_t}})$
}
\Repeat
  \State $\zeta_t \gets y_t - \tfrac{1}{L_1}\,\nabla f(y_t)$
  \If{$f(\zeta_t) > f(y_t) - \tfrac{1}{2L_1}\,\|\nabla f(y_t)\|^{2}$}
    \State $L_1 \gets \gamma\, L_1$
  \EndIf
\Until{$f(\zeta_t) \le f(y_t) - \tfrac{1}{2L_1}\,\|\nabla f(y_t)\|^{2}$}
\State \parbox[t]{\linewidth}{%
  $L_1^{(t)} \gets L_1$, $Q_t \gets L_1/\sigma$
}
\State \Return $(x_t, y_t, \zeta_t, Q_t)$
\end{algorithmic}
\end{algorithm}
\end{minipage}
\end{figure}

\paragraph{Technical contributions.} The Lyapunov analysis of \citet{cavalcanti2025adaptive} establishes the recursion only for $w = x^*$, where $\nabla f(w) = 0$ eliminates the gradient slack and $f(y_k) - f(x^*) \geq 0$ holds automatically. Both properties fail in our setting, since \textsc{Certify-Progress} must check progress against observable iterates ($y_0$ and $w^{\min}_t$) rather than against an unknown minimizer. We therefore extend the analysis along two tracks. For an \emph{arbitrary} reference point $w$ (Appendix~\ref{sec:lyap-arbitrary}), the recursion acquires an additive $\|\nabla f(w)\|^2$ slack as in~\eqref{eq:lyapunov-slack}; a concrete counterexample (Appendix~\ref{app:counter_example}) shows this slack cannot be removed in general. For a \emph{descent} reference point $w$ satisfying~\eqref{eq:descent}, the slack vanishes at the cost of a $(1 + 2\Delta_k)$ factor (Appendix~\ref{sec:lyap-descent}) whose product across iterations contributes the $(3Q_t)^m$ term in~\eqref{eq:progress_bound2}. This necessitates the introduction of $w^{\min}_t$, which satisfies~\eqref{eq:descent} by construction and enables the tighter bound. A second problem is that $f(y_t) > f(y_0)$ can occur during accelerated steps and break the monotonicity requirement for Corollary~\ref{cor:modified_cor1_main} \eqref{eq:iterate-bound}. We introduce a restart in this case. Naively overwriting $y_t$ destroys the Lyapunov chain because the auxiliary point $z_t(y_t, x_t)$ jumps. We resolve this with a momentum correction (Algorithm~\ref{alg:restart-handler}, line~4) chosen precisely so that $z_t(y_t, x_t) = z$ remains invariant before and after the restart, preserving the telescoping structure and the accelerated rate. Chaining these recursions inductively while accounting for possible restarts yields the following proposition; see Appendix~\ref{sec:partial-momentum} for the full proof.

\begin{proposition}[Progress bounds] \label{prop:progress}
Let $f:\mathbb{R}^d\to\mathbb{R}$ be $L_1$-smooth. Fix $w\in\mathbb{R}^d$ and assume $f$ is pathwise strongly convex. Given Algorithm~\ref{alg:agd-until-guilty}, $x_0=y_0$ and $L_1^{(0)}>\sigma$, with $c\in[1/2,1)$ and restarts handled by Algorithm~\ref{alg:restart-handler}, with $\psi(w)=f(y_0)-f(w)+\frac{\sigma}{2}\|w-y_0\|^2$:
\begin{equation}\label{eq:progress_bound1}
  f(y_t)-f(w)\le e^{-t/\sqrt{Q_t}}Q^{3/2}_t\psi(w)+\tfrac{2Q_{t}^2}{\sigma}\|\nabla f(w)\|^2,
\end{equation}
where $Q_t$ is the condition number at iteration~$t$. If additionally $w$ satisfies~\eqref{eq:descent},
\begin{equation}\label{eq:progress_bound2}
  f(y_t)-f(w) \le e^{-t/\sqrt{Q_t}}\;(3Q_t)^m \tfrac{Q_t^{3/2}}{Q_0^{3/2}}\;\psi(w),
\end{equation}
where $m\le\lfloor\log_{\min(\gamma, 1/\rho)}(L_1/L_1^{(0)})\rfloor+1$ counts the total number of condition-number increases.
\end{proposition}

The presence of restarts affects both bounds slightly: bound (7) acquires a factor of two in the slack term and bound (8) involves $m$. We track this in lines~5 and~11 of \textsc{Modified-AGD}. Next, we look at the key subroutine which detects violations based on the progress bound.

\paragraph{\textsc{Certify-Progress}.} 

Algorithm~\ref{alg:certify-progress} checks whether the iterates maintain sufficient progress under the pathwise strong convexity assumption, using the bounds of Proposition~\ref{prop:progress}. Line~1 checks bound~\eqref{eq:progress_bound1} with $w = y_0$. Line~4 checks bound~\eqref{eq:progress_bound2} with $w = w^{\min}_t$, which satisfies the descent condition~\eqref{eq:descent} by construction. It returns one of three outcomes: \textsc{Null} (no violation detected); a witness $w_t \in \{y_0, w^{\min}_t\}$ whose progress bound has been violated, triggering \textsc{Find-Witness}; or \textsc{Restart}, triggered when $f(y_t) > f(y_0)$, which is required for~\eqref{eq:iterate-bound} to apply.\vspace*{-0.35cm}

\paragraph{\textsc{Restart-Handler}.} A restart is triggered by $f(y_t) > f(y_0)$. Since $f(y_s) \le f(y_0)$ for all $s < t$ (otherwise a restart would have been triggered earlier), the algorithm backtracks to $y_{t-1}$ and takes a steepest descent step with line search. Provided $L_1^{(t)} \ge L_1$, this guarantees $f(y_t) < f(y_{t-1}) \le f(y_0)$, restoring the descent condition. To prevent the restart from breaking the Lyapunov recursion, the Lyapunov point $z$ must be kept invariant: the update $x_t \gets (z + \scriptstyle \sqrt{Q_t^{\text{agd}}} y_t) / (1 + \scriptstyle \sqrt{Q_t^{\text{agd}}})$ in line 4 of Algorithm~\ref{alg:restart-handler} achieves exactly this, ensuring $z_t(y_t, x_t) = z$ regardless of the new iterates.
\vspace*{-0.3cm}

\paragraph{\textsc{Find-Witness}.} 
Given the witness $w_t$ returned by \textsc{Certify-Progress}, \textsc{Find-Witness} iterates over previous indices $j<t$ searching for the exact pair $(u,v)$ violating~\eqref{eq:pathwise-sc}. At each step it checks the four candidate pairs $\{(y_j, x_j), (w_t, x_j), (y_j, w_t), (w_t, y_j)\}$; these are precisely the pairs whose strong convexity is asserted by assumption~\eqref{eq:pathwise-sc}. By Corollary \ref{cor:modified_cor1_main}, the search is guaranteed to succeed before all indices are exhausted, ensuring line 5 is never reached. This result also establishes the inner loop's iteration complexity and provides the non-convexity certificate. 

\begin{corollary}\label{cor:modified_cor1_main}
Let $f:\R^d\to\R$ be $L_1$-smooth, $y_0\in\R^d$, $\varepsilon>0$, and $0<\sigma\le L_1$. 
Let $(x_t,y_t,u,v)=\textcolor{algoblue}{\textsc{Modified-AGD}}(f,y_0,\varepsilon,L_1^{(0)},\sigma,\gamma)$.
Define $Q_t:=\tfrac{L_1^{(t)}}{\sigma}$, $\bar{L}_1:=\max \scriptstyle\left\{L_1^{(0)},\tfrac{L_1}{2(1-c)\rho}\right\}$,
$\bar{Q}:=\tfrac{\bar{L}_1}{\sigma}$. Then the number of AGD steps $t$ satisfies
\begin{equation}\label{eq: t_bound}
    t \;\le\; 1+\max\left\{0,\;
    \sqrt{\bar{Q}}\log\left(
    \tfrac{2L_1^{(t-1)} Q_{t-1}^{3/2}(3Q_{t-1})^m \psi(w_{t-1}^{\min})}{\varepsilon^2}
    \right)\right\},
\end{equation}
where $\psi(w)=f(y_0)-f(w)+\tfrac{\sigma}{2}\|w-y_0\|^2$ and
$m\le\lfloor\log_\gamma(\bar{L}_1/L_1^{(0)})\rfloor+1$.
If $u,v\neq\textsc{Null}$, then
\begin{equation}
    f(u) < f(v)+\langle\nabla f(v),u-v\rangle+\tfrac{\sigma}{2}\|u-v\|^2, \label{eq:sc-check}
\end{equation}
for some $0\le j<t$, where $(u,v)\in\{(y_j,x_j),(w_t,x_j),(y_j,w_t),(w_t,y_j)\}$.
Moreover,
\begin{equation}
    \max\{f(y_1),\dots,f(y_{t-1}),f(u)\}\le f(y_0). \label{eq:iterate-bound}
\end{equation}
\end{corollary}

\begin{proof}
The $\log$ factor in the bound comes from \eqref{eq: t_bound}. For $t=1$, the iteration bound \eqref{eq: t_bound} is immediate. For $t > 1$, since the algorithm did not terminate at iteration $t-1$, neither the line 16 condition of \textsc{Modified-AGD} nor the line 4 condition of \textsc{Certify-Progress} held. Thus $\varepsilon^2 < \|\nabla f(y_{t-1})\|^2\leq 2L^{(t-1)}_1Q_{t-1}^{3/2}{(3Q_{t-1})^m}\,\psi({w^{\min}_{t-1}})\,e^{-(t-1)/\sqrt{Q_{t-1}}}$, which gives \eqref{eq: t_bound} when rearranged. Proofs of \eqref{eq:sc-check} and \eqref{eq:iterate-bound} are deferred to Appendix~\ref{sec:inner_iters}. 
\end{proof}

\begin{algorithm}[t]
\caption{\small\textcolor{algoblue}{\textsc{PF-AGD}}$(f, p_0, L^{(0)}_1, M_0, \gamma, \epsilon)$}
\label{alg:PF-AGD}
\begin{algorithmic}[1]
\small
\For{$k = 1, 2, \dots$}
    \State $M_k \gets M_{k-1}$
    \While{\textbf{true}}
        \State $\alpha \gets 2M_k^{1/3}\epsilon^{2/3}$,\ $\tau \gets \sqrt{\alpha/(32M_k)}$,\ $\eta \gets \sqrt{2\alpha/M_k}$;\ $\hat{f}(x) := f(x) + \alpha\|x-p_{k-1}\|^2$
        \State $(x_0^t, y_0^t, u, v) \gets \textcolor{algoblue}{\textsc{Modified-AGD}}(\hat{f}, p_{k-1}, \epsilon/10, L^{(k-1)}_1, \alpha, \gamma)$
        \If{$(u, v) = \textsc{Null}$} \Comment{$\hat f$ effectively\ str.\ convex}
            \State $p_k \gets y_t$;\ \textbf{break}
        \EndIf
        \State $b^{(1)} \gets \textcolor{algoblue}{\textsc{Find-Best-Iterate}_3}(f, y_0^t, u, v)$;\ $b^{(2)} \gets \textcolor{algoblue}{\textsc{Exploit-NC-Pair}_3}(f, u, v, \eta)$
        \If{$f(b^{(1)}) \le f(y_0) - \alpha\tau^2$} \Comment{proxy check}
            \State $p_k \gets b^{(1)}$;\ \textbf{break}
        \ElsIf{$f(b^{(2)}) > \max\{f(v) - \tfrac{\alpha\eta^2}{4},\, f(u) - \tfrac{\alpha\eta^2}{12}\}$ \textbf{or} $f(v) > f(y_0) + 14\alpha\tau^2$}
            \State $M_k \gets \gamma M_k$
        \Else
            \State $p_k \gets \arg\min_{z \in \{b^{(1)}, b^{(2)}\}} f(z)$;\ \textbf{break}
        \EndIf
    \EndWhile
    \If{$\|\nabla f(p_k)\| \le \epsilon$} \Return $p_k$ \EndIf
\EndFor
\end{algorithmic}
\end{algorithm}

\begin{figure}[t]
\vspace{-16pt}
\noindent
\begin{minipage}[t]{0.49\textwidth}
\begin{algorithm}[H]
\caption{\small\textcolor{algoblue}{\textsc{Find-Best-Iterate}$_3$}%
  $(f, y_0^t, u, v)$}
\begin{algorithmic}[1]
\small
\State Let $0 \le j < t$ be such that $v = x_j$
\vspace{0.09em}
\State $c_j \leftarrow (y_j + y_{j-1})/2
       \quad \textbf{if } j > 0 \textbf{ else } y_0$
\vspace{0.08em}
\State $q_j \leftarrow -2y_j + 3y_{j-1}
       \quad \textbf{if } j > 0 \textbf{ else } y_0$
\State \textbf{return}
       $\arg\min_{z \in \{y_0, \dots, y_t, c_j, q_j, u\}} f(z)$
       
\end{algorithmic}
\end{algorithm}
\end{minipage}%
\hfill
\begin{minipage}[t]{0.495\textwidth}
\begin{algorithm}[H]
\caption{\small\textcolor{algoblue}{\textsc{Exploit-NC-Pair}$_3$}%
  $(f, u, v, \eta)$}
\begin{algorithmic}[1]
\small
\State $\delta \leftarrow (u - v) / \|u - v\|$
\State $\eta' \leftarrow \sqrt{\eta(\eta + \|u - v\|)} - \|u - v\|$
\State $u_+ \leftarrow u + \eta' \delta,\quad v_- \leftarrow v - \eta \delta$
\State \textbf{return} $\arg\min_{z\in\{v_-,\,u_+\}} f(z)$
\end{algorithmic}
\end{algorithm}
\end{minipage}
\end{figure}

\subsection{The Outer Loop}
\label{subsec:outer_loop}
Algorithm \ref{alg:PF-AGD} is the main method, building on the parameter-dependent \textsc{AGD-Until-Guilty} \cite{pmlr-v70-carmon17a} but unprecedentedly, estimating the third-order Lipschitz constant $L_3$. At each outer iteration $k$, we maintain an estimate $M_k$ and solve a regularized objective $\hat{f}(x) := f(x) + \alpha(M_k)\|x - p_{k-1}\|^2$ using the \textsc{Modified-AGD} inner loop to accuracy $\varepsilon=\epsilon/10$. If the inner loop returns $\|\nabla \hat{f}(y_t)\| \le \varepsilon$ with no negative curvature ($u, v = \text{NULL}$), $\hat{f}$ is treated as convex and the outer iterate advances. Otherwise, the certificate $(u,v)$ is used to assess whether $M_k$ is underestimating $L_3$ or whether negative curvature can be exploited to obtain sufficient progress. The algorithm computes two candidate points: $b^{(1)} = \textcolor{algoblue}{\textsc{Find-Best-Iterate}_3}(f, y_0^{t}, u, v)$ and $b^{(2)} = \textcolor{algoblue}{\textsc{Exploit-NC-Pair}_3}(f, u, v, \eta)$. When line~9 is satisfied, we already have sufficient progress (Lemma~\ref{lem:sufficient_progress}). Otherwise, we use the following inequalities as proxies for whether $M_k$ is large enough:
\begin{align}
\label{eq:proxy}
f(v)\le f(y_0)+14\alpha\tau^2 \; \text{and} \;  f(b^{(2)})\le \max\bigl\{f(v)-\tfrac{\alpha\eta^2}{4},f(u)-\tfrac{\alpha\eta^2}{12}\bigr\}.
\end{align}
The key insight is that whenever \eqref{eq:proxy} holds, we again get sufficient progress by Lemma~\ref{lem:sufficient_progress}. If it fails, we conclude that $M_k < L_3$ and update $M_k \gets \gamma M_k$; eventually, once $M_k \geq L_3$, the conditions are guaranteed to hold and the estimate is never increased again. In particular, the following lemma shows that $M_k$ is bounded and we have sufficient function decrease at each outer iterate. Here, $\alpha$ takes $M_k$ as input and we use $\alpha(M_k)$ to denote this dependence.

\begin{lemma}\label{lem:sufficient_progress}
Let $f : \mathbb{R}^d \to \mathbb{R}$ be $L_1$-smooth and have $L_3$-Lipschitz continuous third-order derivatives, let $\epsilon, \alpha, M_0 > 0$, $\gamma>1$ and $p_0 \in \mathbb{R}^d$. If $\{M_k\}_{k \geq 0}$ and $p_0^K$ are the sequences of estimates and iterates produced by \textcolor{algoblue}{\textsc{PF-AGD}}$(f, p_0, L_1^{(0)}, M_0, \gamma, \epsilon)$, then for every $1 \leq k < K$,
\begin{equation*}
f(p_k) \leq f(p_{k-1}) - \min \bigr\{ \tfrac{\epsilon^2}{5\alpha(M_k)}, \tfrac{\alpha(M_k)^2}{32M_k} \bigr\},\quad
    M_k \leq \bar{M} \coloneqq \max\{M_0,\, \gamma L_3\}.
\end{equation*}
In particular, the estimate $M_k$ is increased at most $\mathcal{O}(\log_\gamma(L_3 / M_0))$ times in total.
\end{lemma}

\section{Global  Rate of Convergence for PF-AGD}\label{sec: convergence_rate}
The total complexity is obtained by combining the number of outer iterations required to reduce the objective value (\S\ref{subsec:outer}), and the number of inner iterations needed to either certify sufficient progress or detect non-convexity (\S\ref{subsec:inner}). We bound these separately and combine them in \S\ref{subsec:main_bound}. Although the restart overwrites the previous iterate, it suffices to bound the total number of iterations $T$. The total number of gradient evaluations in one epoch of \textsc{Modified-AGD} with~$T$ iterations and~$b$ restarts is at most $G \le 2T + 2b + 2\bigl\lceil\log_\gamma(\bar{L}_1/L_1^{(0)})\bigr\rceil$. Since $b\le T$, this gives $G\le 4T + O(\log(\bar{L}_1/L_1^{(0)}))$.

\subsection{Bounding the Outer Iterations}\label{subsec:outer}
With the boundedness of $M_k$ in mind, we split our analysis into cases. If $M_0 \leq \gamma L_3$, we have
\begin{equation}\label{ineq1}
f(p_k) \leq f(p_{k-1}) - \min \bigl\{ \tfrac{\epsilon^2}{5\alpha(M_k)}, \tfrac{\alpha(M_k)^2}{32M_k} \bigr\} \leq f(p_{k-1}) - \min \bigl\{ \tfrac{\epsilon^2}{5\gamma^{1/3}\alpha(L_3)}, \tfrac{\alpha(L_3)^2}{32\gamma^{1/3} L_3} \bigr\}. 
\end{equation}

Otherwise, $M_k=M_0$ and we have
\begin{equation}\label{ineq2}
    f(p_k) \leq f(p_{k-1}) - \min \bigl\{ \tfrac{\epsilon^2}{5\alpha(M_0)}, \tfrac{\alpha(M_0)^2}{32M_0} \bigr\} = f(p_{k-1}) - \tfrac{\epsilon^{4/3}}{10M_0^{1/3}}.
\end{equation}

With these new progress bounds we can derive the upper bound $K$ of the number of iterations of \textsc{PF-AGD} by telescoping \eqref{ineq1}. Let $p_0 \in \mathbb{R}^d$, $\Delta_f = f(p_0) - \inf_{z \in \mathbb{R}^d} f(z)$, then $\Delta_f \ge  \sum_{k=1}^{K-1} \left( f(p_{k-1}) - f(p_k) \right) \ge (K-1) \gamma^{-1/3}\cdot \min \bigl\{ \tfrac{\epsilon^2}{5\alpha(L_3)}, \tfrac{\alpha(L_3)^2}{32L_3} \bigr\} \ge (K-1) \tfrac{\epsilon^{4/3}}{10\gamma^{1/3}L_3^{1/3}}$. 

In the case of \eqref{ineq2}, $\Delta_f \ge \sum_{k=1}^{K-1} \left( f(p_{k-1}) - f(p_k) \right)  \ge (K-1) \tfrac{\epsilon^{4/3}}{10M_0^{1/3}}$. We conclude that $K\leq 1+10\epsilon^{-4/3}\Delta_f\bar M^{1/3}$, where $\bar{M} := \max\{\gamma L_3,M_0\}$.

\subsection{Bounding the Inner Iterations}\label{subsec:inner}
To bound the number of steps $T$ of \textsc{Modified-AGD}, note that for every $w \in \mathbb{R}^d$
\[
\psi(w) = \hat{f}(y_0) - \hat{f}(w) + \tfrac{\alpha(M_k)}{2}\|w - y_0\|^2
= f(y_0) - f(w) - \tfrac{\alpha(M_k)}{2}\|w - y_0\|^2
\le \Delta_f .
\]

We take fixed backtracking constants $c=0.5,\rho = 0.8$, and $\gamma=2$ for Algorithm \ref{alg:ABLS}, so that $\frac{1}{2(1-c)\rho} \leq \gamma$, which gives $\bar L = \max\bigl\{L^{(0)},\tfrac{L_1+2\alpha(\bar M)}{2(1-c)\rho}, \gamma(L_1+2\alpha(\bar M)) \bigr\}= \max\bigl\{L^{(0)},2(L_1+2\alpha(\bar M))\bigr\}$ and $\bar Q := \tfrac{\bar L}{\alpha(M_0)}$. Substituting $\varepsilon = \epsilon/10$ and $\sigma = \alpha(M_k) = 2M_k^{1/3}\epsilon^{2/3}$ into Corollary~\ref{cor:modified_cor1_main} \eqref{eq: t_bound} we obtain,
\begin{align*}
T &\leq 1+\sqrt{\max\bigl\{Q_0,4 + \tfrac{L_1}{M_0^{1/3}\epsilon^{2/3}}\bigr\}} \log_+ \bigl( \tfrac{200\bar L\,\bar Q^{3/2}(3\bar Q)^m\Delta_f}{\epsilon^2} \bigr), 
\end{align*}
where $\log_+(\cdot)$ is shorthand for $\max\{0, \log(\cdot)\}$ and $m \leq \lfloor \log_{\min(\gamma, 1/\rho)} ((L_1+2\alpha(\bar M)) / L^{(0)}) \rfloor + 1$.

\subsection{Main Complexity Bound}\label{subsec:main_bound}
We assemble the final bound. The total gradient evaluations decompose as $2KT$, where $K$ is the number of outer iterations of \textsc{PF-AGD} and $T$ is the maximum number of accepted steps in any single call to \textsc{Modified-AGD}. Combining the results of $\S \ref{subsec:inner}$ and $\S \ref{subsec:outer}$ yields our main complexity bound.

\begin{theorem}\label{Thm1} Let $f : \mathbb{R}^d \to \mathbb{R}$ be $L_1$-smooth and have $L_3$-Lipschitz continuous third-order derivatives. Let $0 < \epsilon\le \min \{ \Delta_f^{3/4} \bar M^{1/4}, L_1^{3/2}/ (8 \bar M^{1/3})^{3/2} \}$. If we set $\sigma = \alpha(M_k) = 2 M_k^{1/3} \epsilon^{2/3}$, $\gamma = 2$, \textcolor{algoblue}{\textsc{PF-AGD}}$(f, p_0, L_1^{(0)}, M_0, \gamma, \epsilon)$ finds a point $p_K$ such that $\|\nabla f(p_K)\| \le \epsilon$ and requires at most
\begin{equation*}
27\tfrac{\Delta_fL^{1/2}_1\bar M^{1/3}}{M_0^{1/6}\epsilon^{5/3}}\log \left( \tfrac{1473L^{5/2+m}_1 (9/2)^m\,\Delta_f}{M_0^{1/2+m/3}\epsilon^{3+2m/3}} \right)\quad when \quad L^{(0)}\leq 2(L_1+2\alpha(\bar M))
\end{equation*}
\begin{equation*}
16\tfrac{\Delta_f L^{(0)^{1/2}}\bar M^{1/3}}{M_0^{1/6}\epsilon^{5/3}}\log \left( \tfrac{193L^{(0)^{5/2}}(3Q_0)^m\,\Delta_f}{M_0^{1/2}\epsilon^{3}} \right)\quad when \quad L^{(0)}\geq 2(L_1+2\alpha(\bar M))
\end{equation*}
gradient evaluations, where $m \leq \lfloor \log_{\min(\gamma, 1/\rho)} ((L_1+2\alpha(\bar M)) / L^{(0)}) \rfloor + 1$.
\end{theorem}
Ignoring problem-dependent constants, this result matches the best known $\mathcal{O}(\epsilon^{-5/3}\log(1/\epsilon))$ rate for deterministic first-order methods under third-order smoothness, while removing the need to know smoothness parameters in advance. The two-case structure arises because when $L^{(0)}$ is sufficiently large, it already serves as an adequate upper bound on the Lipschitz constant of the true objective, and the bound depends directly on $L^{(0)}$; for smaller initializations, it reduces to a bound expressed in terms of the true smoothness parameter $L_1$. Notably, the bound is independent of dimension $d$ but depends on $m$, which is logarithmically bounded and independent of $\epsilon$.

\section{Experiments}\label{section:experiments}
\begin{figure}[t]
  \centering
  \includegraphics[width=0.98\linewidth]{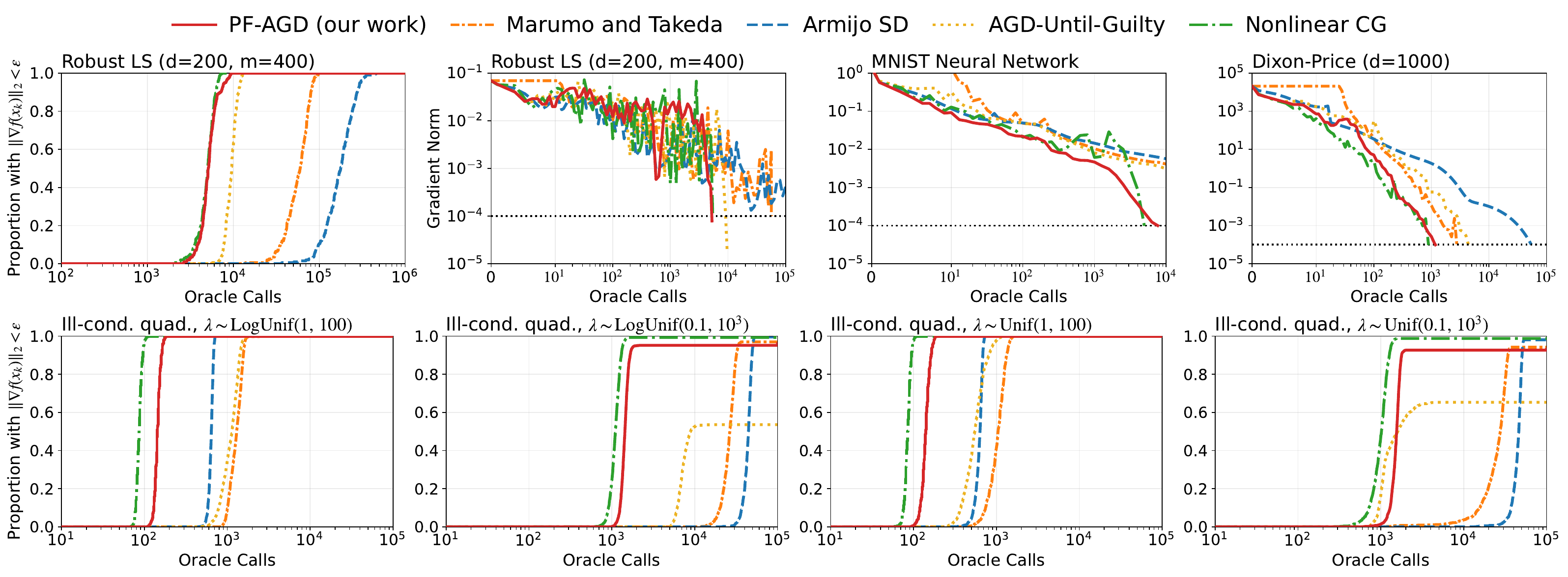}
  \caption{\textbf{PF-AGD vs. state-of-the-art first-order methods.} Top row: convergence on robust least squares, MNIST neural network, and the Dixon--Price function. Bottom row: robustness to condition number on ill-conditioned quadratics under various spectral distributions. All experiments are run on iMac M1 CPU. As all methods are first-order and dominated by gradient evaluations, oracle calls serve as a proxy for wall-clock time.}
\label{fig:results}
\end{figure}
We evaluate whether the theoretical advantages of \textsc{PF-AGD} translate into practical performance, particularly in settings where smoothness constants are unknown or difficult to estimate. Specifically, we compare \textsc{PF-AGD} to \textsc{AGD-Until-Guilty}~\cite{pmlr-v70-carmon17a}, nonlinear CG~\cite{PolakRibiere1969}, and contrast exploiting $L_3$-smoothness against parameter-free $L_2$-Lipschitz Hessian methods \cite{doi:10.1137/22M1540934}. Armijo steepest descent (SD)~\citep{Armijo1966MinimizationOF} is included as a baseline. All methods use accuracy $\varepsilon = 10^{-4}$ with fixed backtracking constants. We mitigate the number of restarts $m$ by initializing $L_1^{(0)}$ via a local finite-difference estimate of the gradient Lipschitz constant. Our benchmarks cover ML tasks, smooth objectives, and insufficiently-smooth ones; implementation specifics and further benchmarks on diverse landscapes (regularized quadratic, Qing \cite{Qing2006}, Rosenbrock \cite{Rosenbrock1960AnAM}, Ackley \cite{Ackley1987}, Powell \cite{Powell1962}, and \textsc{SCosine} functions) are given in Appendix~\ref{app:implementation} and their performances detailed in Table~\ref{tab:pfagd_empirical_summary}. Figure~\ref{fig:results} reports the oracle calls required to reach $\|\nabla f(x_k)\| < \varepsilon$ via empirical CDFs and gradient norm trajectories, where $x_k$ is the $k$-th oracle (outer loop) iterate. Across all problems, \textsc{PF-AGD} consistently outperforms \textsc{AGD-Until-Guilty} and is competitive against nonlinear CG. 

\paragraph{ML Tasks.}
We minimize the smooth biweight loss~\cite{Beaton1974} $f(x) = \frac{1}{m}\sum_{i=1}^{m}\phi(a_i^\top x - b_i)$ where $\phi(\theta) = \theta^2/(1+\theta^2)$ ($d=200$, $m=400$). The CDF across 1{,}000 seeds in Figure~\ref{fig:results} (top-left) reveals that \textsc{PF-AGD} performs comparably to nonlinear CG, though slightly less efficiently, while \textsc{AGD-Until-Guilty} requires nearly double the oracle calls. Moreover, the method of \citet{doi:10.1137/22M1540934} is even less practical, requiring more than $10^5$ oracle calls to fully converge. On a single seed (top, second panel), \textsc{PF-AGD} beats nonlinear CG with both methods converging comparably around 5{,}000 oracle calls; the former having significantly lower variance with both methods outperforming Armijo SD and \textsc{AGD-Until-Guilty}. We further evaluate performance on MNIST classification~\cite{mnist} using a fully connected neural network with (128, 64, 32) hidden units (12{,}074 parameters) and mean cross-entropy loss. All 60{,}000 training images are projected onto their top 10 principal components with Xavier-initialized weights \cite{pmlr-v9-glorot10a}. Figure~\ref{fig:results} (top, third panel) shows that, after $10^4$ oracle calls, only nonlinear CG and \textsc{PF-AGD} achieve the target gradient norm accuracy, with nonlinear CG exhibiting a slight advantage. Despite this, both methods follow nearly identical loss function trajectories (see Appendix~\ref{app:further_exp}, Figure~\ref{fig:mnist_nn}). Exact gradient computations currently limit our approach's applicability to neural network models of modern scale.

\paragraph{Quadratic Objectives.}
We minimize the quadratic $f(x) = \frac{1}{2}x^\top Hx + \mathbf{1}^\top x$, where $x\in\R^{100}$. The Hessian $H$ is a Haar-random rotation of a diagonal matrix. By varying the condition number $\kappa(H)$ from $10^2$ to $10^4$ across uniform and $\log$-uniform spectra, we observe in Figure~\ref{fig:results} (bottom row) that while all methods converge reliably at low $\kappa$, with nonlinear CG marginally outperforming \textsc{PF-AGD}, performance diverges as ill-conditioning increases. Notably, \textsc{AGD-Until-Guilty} stagnates below 70\% convergence and exhibits significant spectral sensitivity, characterized by an $8\times$ variation in median performance. In contrast, \textsc{PF-AGD} achieves a convergence rate exceeding 90\% and exhibits consistent performance across distributions, closing the gap to nonlinear CG as $\kappa$ grows.

\paragraph{Insufficiently-Smooth  Objectives.}
The Dixon--Price function~\citep{Dixon1989} $f(x) = (x_1-1)^2 + \sum_{i=2}^{d} i\,(2x_i^2 - x_{i-1})^2$ ($d=1{,}000$) violates $L_1$ and $L_3$-smoothness; we initialize experiments near the optimum at $x_0\sim \mathcal{N}(x^*, 10^{-1}I_d)$. Figure~\ref{fig:results} (top right), illustrates the evolution of the gradient norm for a single run. Nonlinear CG remains the most efficient method, with \textsc{PF-AGD} achieving comparable results; both methods reaching $10^{-4}$ accuracy within 200--300 oracle calls, notably outperforming both \cite{doi:10.1137/22M1540934} and \textsc{AGD-Until-Guilty}. This demonstrates the robustness of our approach beyond our assumptions.

\paragraph{Conclusion and Future Work.}
Our results suggest that theoretically optimal methods can also be practically competitive, with \textsc{PF-AGD} matching the empirical performance of nonlinear CG methods. Taken together, our results indicate that the dependence of accelerated non-convex methods on global smoothness constants is not fundamental. Moreover, comparison with \citet{doi:10.1137/22M1540934} suggests that leveraging $L_3$ yields an empirically more efficient algorithm than parameter-free methods assuming an $L_2$-Lipschitz Hessian. We note that negative curvature remains rarely exploited in practice, suggesting the loss function is ``effectively convex'' in large portions of the optimization trajectory, consistent with empirical observations of \citet{pmlr-v70-carmon17a}. This phenomenon merits further investigation. 

We identify three avenues for extending this work. Future research can focus on removing the logarithmic factor, a byproduct of the inner loop, via alternative acceleration schemes. \citet{doi:10.1137/22M1540934} offer a parameter-free approach without the logarithmic overhead; however, it necessitates iterate averaging and a non-standard momentum parameter of $\omega_k=k/(k+1)$, where $k$ counts steps within the current epoch. Crucially, their convergence analysis relies on a Jensen-type inequality for gradients, which we hypothesize cannot be directly generalized to third-order smoothness. While nonlinear CG is empirically robust for non-convex problems, theoretical guarantees are lacking. We defer the development of analysis for nonlinear CG on non-convex landscapes to future research. Finally, the current algorithm is deterministic and relies on exact gradient computations. This requirement limits scalability in modern deep learning contexts where mini-batching is standard (e.g., Adam \cite{Adam} and SGD \cite{Robbins&Monro:1951}). Future work will seek to extend this framework to the stochastic setting by integrating the stochastic line search techniques proposed by \citet{doi:10.1137/18M1216250, doi:10.1137/19M1291832}, and \citet{Vaswani2019PainlessSG}.

\section*{Acknowledgements}
Coralia Cartis  was supported by the Hong Kong Innovation and Technology Commission (InnoHK Project CIMDA) and by the EPSRC grant EP/Y028872/1, Mathematical Foundations of
Intelligence: An “Erlangen Programme” for AI. Sadok Jerad was supported by the Hong Kong Innovation and Technology Commission (InnoHK Project CIMDA).

\bibliography{references}
\nocite{*}

\newpage

\appendix

\section{Proofs}\label{app:proofs}

\subsection{Preliminaries}\label{sec:prelims}
In this section, we introduce notation and briefly overview the definitions and results we use throughout. We begin by characterizing the regularity of the functions under consideration.

\begin{definition}[\(L_p\)-Lipschitz \(p\)-th derivative]
\label{def:Lipschitz derivative}
Let $p\in\mathbb{N}$ and $L_p>0$. A function $f:\mathbb{R}^d\to\mathbb{R}$ is said to have an $L_p$-Lipschitz $p$-th derivative if $f$ is $p$-times continuously differentiable and
\[
\|\nabla^p f(x)-\nabla^p f(y)\|\le L_p\|x-y\|\qquad \text{for all } x,y\in\mathbb{R}^d,
\]
where $\|\cdot\|$ denotes the operator norm induced by the Euclidean norm. In particular, for $p=1$, we say that $f$ is \emph{$L_1$-smooth} if its gradient is $L_1$-Lipschitz continuous.
\end{definition}

While Lipschitz continuity of the derivatives provides an upper bound on how quickly the function can change, we also require a lower bound on the curvature to ensure the existence of a unique minimizer. This is captured by the notion of strong convexity.

\begin{definition}[$\sigma$-strong convexity]
\label{def:strong_convexity}
A function $f: \mathbb{R}^d \to \mathbb{R}$ is said to be $\sigma$-strongly convex with $\sigma > 0$ if, for all $x, y \in \mathbb{R}^d$, the following inequality holds:
\begin{equation*}
    f(y) \ge f(x) + \langle \nabla f(x), y - x \rangle + \tfrac{\sigma}{2} \|y - x\|^2.
\end{equation*}
\end{definition}

By combining the upper bound provided by $L_1$-smoothness and the lower bound provided by $\sigma$-strong convexity, we can characterize the overall ``difficulty'' of optimizing the function $f$. This relationship is formalized by the condition number.

\begin{definition}[Condition Number]
\label{def:condition_number}
Let $f: \mathbb{R}^d \to \mathbb{R}$ be $L_1$-smooth and $\sigma$-strongly convex. The condition number of $f$, denoted by $Q$, is defined as the ratio:
\begin{equation*}
    Q = \frac{L_1}{\sigma}.
\end{equation*}
Since $L_1 \ge \sigma$ by definition, it follows that $Q \ge 1$.
\end{definition}

\subsubsection{Adaptive Backtracking for Accelerated Gradient Descent}
For completeness, we state an additional assumption on $d_k$, noting that the steepest descent direction, $d_k = -\nabla f(x_k)$, is inherently gradient-related with $c_1 = c_2 = 1$.

\begin{definition}[Gradient related]
The directions $d_k$ are gradient related if there are $c_1 >0$ and $c_2>0$ such that $\langle\nabla f(x_k),-d_k\rangle \geq c_1 \|\nabla f(x_k)\|^2$ and $\|d_k\| \leq c_2\|\nabla f(x_k)\|$, for all $k \geq 0$.
\end{definition}

The following proposition establishes that under standard smoothness assumptions, the step sizes generated by the adaptive scheme in Algorithm~\ref{alg:ABLS} are strictly bounded away from zero, preventing the algorithm from stalling.

\begin{proposition}[Cavalcanti et al. {\cite[Proposition 9]{cavalcanti2025adaptive}}]\label{prop:1}
Let $f$ be $L_1$-smooth and $d_k$ gradient related. Given appropriate inputs, Algorithm~\ref{alg:ABLS} returns a step size $\alpha_k$ such that
\[
  \alpha_k \ge \min\Bigl\{\alpha_0,\,
    \rho \frac{2(1-c)c_1}{L_1 c_2^2}
  \Bigr\} > 0.
\]
\end{proposition}

Using the steepest descent direction and step size, we can infer the estimate of $L_1$ at iteration $k$ by taking the reciprocal of the step size $\alpha_k$. Since the latter is bounded below, our estimate of $L_1$ must be bounded above by some $\bar L_1$ given below.

\begin{corollary}
Defining $L_1^{(k)} := \frac{1}{\alpha_k}$ as the estimate of $L_1$ at iteration $k$ and setting $c_1,c_2 = 1$, we have
\[
  L_1^{(k)} \le
  \max\Bigl\{
    L_1^{(0)},\; \frac{L_1}{2(1-c)\rho}
  \Bigr\} := \bar{L}_1
\]
\end{corollary}
\begin{proof}
Applying Proposition~\ref{prop:1} with $c_1,c_2 = 1$ gives 
$\alpha_k \ge \min\Bigl\{\alpha_0,\, \rho \frac{2(1-c)}{L_1}\Bigr\}$
and we get the desired bound by taking the reciprocal as per the definition of $L_1^{(k)}$.
\end{proof}

It remains to combine adaptive line search with NAG, leading to Algorithm~\ref{alg:AGD}. For adaptive algorithms, we denote the $k$-th iteration estimate of the condition number as $Q_k := L_1^{(k)}/\sigma$ with the convention that $Q_{-1}:=Q_0$.

\begin{algorithm}[H]
\caption{Nesterov's Accelerated Gradient Descent (Adaptive)}
\label{alg:AGD}
\begin{algorithmic}[1]
  \State \textbf{Input:} $x_k, y_k \in \mathbb{R}^d$, $\nabla f(x_k)$,
    estimate $L_1^{(k)}$ of the Lipschitz constant with $L_1^{(k)} > \sigma > 0$
  \State \textbf{Output:} next points $x_{k+1}, y_{k+1}$
  \State $y_{k+1} \gets x_k - \frac{1}{L_1^{(k)}}\,\nabla f(x_k)$
  \State $\omega_k \gets \tfrac{\sqrt{L_1^{(k)}} - \sqrt{\sigma}}{\sqrt{L_1^{(k)}} + \sqrt{\sigma}}$
  \State $x_{k+1} \gets (1 + \omega_k)\,y_{k+1} - \omega_k\,y_k$
\end{algorithmic}
\end{algorithm}

To establish that adaptive line search preserves the accelerated rate of NAG, \citet{cavalcanti2025adaptive} employ a Lyapunov argument based on the function $V^{(w)}_t$ defined by
\begin{equation}
\label{lyapunov_function}
  V^{(w)}_t(y_k, x_k)
  = f(y_k) - f(w) + \frac{\sigma}{2}\,\|z_t(y_k,x_k) - w\|^2,
\end{equation}
where $w \in \mathbb{R}^d$ and the auxiliary point $z_t(y_k,x_k)$ is defined as
\begin{equation*}
  z_t(y_k,x_k) = x_k + \sqrt{Q_{t-1}}\, (x_k - y_k).
\end{equation*}

The first result gives a recursion for the iterates while keeping $Q_k$ fixed in the auxiliary point.
\begin{lemma}[Cavalcanti et al. {\cite[Lemma 1]{cavalcanti2025adaptive}}]\label{lem:app-AGD-Lyap-1}
    Let $f$ be $L_1$-smooth and $\sigma$-strongly convex. If the Lipschitz constant estimates \(L^{(k)}_1\) of accelerated gradient descent (Algorithm~\ref{alg:AGD}) are generated by adaptive backtracking (Algorithm~\ref{alg:ABLS}) with \(c\in\mathopen{[}1/2, 1\mathopen{)}\), $w=x^*$ and \(L^{(0)}_1>\sigma\), then for \(k\geq 0\)
    \begin{align*}
        (1+\delta_{k+1})V^{(x^*)}_{k+1}(y_{k+1}, x_{k+1}) - V^{(x^*)}_{k+1}(y_{k}, x_{k})
        \leq 0,
    \end{align*}
    where \(\delta_{k+1} = 1/(\sqrt{Q_{k}} - 1)\).
\end{lemma}
Note that the proof does not require $w=x^*$ to be a minimizer. Next we want to keep the iterates fixed and obtain a recursion for varying $t$ indices of $z_t(y_k,x_k)$, i.e., $Q_t$.

\begin{lemma}[Cavalcanti et al. {\cite[Lemma 2]{cavalcanti2025adaptive}}]\label{lem:app-AGD-Lyap-2}
    Let $f$ be $L_1$-smooth and $\sigma$-strongly convex. Given initial points \(x_{0}=y_{0}\), if the estimates \(L^{(k)}_1\) of the Lipschitz constant in accelerated gradient descent (Algorithm~\ref{alg:AGD}) are generated monotonically by adaptive backtracking (Algorithm~\ref{alg:ABLS}) with \(c\in\mathopen{[}1/2, 1\mathopen{)}\) and \(L^{(0)}_1>\sigma\), then for \(k\geq 0\)
    \begin{align*}
        V^{(x^*)}_{k+1}(y_{k}, x_{k})
        \leq \frac{Q_{k}^{2}}{Q_{k-1}^{2}}V^{(x^*)}_{k}(y_{k}, x_{k}).
    \end{align*}
\end{lemma}

This proof does rely on the optimality of $w=x^*$; specifically, the non-negativity of the objective gap $f(y_k) - f(x^*) \geq 0$ in (30) and (36), and implicitly, $\nabla f(x^*) = 0$ in (37).

\subsection{Generalized Lyapunov Analysis for Adaptive AGD}\label{sec:Lyapunov_analysis}
The analysis proceeds along two parallel tracks, distinguished by what is known about $w$. In the first track (Sections~\ref{sec:lyap-arbitrary} and~\ref{sec:conv-arbitrary}), we show that the Lyapunov recursion acquires an additional slack term proportional to $\|\nabla f(w)\|^2$. A concrete counterexample in Section~\ref{sec:lyap-arbitrary} shows this slack is necessary. In the second track (Sections~\ref{sec:lyap-descent} and~\ref{sec:conv-descent}), the additional assumption that $w$ satisfies the descent condition allows the slack term to be eliminated entirely, recovering a cleaner recursion. The first track feeds into the non-convexity certificate in line~1 of \textsc{Certify-Progress} with $w=y_0$; the second feeds into the stall detection in line~4 with $w=w^{\min}_t$.

\subsubsection{Lyapunov Recursion for Arbitrary Reference Points}\label{sec:lyap-arbitrary}
As previously mentioned, we no longer have the fact that $f(y_k) - f(w)$ is non-negative. Nevertheless, the following lemma provides us a good lower bound for our analysis.

\begin{lemma}
Let $f:\mathbb{R}^d\to\mathbb{R}$ be differentiable. Fix $w \in \mathbb{R}^d$. If for $k = 0,1,\ldots,t-1$ we have
\begin{equation}
\label{pathwise sc w}
f(y_k) \ge f(w) + \nabla f(w)^\top (y_k - w) + \frac{\sigma}{2}\lVert y_k - w \rVert^2.
\end{equation}

Then for all $w,y_k\in\mathbb{R}^d$,
\[
f(y_k)-f(w)\ge -\frac{1}{\sigma}\|\nabla f(w)\|^2.
\]
\label{lower_bound}
\end{lemma}

\begin{proof}
By the pathwise strong convexity assumption,
\[
f(y_k)-f(w)\ge \langle \nabla f(w),y_k-w\rangle+\frac{\sigma}{2}\|y_k-w\|^2.
\]
Using Young's inequality, for any $\gamma>0$,
\[
\langle a,b\rangle \ge -\frac{1}{2\gamma}\|a\|^2-\frac{\gamma}{2}\|b\|^2,
\]
and choosing $\gamma=\frac{\sigma}{2}$ gives
\[
\langle \nabla f(w),y_k-w\rangle \ge -\frac{1}{\sigma}\|\nabla f(w)\|^2-\frac{\sigma}{4}\|y_k-w\|^2.
\]
Substituting into the previous inequality yields
\[
f(y_k)-f(w)
\ge -\frac{1}{\sigma}\|\nabla f(w)\|^2+\frac{\sigma}{4}\|y_k-w\|^2 \ge -\frac{1}{\sigma}\|\nabla f(w)\|^2.
\]
\end{proof}

With this lemma we can derive the following recursion for the auxiliary point with an additional slack term proportional to $\|\nabla f(w)\|^2$.

\begin{lemma}\label{lemma-5}
Let $f$ be $L_1$-smooth. Fix $w \in \mathbb{R}^d$ and assume that the pathwise strong convexity condition in \eqref{pathwise sc w} holds. Given initial points $x_0 = y_0$, if the estimates $L^{(k)}_1$ of the Lipschitz constant in accelerated gradient descent (Algorithm~\ref{alg:AGD}) are generated monotonically by adaptive backtracking (Algorithm~\ref{alg:ABLS}) with $c \in [1/2,1)$ and $L^{(0)}_1 > \sigma$, then for $k \ge 0$,
\begin{equation}
    V^{(w)}_{k+1}(y_{k}, x_{k})
    \leq \frac{Q_{k}^{3/2}}{Q_{k-1}^{3/2}}V^{(w)}_{k}(y_{k}, x_{k})+\frac{Q_{k}^{3/2}}{\sigma Q_{k-1}^{3/2}}\|\nabla f(w)\|^2.
    \label{eq:Vk-recursion}
\end{equation}

In particular, when $w = x^*$ is a minimizer of $f$, $\nabla f(w)=0$ and~\eqref{eq:Vk-recursion} reduces to
\[
V^{(x^*)}_{k+1}(y_k,x_k) \;\le\; \frac{Q_k^{3/2}}{Q_{k-1}^{3/2}} V^{(x^*)}_k(y_k,x_k).
\]
\end{lemma}
\begin{proof}
    We prove the bound for each $k\geq 0$ directly. The case $k=0$ follows from \(Q_{-1}:=Q_{0}\),
    \begin{align*}
        z_{1}(y_0,x_0)
        = x_{0} + \sqrt{Q_{0}}(x_{0} - y_{0})
        = x_{0}
        = x_{0} + \sqrt{Q_{-1}}(x_{0} - y_{0})
        = z_{0}(y_0,x_0).
    \end{align*}
    Moreover, we have that
    \begin{align*}
        V^{(w)}_{1}(y_{0}, x_{0})
        &= f(y_{0}) - f(w) + \frac{\sigma}{2}\Vert z_{1}(y_0,x_0) - w \Vert^{2}
        \\
        &= \frac{Q_{0}^{3/2}}{Q_{-1}^{3/2}}\Bigl(f(y_{0}) - f(w) + \frac{\sigma}{2}\Vert z_{0}(y_0,x_0) - w \Vert^{2}\Bigr)
        \\
        &= \frac{Q_{0}^{3/2}}{Q_{-1}^{3/2}}V^{(w)}_{0}(y_{0}, x_{0}),
    \end{align*}
    which establishes the base case.
    To prove the inductive step, we divide the analysis into two cases, each representing a possible sign of \(\langle x_{k} - y_{k}, x_{k} - w \rangle\). 
    For each case, we bound
    \begin{align}
        D_k :=
        &\,\Vert x_{k} - w + \sqrt{Q_{k}}(x_{k} - y_{k}) \Vert^{2}-\Vert z_k(y_k,x_k) - w \Vert^{2}
        \nonumber\\
        =&\ 2(\sqrt{Q_{k}}-\sqrt{Q_{k-1}})\langle x_{k} - w,x_{k} - y_{k} \rangle
        + (Q_{k}-Q_{k-1})\Vert x_{k} - y_{k} \Vert^{2}.
        \label{id:app-AGD-Lyap-ascent-2norm-gap}
    \end{align}
    In turn, bounds on \eqref{id:app-AGD-Lyap-ascent-2norm-gap} translate into bounds on \(V^{(w)}_{k+1}(y_{k}, x_{k})-V^{(w)}_{k}(y_{k}, x_{k})\), since
    \begin{align}
        V^{(w)}_{k+1}(y_{k}, x_{k})-V^{(w)}_{k}(y_{k}, x_{k})
        =&\ \frac{\sigma}{2}D_k.
        \label{ineq:app-AGD-Lyap-ascent-V-gap}
    \end{align}
    Then, to prove the inductive step, we express bounds on \eqref{ineq:app-AGD-Lyap-ascent-V-gap} in terms of \(V^{(w)}_{k+1}\) and \(V^{(w)}_{k}\).
    
    \paragraph{Case 1: \(\langle x_k - y_k, x_k - w\rangle \ge 0\).}
    Since \(L^{(k)}_1\geq L^{(k-1)}_1\), then \(\sqrt{Q_{k-1}}/\sqrt{Q_{k}}\leq 1\), so that
    \begin{align*}
        \sqrt{Q_{k}}-\sqrt{Q_{k-1}}
        \leq \frac{Q_{k}}{\sqrt{Q_{k}}} - \sqrt{Q_{k-1}}\frac{\sqrt{Q_{k-1}}}{\sqrt{Q_{k}}}
        = \frac{Q_{k}-Q_{k-1}}{\sqrt{Q_{k}}}.
    \end{align*}
    Hence, applying the inequality above to \eqref{id:app-AGD-Lyap-ascent-2norm-gap} and then adding a non-negative \(\Vert x_{k} - w \Vert^{2}\) term to it, we get
    \begin{align}
        D_k
        \nonumber
        \leq&\ 2\frac{Q_{k}-Q_{k-1}}{\sqrt{Q_{k}}}\langle x_{k} - w,x_{k} - y_{k} \rangle
        + (Q_{k}-Q_{k-1})\Vert x_{k} - y_{k} \Vert^{2}
        + \frac{Q_{k}-Q_{k-1}}{Q_{k}}\Vert x_{k} - w \Vert^{2}
        \nonumber\\
        =&\ \frac{Q_{k}-Q_{k-1}}{Q_{k}}\Vert x_{k} - w + \sqrt{Q_{k}}(x_{k} - y_{k}) \Vert^{2}.
        \label{ineq:app-AGD-Lyap-ascent-case1-aux}
    \end{align}
    Plugging \eqref{ineq:app-AGD-Lyap-ascent-case1-aux} back into \eqref{ineq:app-AGD-Lyap-ascent-V-gap} yields
    \begin{align}
        V^{(w)}_{k+1}(y_{k}, x_{k})-V^{(w)}_{k}(y_{k}, x_{k})
        &\leq \frac{Q_{k}-Q_{k-1}}{Q_{k}}\frac{\sigma}{2}\Vert x_{k} - w + \sqrt{Q_{k}}(x_{k} - y_{k}) \Vert^{2}
        \nonumber\\
        &\nonumber = \frac{Q_{k}-Q_{k-1}}{Q_{k}}
        \Bigl(V^{(w)}_{k+1}(y_{k}, x_{k})
        - (f(y_{k}) - f(w))\Bigr)\\
        &\leq \frac{Q_{k}-Q_{k-1}}{Q_{k}}V^{(w)}_{k+1}(y_{k}, x_{k}) + \frac{Q_{k}-Q_{k-1}}{\sigma Q_{k}}\|\nabla f(w)\|^2,
        \label{ineq:app-AGD-Lyap-ascent-case1-aux1}
    \end{align}
    where the last inequality follows from Lemma~\ref{lower_bound}, namely \(f(y_{k}) - f(w)\geq -\frac{1}{\sigma}\|\nabla f(w)\|^2\).
    Thus, rearranging terms in \eqref{ineq:app-AGD-Lyap-ascent-case1-aux1} and then multiplying both sides by \(Q_{k}/Q_{k-1}\), we obtain
    \begin{align*}
        V^{(w)}_{k+1}(y_{k}, x_{k})&
        \leq \frac{Q_{k}}{Q_{k-1}}V^{(w)}_{k}(y_{k}, x_{k})+ \frac{Q_{k}-Q_{k-1}}{\sigma Q_{k-1}}\|\nabla f(w)\|^2\\
        &\leq \frac{Q^{3/2}_{k}}{Q^{3/2}_{k-1}}V^{(w)}_{k}(y_{k}, x_{k})+ \frac{Q^{3/2}_{k}}{\sigma Q^{3/2}_{k-1}}\|\nabla f(w)\|^2.
    \end{align*}
    
    \paragraph{Case 2: \(\langle x_k - y_k, x_k - w\rangle < 0\).}
    As in the previous case, we start by bounding the gap \eqref{id:app-AGD-Lyap-ascent-2norm-gap}. The assumption gives
    \begin{equation}
    \|y_k - w\|^2 = \|x_k - w\|^2 - 2\langle x_k - w, x_k - y_k\rangle + \|x_k - y_k\|^2 \;\ge\; \|x_k - w\|^2.
    \label{ineq:app-AGD-Lyap-ascent-yk-ub}
    \end{equation}

    Splitting \(Q_k - Q_{k-1} = \sqrt{Q_k}(\sqrt{Q_k} - \sqrt{Q_{k-1}}) + \sqrt{Q_{k-1}}(\sqrt{Q_k} - \sqrt{Q_{k-1}})\) and adding/subtracting \((\sqrt{Q_k} - \sqrt{Q_{k-1}})/\sqrt{Q_k} \cdot \|x_k - w\|^2\), \eqref{id:app-AGD-Lyap-ascent-2norm-gap} can be rewritten as
    \begin{align}
    D_k
    =& \,2\frac{\sqrt{Q_{k}}-\sqrt{Q_{k-1}}}{\sqrt{Q_{k}}}\langle x_{k} - w,\sqrt{Q_{k}}(x_{k} - y_{k}) \rangle		
        + \sqrt{Q_{k}}(\sqrt{Q_{k}}-\sqrt{Q_{k-1}})\Vert x_{k} - y_{k} \Vert^{2}
        \nonumber\\
        &+ \sqrt{Q_{k-1}}(\sqrt{Q_{k}}-\sqrt{Q_{k-1}})\Vert x_{k} - y_{k} \Vert^{2}
        + (1-1) \frac{\sqrt{Q_{k}}-\sqrt{Q_{k-1}}}{\sqrt{Q_{k}}}\Vert x_{k} - w \Vert^{2}
        \nonumber\\
    =&\, \frac{\sqrt{Q_k} - \sqrt{Q_{k-1}}}{\sqrt{Q_k}}\,\big\| x_k - w + \sqrt{Q_k}(x_k - y_k)\big\|^2       \nonumber\\
    &\quad + \sqrt{Q_{k-1}}(\sqrt{Q_k} - \sqrt{Q_{k-1}})\,\|x_k - y_k\|^2
       - \frac{\sqrt{Q_k} - \sqrt{Q_{k-1}}}{\sqrt{Q_k}}\,\|x_k - w\|^2.
    \label{id:app-AGD-Lyap-ascent-aux1}
    \end{align}

    Next, we use the following elementary inequality, which is a consequence of \(\Vert a/\alpha + b\alpha \Vert^{2} \geq 0\):
    \begin{align*}
        \Vert a - b \Vert^{2}
        = \Vert a \Vert^{2} - 2\langle a, b \rangle + \Vert b \Vert^{2}
        \leq (1 + 1/\alpha^{2})\Vert a \Vert^{2} + (1 + \alpha^{2})\Vert b \Vert^{2}.
    \end{align*}

    Since \(z_k(y_k, x_k) - x_k = \sqrt{Q_{k-1}}(x_k - y_k)\), we have \(Q_{k-1}\|x_k - y_k\|^2 = \|z_k(y_k, x_k) - x_k\|^2 = \|(z_k(y_k, x_k) - w) - (x_k - w)\|^2\). Applying the inequality with \(a = z_k(y_k, x_k) - w\), \(b = x_k - w\), \(\alpha^2 = \sqrt{Q_{k-1}}/\sqrt{Q_k}\),
    \begin{align}
    Q_{k-1}\|x_k - y_k\|^2
    &\le \Big(1 + \frac{\sqrt{Q_k}}{\sqrt{Q_{k-1}}}\Big)\|z_k(y_k, x_k) - w\|^2
       + \Big(1 + \frac{\sqrt{Q_{k-1}}}{\sqrt{Q_k}}\Big)\|x_k - w\|^2 \notag\\
    &= \frac{\sqrt{Q_k} + \sqrt{Q_{k-1}}}{\sqrt{Q_{k-1}}}\,\|z_k(y_k, x_k) - w\|^2
       + \frac{\sqrt{Q_k} + \sqrt{Q_{k-1}}}{\sqrt{Q_k}}\,\|x_k - w\|^2.
    \label{eq:young-bound}
    \end{align}
    
    Multiplying~\eqref{eq:young-bound} by \((\sqrt{Q_k} - \sqrt{Q_{k-1}})/\sqrt{Q_{k-1}}\) and using
    \((\sqrt{Q_k} - \sqrt{Q_{k-1}})(\sqrt{Q_k} + \sqrt{Q_{k-1}}) = Q_k - Q_{k-1}\) yields
    \begin{equation}
    \begin{aligned}
    &\sqrt{Q_{k-1}}(\sqrt{Q_k} - \sqrt{Q_{k-1}})\,\|x_k - y_k\|^2 \\
    &\qquad\le\; \frac{Q_k - Q_{k-1}}{Q_{k-1}}\,\|z_k(y_k, x_k) - w\|^2
           + \frac{\sqrt{Q_k} - \sqrt{Q_{k-1}}}{\sqrt{Q_k}}\cdot\frac{\sqrt{Q_k} + \sqrt{Q_{k-1}}}{\sqrt{Q_{k-1}}}\,\|x_k - w\|^2.
    \end{aligned}
    \label{ineq:app-AGD-Lyap-ascent-aux1}
    \end{equation}

    Substituting~\eqref{ineq:app-AGD-Lyap-ascent-aux1} into~\eqref{id:app-AGD-Lyap-ascent-aux1} and applying~\eqref{ineq:app-AGD-Lyap-ascent-yk-ub},
    \begin{equation}
    \begin{aligned}
    D_k
    &\le \frac{\sqrt{Q_k} - \sqrt{Q_{k-1}}}{\sqrt{Q_k}}\,\big\| x_k - w + \sqrt{Q_k}(x_k - y_k)\big\|^2 \\
    &\quad + \frac{Q_k - Q_{k-1}}{Q_{k-1}}\,\|z_k - w\|^2
       + \frac{\sqrt{Q_k} - \sqrt{Q_{k-1}}}{\sqrt{Q_{k-1}}}\,\|y_k - w\|^2.
    \end{aligned}
    \label{ineq:app-AGD-Lyap-ascent-aux2}
    \end{equation}

    Combining~\eqref{ineq:app-AGD-Lyap-ascent-aux2} with~\eqref{ineq:app-AGD-Lyap-ascent-V-gap} gives
    \begin{equation}
    V_{k+1}^{(w)}(y_k, x_k) - V_k^{(w)}(y_k, x_k)
    \le \frac{\sigma}{2}\Big(
       A\| x_k - w + \sqrt{Q_k}(x_k - y_k)\|^2
     + B\|z_k(y_k, x_k) - w\|^2
     + C\|y_k - w\|^2 \Big),
    \label{ineq:app-AGD-Lyap-ascent-aux3}
    \end{equation}
    where \(A := (\sqrt{Q_k} - \sqrt{Q_{k-1}})/\sqrt{Q_k}\),
    \(B := (Q_k - Q_{k-1})/Q_{k-1}\),
    \(C := (\sqrt{Q_k} - \sqrt{Q_{k-1}})/\sqrt{Q_{k-1}}\).

    It remains to express the three norms in terms of the Lyapunov values. By Lemma~\ref{lower_bound},
    \begin{align}
    \frac{\sigma}{2}\,\| x_k - w + \sqrt{Q_k}(x_k - y_k)\|^2
    &\le V_{k+1}^{(w)}(y_k, x_k) + \frac{1}{\sigma}\|\nabla f(w)\|^2,
    \label{ineq:app-AGD-Lyap-ascent-lb2}\\
    \frac{\sigma}{2}\,\|z_k(y_k,x_k) - w\|^2
    &\le V_k^{(w)}(y_k, x_k) + \frac{1}{\sigma}\|\nabla f(w)\|^2.
    \label{ineq:app-AGD-Lyap-ascent-lb}
    \end{align}

    Applying pathwise strong convexity with \eqref{pathwise sc w}, we obtain
    \begin{align}
        V^{(w)}_{k}(y_{k}, x_{k})
        &= f(y_{k}) - f(w) + \frac{\sigma}{2}\Vert z_k(y_k,x_k) - w\Vert^{2}\\
        &\geq \frac{\sigma}{2}\Vert y_{k} - w\Vert^{2}+ \frac{\sigma}{2}\Vert z_k(y_k,x_k) - w\Vert^{2} + \langle\nabla f(w),y_k-w\rangle.
        \label{ineq:app-AGD-Lyap-ascent-lb3}
    \end{align}
    We can further lower bound the inner product using Young's inequality. Fix $\gamma = \sigma/2$ and use
    \[
    \langle a,b\rangle
    \;\ge\;
    -\frac{1}{2\gamma}\|a\|^2 - \frac{\gamma}{2}\|b\|^2
    = -\frac{1}{\sigma}\|a\|^2 - \frac{\sigma}{4}\|b\|^2.
    \]
    With $a=\nabla f(w)$, $b=y_k-w$ this gives
    \[
    \langle \nabla f(w),y_k-w\rangle
    \;\ge\;
    -\frac{1}{\sigma}\|\nabla f(w)\|^2 - \frac{\sigma}{4}\|y_k-w\|^2.
    \]

    Plugging in \eqref{ineq:app-AGD-Lyap-ascent-lb3} we get 
    \[
        V^{(w)}_{k}(y_{k}, x_{k})
        \geq \frac{\sigma}{4}\Vert y_{k} - w\Vert^{2}+ \frac{\sigma}{2}\Vert z_k(y_k,x_k) - w\Vert^{2} -\frac{1}{\sigma}\|\nabla f(w)\|^2
    \]
    so rearranging we obtain
    \begin{align}
    \frac{\sigma}{2}\Vert y_{k} - w\Vert^{2} \leq 2V^{(w)}_{k}(y_{k}, x_{k})  + \frac{2}{\sigma}\|\nabla f(w)\|^2 -\sigma\Vert z_k(y_k,x_k) - w\Vert^{2}
    \label{eq:25}
    \end{align}

    We begin by substituting \eqref{eq:25} into \eqref{ineq:app-AGD-Lyap-ascent-aux3}, yielding
    \begin{align}
    \nonumber V^{(w)}_{k+1}(y_k,x_k)-V^{(w)}_{k}(y_k,x_k)
    \le\;&
    \frac{\sqrt{Q_k}-\sqrt{Q_{k-1}}}{\sqrt{Q_k}}\frac{\sigma}{2}\|x_k-w+\sqrt{Q_k}(x_k-y_k)\|^2
    \\&\nonumber \quad +\Bigl(\frac{Q_k-Q_{k-1}}{Q_{k-1}}-2\frac{\sqrt{Q_k}-\sqrt{Q_{k-1}}}{\sqrt{Q_{k-1}}}\Bigr)\frac{\sigma}{2}\|z_k(y_k,x_k)-w\|^2 \nonumber\\
    &\quad
    +2\frac{\sqrt{Q_k}-\sqrt{Q_{k-1}}}{\sqrt{Q_{k-1}}}V^{(w)}_{k}(y_k,x_k)
    +\frac{2}{\sigma}\frac{\sqrt{Q_k}-\sqrt{Q_{k-1}}}{\sqrt{Q_{k-1}}}\|\nabla f(w)\|^2.
    \label{eq:after-eq25-sub}
    \end{align}

    Note that the new coefficient in front of $\frac{\sigma}{2}\|z_k(y_k,x_k)-w\|^2$ is non-negative:
    \begin{align*}
    \frac{Q_{k}-Q_{k-1}}{Q_{k-1}}-2\frac{\sqrt{Q_{k}}-\sqrt{Q_{k-1}}}{\sqrt{Q_{k-1}}} = \Bigl(\frac{\sqrt{Q_k}}{\sqrt{Q_{k-1}}}-1\Bigr)^2 \geq 0.
    \end{align*}
    
    Hence we may apply \eqref{ineq:app-AGD-Lyap-ascent-lb2} and \eqref{ineq:app-AGD-Lyap-ascent-lb} to upper bound the remaining normed terms:
    \[
    \begin{aligned}
    \frac{\sigma}{2}\|x_k-w+\sqrt{Q_k}(x_k-y_k)\|^2
    &\le V^{(w)}_{k+1}(y_k,x_k)+\frac{1}{\sigma}\|\nabla f(w)\|^2, \\
    \frac{\sigma}{2}\|z_k(y_k,x_k)-w\|^2
    &\le V^{(w)}_{k}(y_k,x_k)+\frac{1}{\sigma}\|\nabla f(w)\|^2.
    \end{aligned}
    \]
    
    Plugging these into \eqref{eq:after-eq25-sub} gives
    \begin{align}
    V^{(w)}_{k+1}(y_k,x_k)-V^{(w)}_{k}(y_k,x_k)
    \le\;&
    \frac{\sqrt{Q_k}-\sqrt{Q_{k-1}}}{\sqrt{Q_k}}
    \Bigl(V^{(w)}_{k+1}(y_k,x_k)+\frac{1}{\sigma}\|\nabla f(w)\|^2\Bigr) \nonumber\\
    &\!
    +\Bigl(\frac{Q_k-Q_{k-1}}{Q_{k-1}}-2\frac{\sqrt{Q_k}-\sqrt{Q_{k-1}}}{\sqrt{Q_{k-1}}}\Bigr)
    \Bigl(V^{(w)}_{k}(y_k,x_k)+\frac{1}{\sigma}\|\nabla f(w)\|^2\Bigr) \nonumber\\
    &\!
    +2\frac{\sqrt{Q_k}-\sqrt{Q_{k-1}}}{\sqrt{Q_{k-1}}}V^{(w)}_{k}(y_k,x_k)
    +\frac{2}{\sigma}\frac{\sqrt{Q_k}-\sqrt{Q_{k-1}}}{\sqrt{Q_{k-1}}}\|\nabla f(w)\|^2.
    \label{eq:before-cancel}
    \end{align}
    
    Now the cancellation is explicit on the $V^{(w)}_{k}(y_k,x_k)$ terms:
    \begin{align*}
    \Bigl(
        \frac{Q_k-Q_{k-1}}{Q_{k-1}}
        -2\frac{\sqrt{Q_k}-\sqrt{Q_{k-1}}}{\sqrt{Q_{k-1}}}
        +2\frac{\sqrt{Q_k}-\sqrt{Q_{k-1}}}{\sqrt{Q_{k-1}}}
    \Bigr)
    V^{(w)}_{k}(y_k,x_k)
    &=
    \frac{Q_k-Q_{k-1}}{Q_{k-1}}V^{(w)}_{k}(y_k,x_k).
    \end{align*}

    Similarly, the $\frac{1}{\sigma}\|\nabla f(w)\|^2$ coefficients combine as
    \begin{align*}
    &\frac{\sqrt{Q_k}-\sqrt{Q_{k-1}}}{\sqrt{Q_k}}
    +\Bigl(
        \frac{Q_k-Q_{k-1}}{Q_{k-1}}
        -2\frac{\sqrt{Q_k}-\sqrt{Q_{k-1}}}{\sqrt{Q_{k-1}}}
    \Bigr)
    +2\frac{\sqrt{Q_k}-\sqrt{Q_{k-1}}}{\sqrt{Q_{k-1}}} \\
    &\qquad=
    \frac{\sqrt{Q_k}-\sqrt{Q_{k-1}}}{\sqrt{Q_k}}
    +\frac{Q_k-Q_{k-1}}{Q_{k-1}}.
    \end{align*}

    Therefore \eqref{eq:before-cancel} simplifies to
    \begin{align*}
    V^{(w)}_{k+1}(y_k,x_k)-V^{(w)}_{k}(y_k,x_k)
    \le\;&
    \frac{\sqrt{Q_k}-\sqrt{Q_{k-1}}}{\sqrt{Q_k}}\,V^{(w)}_{k+1}(y_k,x_k)
    +\frac{Q_k-Q_{k-1}}{Q_{k-1}}\,V^{(w)}_{k}(y_k,x_k)\\
    &+\frac{1}{\sigma}\Bigl(\frac{\sqrt{Q_k}-\sqrt{Q_{k-1}}}{\sqrt{Q_k}}+\frac{Q_k-Q_{k-1}}{Q_{k-1}}\Bigr)\|\nabla f(w)\|^2.
    \end{align*}
    
    Moving all \(V^{(w)}_{k+1}(y_{k}, x_{k})\) terms to the left-hand side and all \(V^{(w)}_{k}(y_{k}, x_{k})\) to the right-hand side, we obtain
    \begin{align}
        \frac{\sqrt{Q_{k-1}}}{\sqrt{Q_{k}}}V^{(w)}_{k+1}(y_{k}, x_{k})
        \leq&\ \frac{Q_{k}}{Q_{k-1}}V^{(w)}_{k}(y_{k}, x_{k}) + C_k\|\nabla f(w)\|^2,
        \label{ineq:app-AGD-Lyap-ascent-aux4}
    \end{align}
    where \begin{equation*}
    C_k = \frac{\sqrt{Q_{k}}-\sqrt{Q_{k-1}}}{\sigma\sqrt{Q_{k}}} +\frac{Q_{k}-Q_{k-1}}{\sigma Q_{k-1}}=\frac{1}{\sigma}\Bigl(\frac{Q_{k}}{Q_{k-1}}-\frac{\sqrt{Q_{k-1}}}{\sqrt{Q_{k}}}\Bigr) \leq \frac{Q_{k}}{\sigma Q_{k-1}}
    \end{equation*}
    
    Multiplying both sides of \eqref{ineq:app-AGD-Lyap-ascent-aux4} by \(\sqrt{Q_{k}}/\sqrt{Q_{k-1}}\), and then using the fact that \(\sqrt{Q_{k}}\geq 0\) yields
    \begin{align*}
        V^{(w)}_{k+1}(y_{k}, x_{k})
        &\leq \frac{\sqrt{Q}_{k}}{\sqrt{Q_{k-1}}}\frac{Q_{k}}{Q_{k-1}} V^{(w)}_{k}(y_{k}, x_{k}) +\frac{\sqrt{Q}_{k}}{\sqrt{Q_{k-1}}}C_k\|\nabla f(w)\|^2\\
        &\leq \frac{Q_{k}^{3/2}}{Q_{k-1}^{3/2}}V^{(w)}_{k}(y_{k}, x_{k})+\frac{Q_{k}^{3/2}}{\sigma Q_{k-1}^{3/2}}\|\nabla f(w)\|^2
    \end{align*}
    Therefore, both when \(\langle x_{k} - w,x_{k} - y_{k} \rangle\geq 0\) and when \(\langle x_{k} - w,x_{k} - y_{k} \rangle< 0\), the inequality
    \begin{align*}
        V^{(w)}_{k+1}(y_{k}, x_{k})
        &\leq \frac{Q_{k}^{3/2}}{Q_{k-1}^{3/2}}V^{(w)}_{k}(y_{k}, x_{k})+\frac{Q_{k}^{3/2}}{\sigma Q_{k-1}^{3/2}}\|\nabla f(w)\|^2
    \end{align*}
    holds generically for all \(y_{k}, x_{k}\), proving the lemma.
\end{proof}

\begin{figure}[t]
\centering
\begin{tikzpicture}
  \begin{axis}[
    width=0.8\linewidth, height=6cm,
    axis lines=middle,
    xlabel={$x$}, ylabel={$f(x)$},
    grid=both,
    domain=-4:4,
    samples=400,
  ]
    \addplot[thick] {0.5*x^2 + ln(cosh(x))};
  \end{axis}
\end{tikzpicture}
\caption{$f(x)=\tfrac12 x^2+\log(\cosh (x))$.}
\label{fig:counter}
\end{figure}
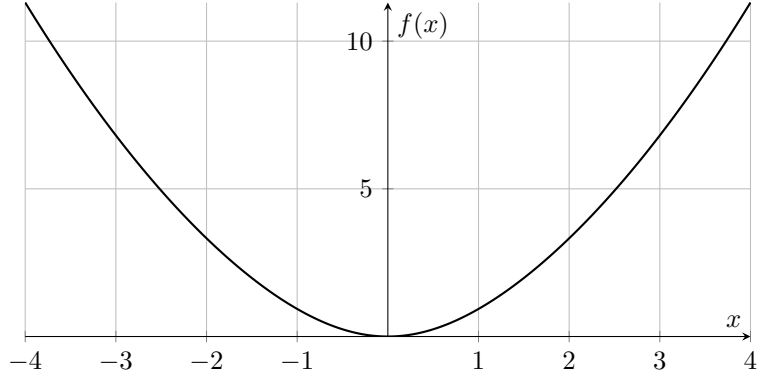

\subsubsection{Necessity of the Slack Term}\label{app:counter_example}
The $\|\nabla f(w)\|^2$ term cannot be removed in general. The following example (Figure \ref{fig:counter}) shows this slack is tight.
\[
f(x)=\frac{1}{2}x^2 + \log(\cosh(x)),\qquad \sigma=1,\qquad L_1=2,
\]
initialized with $x_0=y_0=2$, $L_1^{(0)}=1$, $\alpha_0=1$, backtracking parameters $c=0.5$, $\rho = 0.8$, and comparison point $w=-1$.

At $k=0$, we have $\nabla f(2) = 2+\tanh(2)\approx 2.964$, giving initial Lyapunov value
\[
V_1^{(-1)}(y_0,x_0) = (3.325 - 0.934) + \frac{1}{2}(3)^2 = 6.891.
\]
The backtracking check yields $v(1)\approx 0.559 < 1$, so the Lipschitz estimate is updated to
$L_1^{(1)}\approx 1.802$, with growth factor $R=(1.802)^{3/2}\approx 2.419$.
The next iterates are
\[
y_1 \approx 0.355, \qquad x_1 \approx 0.115,
\]
and since $f(y_1)-f(w)\approx -0.809$, the Lyapunov values at $(y_1,x_1)$ are
\[
V_1^{(-1)}(y_1,x_1) = -0.809 + \frac{1}{2}(0.875)^2 \approx -0.426,\quad
V_2^{(-1)}(y_1,x_1) = -0.809 + \frac{1}{2}(0.793)^2 \approx -0.495.
\]
The recursion without the gradient term would require
\[
V_2^{(-1)}(y_1,x_1)\;\le\; R\cdot V_1^{(-1)}(y_1,x_1),
\qquad\text{i.e.,}\qquad
-0.495 \le 2.419\times(-0.426)\approx -1.033,
\]
which is false.

\subsubsection{Lyapunov Recursion for Descent Reference Points}\label{sec:lyap-descent}
Under the additional descent condition that \(w\) lies below \(y_k\) by at least one gradient step, we strengthen Lemma~\ref{lower_bound} and obtain a tighter recursion without the \(\|\nabla f(w)\|^2\) slack.
\begin{lemma}
\label{lemma: optimality_bound}
Let $f:\mathbb{R}^d\to\mathbb{R}$ be differentiable. Let $y_k, w \in \mathbb{R}^d$ be points such that the following descent condition holds for some $L^{(k)}_1 > 0$:
\begin{equation}\label{eq:descent_gap}
f(w) \le f(y_k) - \frac{1}{2L^{(k)}_1} \|\nabla f(y_k)\|^2.
\end{equation}
Assume that for all $k = 0,\dots,t-1$ we have
\begin{equation}
f(w) \ge f(y_k) + \nabla f(y_k)^\top (w - y_k) + \frac{\sigma}{2}\lVert w - y_k \rVert^2.
\label{pathwise sc yk}
\end{equation}

Then the squared distance between $y_k$ and $w$ is bounded by the function gap as follows:
\begin{equation*}
\frac{\sigma}{2} \|y_k - w\|^2 \le 4 Q_k (f(y_k) - f(w)).
\end{equation*}
\end{lemma}

\begin{proof}
From the descent assumption \eqref{eq:descent_gap}, we can rearrange terms to obtain:
\begin{equation*}
\|\nabla f(y_k)\|^2 \le 2L^{(k)}_1 (f(y_k) - f(w)).
\end{equation*}

By pathwise strong convexity \eqref{pathwise sc yk}, we have:
\begin{align*}
f(y_k) - f(w) &\le \langle \nabla f(y_k), y_k - w \rangle - \frac{\sigma}{2} \|y_k - w\|^2.
\end{align*}
Applying the Cauchy-Schwarz inequality ($\langle a, b \rangle \le \|a\| \|b\|$) to the inner product term:
\begin{equation}
\label{eq:bound_after_cauchy}
f(y_k) - f(w) \le \|\nabla f(y_k)\| \|y_k - w\| - \frac{\sigma}{2} \|y_k - w\|^2
\end{equation}

Since $f(y_k) - f(w) \ge 0$ by assumption, the right-hand side of (\ref{eq:bound_after_cauchy}) must be non-negative. Thus:
\begin{align*}
\|\nabla f(y_k)\| \|y_k - w\| &\ge \frac{\sigma}{2} \|y_k - w\|^2 \\
\|y_k - w\| &\le \frac{2}{\sigma} \|\nabla f(y_k)\|.
\end{align*}
Squaring both sides yields:
\begin{equation*}
\|y_k - w\|^2 \le \frac{4}{\sigma^2} \|\nabla f(y_k)\|^2.
\end{equation*}

Substituting the bound for $\|\nabla f(y_k)\|^2$ yields
\begin{equation*}
\|y_k - w\|^2 \le \frac{4}{\sigma^2} \left[ 2L^{(k)}_1 (f(y_k) - f(w)) \right] = \frac{8L^{(k)}_1}{\sigma^2} (f(y_k) - f(w)).
\end{equation*}
Multiplying both sides by $\frac{\sigma}{2}$ and using the definition $Q_k = L^{(k)}_1/\sigma$:
\begin{equation*}
\frac{\sigma}{2} \|y_k - w\|^2 \le 4 \frac{L^{(k)}_1}{\sigma} (f(y_k) - f(w)) = 4 Q_k (f(y_k) - f(w)).
\end{equation*}
\end{proof}

Applying the lemma above, we derive a recursion analogous to Lemma~\ref{lem:app-AGD-Lyap-2} that eliminates the $\|\nabla f(w)\|^2$ term at the cost of an additional $(1+2\Delta_k)$ factor.

\begin{lemma}\label{lemma-6}
In addition to Lemma~\ref{lemma-5}, let w be such that $f(w)\leq f(y_k)-\frac{1}{2L^{(k)}_1}\|\nabla f(y_k)\|^2$. Assume that the pathwise strong convexity condition in \eqref{pathwise sc yk} holds. Then
\begin{equation*}
    V^{(w)}_{k+1}(y_{k}, x_{k})
    \leq (1+2\Delta_k)\frac{Q^{3/2}_k}{Q^{3/2}_{k-1}}V^{(w)}_{k}(y_{k}, x_{k}),
\end{equation*}
where $\Delta_k:= Q_k-Q_{k-1}\geq0$.

\end{lemma}
\begin{proof}
    We prove the bound for each $k\geq 0$ directly. The case $k=0$ follows from \(Q_{-1}:=Q_{0}\),
    \begin{align*}
        z_{1}(y_0,x_0)
        = x_{0} + \sqrt{Q_{0}}(x_{0} - y_{0})
        = x_{0}
        = x_{0} + \sqrt{Q_{-1}}(x_{0} - y_{0})
        = z_{0}(y_0,x_0).
    \end{align*}
    Moreover, we have that
    \begin{align*}
        V^{(w)}_{1}(y_{0}, x_{0})
        &= f(y_{0}) - f(w) + \frac{\sigma}{2}\Vert z_{1}(y_0,x_0) - w \Vert^{2}
        \\
        &= (1+2\Delta_0)\frac{Q_{0}^{3/2}}{Q_{-1}^{3/2}}\Bigl(f(y_{0}) - f(w) + \frac{\sigma}{2}\Vert z_{0}(y_0,x_0) - w \Vert^{2}\Bigr)
        \\
        &= (1+2\Delta_0)\frac{Q_{0}^{3/2}}{Q_{-1}^{3/2}}V^{(w)}_{0}(y_{0}, x_{0}),
    \end{align*}
    where $\Delta_0=0$, establishing the base case. To prove the inductive step, we again divide the analysis in two cases. For each case, we bound
    \begin{align}
        D_k :=
        &\,\Vert x_{k} - w + \sqrt{Q_{k}}(x_{k} - y_{k}) \Vert^{2}-\Vert z_k(y_k,x_k) - w \Vert^{2}
        \nonumber\\
        =&\ 2(\sqrt{Q_{k}}-\sqrt{Q_{k-1}})\langle x_{k} - w,x_{k} - y_{k} \rangle
        + (Q_{k}-Q_{k-1})\Vert x_{k} - y_{k} \Vert^{2}.
        \label{id:app-AGD-Lyap-ascent-2norm-gap'}
    \end{align}
    In turn, bounds on \eqref{id:app-AGD-Lyap-ascent-2norm-gap'} translate into bounds on \(V^{(w)}_{k+1}(y_{k}, x_{k})-V^{(w)}_{k}(y_{k}, x_{k})\), since
    \begin{align}
        V^{(w)}_{k+1}(y_{k}, x_{k})-V^{(w)}_{k}(y_{k}, x_{k})
        =&\ \frac{\sigma}{2}D_k.
        \label{ineq:app-AGD-Lyap-ascent-V-gap'}
    \end{align}
    Then, to prove the inductive step, we express bounds on \eqref{ineq:app-AGD-Lyap-ascent-V-gap'} in terms of \(V^{(w)}_{k+1}\) and \(V^{(w)}_{k}\).
    
    \paragraph{Case 1: \(\langle x_k - y_k, x_k - w\rangle \ge 0\).} 
    This time plugging \eqref{ineq:app-AGD-Lyap-ascent-case1-aux} back into \eqref{ineq:app-AGD-Lyap-ascent-V-gap} from Lemma~\ref{lemma-5} yields
    \begin{align}
        V^{(w)}_{k+1}(y_{k}, x_{k})-V^{(w)}_{k}(y_{k}, x_{k})
        &\leq \frac{Q_{k}-Q_{k-1}}{Q_{k}}\frac{\sigma}{2}\Vert x_{k} - w + \sqrt{Q_{k}}(x_{k} - y_{k}) \Vert^{2}
        \nonumber\\
        &\leq \frac{Q_{k}-Q_{k-1}}{Q_{k}}V^{(w)}_{k+1}(y_{k}, x_{k}),
        \label{ineq:app-AGD-Lyap-ascent-case1-aux1'}
    \end{align}
    where the last inequality follows from the definition of \(V^{(w)}_{k}\), as \(f(y_{k}) - f(w)\geq 0\) implies
    \begin{align}
        V^{(w)}_{k+1}(y_{k}, x_{k})
        \geq \frac{\sigma}{2}\Vert x_{k} - w + \sqrt{Q_{k}}(x_{k} - y_{k}) \Vert^{2}.
        \label{ineq:app-AGD-Lyap-ascent-lb'}
    \end{align}
    Thus, rearranging terms in \eqref{ineq:app-AGD-Lyap-ascent-case1-aux1'} and multiplying both sides by \(Q_{k}/Q_{k-1}\), we obtain
    \begin{align*}
        V^{(w)}_{k+1}(y_{k}, x_{k})
        \leq \frac{Q_{k}}{Q_{k-1}}V^{(w)}_{k}(y_{k}, x_{k})
        \leq (1+2\Delta_k)\frac{Q_{k}^{3/2}}{Q_{k-1}^{3/2}}V^{(w)}_{k}(y_{k}, x_{k}),
    \end{align*}
    where the second inequality holds because \(Q_{k}/Q_{k-1}\geq 1\) and $\Delta_k\geq0$.
    
    \paragraph{Case 2: \(\langle x_k - y_k, x_k - w\rangle < 0\).}
    From \eqref{ineq:app-AGD-Lyap-ascent-aux3} in Lemma~\ref{lemma-5} we have
    \begin{equation}
    V_{k+1}^{(w)}(y_k, x_k) - V_k^{(w)}(y_k, x_k)
    \le \frac{\sigma}{2}\Big(
       A\| x_k - w + \sqrt{Q_k}(x_k - y_k)\|^2
     + B\|z_k(y_k, x_k) - w\|^2
     + C\|y_k - w\|^2 \Big),
    \label{ineq:app-AGD-Lyap-ascent-aux3'}
    \end{equation}
    where \(A := (\sqrt{Q_k} - \sqrt{Q_{k-1}})/\sqrt{Q_k}\),
    \(B := (Q_k - Q_{k-1})/Q_{k-1}\),
    \(C := (\sqrt{Q_k} - \sqrt{Q_{k-1}})/\sqrt{Q_{k-1}}\).

    Now, as in \eqref{ineq:app-AGD-Lyap-ascent-lb'}, the fact that \(f(y_{k}) - f(w)\geq 0\) implies
    \begin{align}
        V^{(w)}_{k}(y_{k}, x_{k})
        = f(y_{k}) - f(w) + \frac{\sigma}{2}\Vert z_k(y_k,x_k) - w\Vert^{2}
        \geq \frac{\sigma}{2}\Vert z_k(y_k,x_k) - w\Vert^{2}.
        \label{ineq:app-AGD-Lyap-ascent-lb2'}
    \end{align}
    Applying Lemma~\ref{lemma: optimality_bound}, we obtain
    \begin{align}
        4 Q_k V^{(w)}_{k}(y_{k}, x_{k})
        \geq 4 Q_k (f(y_k) - f(w))
        \geq \frac{\sigma}{2}\Vert y_{k} - w\Vert^{2}.
        \label{ineq:app-AGD-Lyap-ascent-lb3'}
    \end{align}
    Plugging in \eqref{ineq:app-AGD-Lyap-ascent-lb'}, \eqref{ineq:app-AGD-Lyap-ascent-lb2'} and \eqref{ineq:app-AGD-Lyap-ascent-lb3'} back into \eqref{ineq:app-AGD-Lyap-ascent-aux3'}, and then moving all \(V^{(w)}_{k+1}(y_{k}, x_{k})\) terms to the left-hand side and all \(V^{(w)}_{k}(y_{k}, x_{k})\) to the right-hand side, we obtain
    \begin{align}
        \frac{\sqrt{Q_{k-1}}}{\sqrt{Q_{k}}}V^{(w)}_{k+1}(y_{k}, x_{k})
        \leq&\ \Bigl( \frac{Q_{k}}{Q_{k-1}} + 4Q_k\frac{\sqrt{Q_{k}}-\sqrt{Q_{k-1}}}{\sqrt{Q_{k-1}}} \Bigr)V^{(w)}_{k}(y_{k}, x_{k})
        \label{ineq:app-AGD-Lyap-ascent-aux4'}
    \end{align}
    Multiplying both sides of \eqref{ineq:app-AGD-Lyap-ascent-aux4'} by \(\sqrt{Q_{k}}/\sqrt{Q_{k-1}}\), and using the fact that \(\sqrt{Q_{k}}\geq \sqrt{Q_{k-1}}\) yields
    \begin{align*}
        V^{(w)}_{k+1}(y_{k}, x_{k})
        &\leq \frac{\sqrt{Q}_{k}}{\sqrt{Q_{k-1}}}\Bigl(\frac{Q_{k}}{Q_{k-1}} + 4Q_k\frac{\sqrt{Q_{k}}-\sqrt{Q_{k-1}}}{\sqrt{Q_{k-1}}} \Bigr)V^{(w)}_{k}(y_{k}, x_{k})\\
        &= \Bigl(1+\frac{4\Delta_k}{\phi_k+1}\Bigr)\frac{Q^{3/2}_k}{Q^{3/2}_{k-1}}V^{(w)}_{k}(y_{k}, x_{k})\\
        &\leq (1+2\Delta_k)\frac{Q^{3/2}_k}{Q^{3/2}_{k-1}}V^{(w)}_{k}(y_{k}, x_{k}),
    \end{align*}
    where $\Delta_k :=Q_k-Q_{k-1}\geq 0$ and $\phi_k = \sqrt{Q_k/Q_{k-1}}\geq1$.\\
    
    The equality holds because 
    \[
    \frac{\sqrt{Q_k}}{\sqrt{Q_{k-1}}}\Bigl(\frac{Q_k}{Q_{k-1}}+4Q_k\frac{\sqrt{Q_k}-\sqrt{Q_{k-1}}}{\sqrt{Q_{k-1}}}\Bigr)
    = \phi_k\Bigl(\phi_k^2 + 4Q_k(\phi_k-1)\Bigr).
    \]
    Using \(Q_k=\phi_k^2 Q_{k-1}\), this becomes $\phi_k^3\Bigl(1+4Q_{k-1}(\phi_k-1)\Bigr)$. Moreover,
    \[
    \Delta_k = Q_k-Q_{k-1} = Q_{k-1}(\phi_k^2-1)=Q_{k-1}(\phi_k-1)(\phi_k+1),
    \]
    so \(Q_{k-1}(\phi_k-1)=\frac{\Delta_k}{\phi_k+1}\). Substituting yields
    \[
    \frac{\sqrt{Q_k}}{\sqrt{Q_{k-1}}}\Bigl(\frac{Q_k}{Q_{k-1}}+4Q_k\frac{\sqrt{Q_k}-\sqrt{Q_{k-1}}}{\sqrt{Q_{k-1}}}\Bigr)
    = \Bigl(1+\frac{4\Delta_k}{\phi_k+1}\Bigr)\frac{Q_k^{3/2}}{Q_{k-1}^{3/2}}.
    \]
    Multiplying both sides by \(V_k^{(w)}(y_k,x_k)\ge 0\) gives the claimed equality. The final inequality holds because $\phi_k+1\geq2$. Therefore, in both cases, the inequality
    \begin{align*}
        V^{(w)}_{k+1}(y_{k}, x_{k})
        \leq (1+2\Delta_k)\frac{Q^{3/2}_k}{Q^{3/2}_{k-1}}V^{(w)}_{k}(y_{k}, x_k),
    \end{align*}
    holds generically for all \(y_{k}, x_{k}\), proving the lemma.
\end{proof}

\subsubsection{Convergence Bound for Arbitrary Reference Points}\label{sec:conv-arbitrary}
With the recursions above on the auxiliary point as well as Lemma~\ref{lem:app-AGD-Lyap-1}, we are ready to derive the new progress bounds. 
\begin{proposition}
\label{prop:agd-carmon-style}
Let $f : \mathbb{R}^d \to \mathbb{R}$ be $L_1$-smooth. Fix $w \in \mathbb{R}^d$ and assume that the pathwise strong convexity condition in \eqref{eq:pathwise-sc} holds. Given initial points $x_0 = y_0$, if the estimates $L_1^{(t)}$ of the Lipschitz constant in accelerated gradient descent (Algorithm~\ref{alg:AGD}) are generated monotonically by adaptive backtracking (Algorithm~\ref{alg:ABLS}) with $c \in [1/2,1)$ and $L_1^{(0)} > \sigma$, then for $t \ge 0$

\begin{equation*}
  f(y_t) - f(w)
  \;\le\;
  e^{\frac{-t}{\sqrt{Q_{t-1}}}}Q_{t-1}^{3/2}\,
\psi(w)
\;+\;\frac{\|\nabla f(w)\|^2}{\sigma}\;
Q_{t-1}^{3/2}
\Bigl(\sqrt{Q_{t-1}}-1\Bigr),
\end{equation*}
where $\psi(w)=f(y_0)-f(w)+\frac{\sigma}{2}\|w-y_0\|^2$.
\end{proposition}

\begin{proof}
Combining Lemmas~\ref{lem:app-AGD-Lyap-1} and~\ref{lemma-5}, we have for every $0 \leq k \leq t-1$

\begin{align*}
    V^{(w)}_{k+1}(y_{k+1}, x_{k+1})
    &\leq
    \frac{1}{1+\delta_k}V^{(w)}_{k+1}(y_{k}, x_{k})\\
    &\leq \frac{1}{1+\delta_k}\frac{Q_{k}^{3/2}}{Q_{k-1}^{3/2}}V^{(w)}_{k}(y_{k}, x_{k})+\frac{1}{1+\delta_k}\frac{Q_{k}^{3/2}}{\sigma Q_{k-1}^{3/2}}\|\nabla f(w)\|^2\\
    &= (1-\frac{1}{\sqrt{Q_{k-1}}})\frac{Q_{k}^{3/2}}{Q_{k-1}^{3/2}}V^{(w)}_{k}(y_{k}, x_{k})+(1-\frac{1}{\sqrt{Q_{k-1}}})\frac{Q_{k}^{3/2}}{\sigma Q_{k-1}^{3/2}}\|\nabla f(w)\|^2\\
    &\leq (1-\frac{1}{\sqrt{Q_{t-1}}})\frac{Q_{k}^{3/2}}{Q_{k-1}^{3/2}}V^{(w)}_{k}(y_{k}, x_{k})+(1-\frac{1}{\sqrt{Q_{t-1}}})\frac{Q_{k}^{3/2}}{\sigma Q_{k-1}^{3/2}}\|\nabla f(w)\|^2,
\end{align*}
where the last inequality uses $Q_k \le Q_{t-1}$ and the monotonicity of $x\mapsto 1-1/\sqrt{x}$.

By induction on $t$, we obtain that for all $t \ge 0$,
\[
\begin{aligned}
V^{(w)}_{t}(y_t,x_t)
\le\;&
\Bigl(1-\frac1{\sqrt{Q_{t-1}}}\Bigr)^{t}\frac{Q_{t-1}^{3/2}}{Q_{0}^{3/2}}\,
V^{(w)}_{0}(y_0,x_0)
+\frac{\|\nabla f(w)\|^2}{\sigma}\;
\sum_{i=0}^{t-1}
\Bigl(1-\frac1{\sqrt{Q_{t-1}}}\Bigr)^{t-i}\frac{Q_{t-1}^{3/2}}{Q_{i-1}^{3/2}}\\
\le\;&
\Bigl(1-\frac1{\sqrt{Q_{t-1}}}\Bigr)^{t}\frac{Q_{t-1}^{3/2}}{Q_{0}^{3/2}}\,
V^{(w)}_{0}(y_0,x_0)
+\frac{\|\nabla f(w)\|^2}{\sigma}\;
\frac{Q_{t-1}^{3/2}}{Q_{0}^{3/2}}\sum_{i=0}^{t-1}
\Bigl(1-\frac1{\sqrt{Q_{t-1}}}\Bigr)^{t-i}\\
\le\;&
\Bigl(1-\frac1{\sqrt{Q_{t-1}}}\Bigr)^{t}Q_{t-1}^{3/2}\,
\psi(w)
+\frac{\|\nabla f(w)\|^2}{\sigma}\;
Q_{t-1}^{3/2}
\Bigl(\sqrt{Q_{t-1}}-1\Bigr)\Bigl(1-\Bigl(1-\frac1{\sqrt{Q_{t-1}}}\Bigr)^{t}\Bigr).
\end{aligned}
\]

By definition of $V^{(w)}_t$ and non-negativity of the quadratic term,
\[
  f(y_t) - f(w)
  \;\le\;
  V^{(w)}_{t}(y_t,x_t)
  \qquad \forall t.
\]
Combining this with the above yields
\[
\begin{aligned}
  f(y_t) - f(w)
  \le\;&
  \Bigl(1-\frac1{\sqrt{Q_{t-1}}}\Bigr)^{t}Q_{t-1}^{3/2}\,
\psi(w)
+\frac{\|\nabla f(w)\|^2}{\sigma}\;
Q_{t-1}^{3/2}
\Bigl(\sqrt{Q_{t-1}}-1\Bigr)\Bigl(1-\Bigl(1-\frac1{\sqrt{Q_{t-1}}}\Bigr)^{t}\Bigr)\\
\le\;& e^{\frac{-t}{\sqrt{Q_{t-1}}}}Q_{t-1}^{3/2}\,
\psi(w)
+\frac{\|\nabla f(w)\|^2}{\sigma}\;
Q_{t-1}^{3/2}
\Bigl(\sqrt{Q_{t-1}}-1\Bigr)\quad \text{since $Q_{t-1} \geq 1$}.
\end{aligned}
\]
\end{proof}

\subsubsection{Convergence Bound for Descent Reference Points}\label{sec:conv-descent}
Almost identically to the above we get the progress bound for $w$ as a descent point. This also explains why we must keep track of $m$, the number of $Q_t$ increases, within our algorithm.
\begin{proposition}
\label{prop:better}
In addition to Proposition~\ref{prop:agd-carmon-style}, let w be such that $f(w)\leq f(y_k)-\frac{1}{2L^{(k)}_1}\|\nabla f(y_k)\|^2$. Assume that the pathwise strong convexity condition in \eqref{eq:pathwise-sc} holds. Then for $t \ge 0$
\begin{equation*}
  f(y_t) - f(w)
  \;\le\;
  e^{\frac{-t}{\sqrt{Q_{t-1}}}}(3Q_{t-1})^m\frac{Q_{t-1}^{3/2}}{Q_{0}^{3/2}}\,
\psi(w),
\end{equation*}
where $\psi(w)=f(y_0)-f(w)+\frac{\sigma}{2}\|w-y_0\|^2$, $\Delta_k =Q_k-Q_{k-1}\geq0$ and m is the number of iterations with $\Delta_k > 0$. Let $\bar L_1 :=\max\left\{
    L_1^{(0)},\; \frac{L_1}{2(1-c)\rho}
  \right\}$, then $m \leq \lfloor \log_{1/\rho} (\bar{L}_1 / L^{(0)}_1) \rfloor + 1$.
\end{proposition}

\begin{proof}
First we note that $m \leq \lfloor \log_{\rho} (\bar{\alpha} / \alpha_0) \rfloor + 1$ via Algorithm~\ref{alg:ABLS} where  $\bar{\alpha}$ is the final accepted step size globally. We can rewrite this as $m \leq \lfloor \log_{1/\rho} (\bar{L}_1 / L^{(0)}_1) \rfloor + 1$. 

Combining Lemmas~\ref{lem:app-AGD-Lyap-1} and~\ref{lemma-6}, we have that for every $0 \leq k \leq t-1$
\begin{align*}
    V^{(w)}_{k+1}(y_{k+1}, x_{k+1})
    \leq
    \frac{1}{1+\delta_k}V^{(w)}_{k+1}(y_{k}, x_{k})
    &\leq \frac{1}{1+\delta_k}(1+2\Delta_k)\frac{Q^{3/2}_k}{Q^{3/2}_{k-1}}V^{(w)}_{k}(y_{k}, x_{k}),\\
    &\leq \frac{\sqrt{Q_{k-1}}-1}{\sqrt{Q_{k-1}}}(1+2\Delta_k)\frac{Q^{3/2}_k}{Q^{3/2}_{k-1}}V^{(w)}_{k}(y_{k}, x_{k})\\
    &\leq (1-\frac{1}{\sqrt{Q_{k-1}}})(1+2\Delta_k)\frac{Q^{3/2}_k}{Q^{3/2}_{k-1}}V^{(w)}_{k}(y_{k}, x_{k}).
\end{align*}

We rename some of the $Q_k$ with $Q_t$, and by induction on $t$, obtain that for all $t \ge 0$,
\[
\begin{aligned}
V^{(w)}_{t}(y_t,x_t)
\;\le\;&
\Bigl(1-\frac1{\sqrt{Q_{t-1}}}\Bigr)^{t} \Bigl [ \prod_{k=0}^{t-1} (1 + 2\Delta_k) \Bigr] \frac{Q_{t-1}^{3/2}}{Q_{0}^{3/2}}\,
V^{(w)}_{0}(y_0,x_0)\\
\;\le\;&
\Bigl(1-\frac1{\sqrt{Q_{t-1}}}\Bigr)^{t}\Bigl [\prod_{k:\Delta_k > 0} (1 + 2\Delta_k)\Bigr] \frac{Q_{t-1}^{3/2}}{Q_{0}^{3/2}}\,
\psi(w)\\
\;\le\;&
\Bigl(1-\frac1{\sqrt{Q_{t-1}}}\Bigr)^{t}(3Q_{t-1})^m\frac{Q_{t-1}^{3/2}}{Q_{0}^{3/2}}\,
\psi(w),
\end{aligned}
\]
where we have used that $Q_k \le Q_{t-1}$ for $k\leq t-1$ and that $1+2\Delta_k \leq1+2Q_{t-1}\leq3Q_{t-1}$.

By definition of $V^{(w)}_t$ and non-negativity of the quadratic term,
\[
  f(y_t) - f(w)
  \;\le\;
  V^{(w)}_{t}(y_t,x_t)
  \qquad\text{for all }t.
\]
Combining this with the above yields
\[
\begin{aligned}
  f(y_t) - f(w)
  \;\le\;&
  \Bigl(1-\frac1{\sqrt{Q_{t-1}}}\Bigr)^{t}(3Q_{t-1})^m\frac{Q_{t-1}^{3/2}}{Q_{0}^{3/2}}\,
\psi(w)\\
\;\le\;& e^{\frac{-t}{\sqrt{Q_{t-1}}}}(3Q_{t-1})^m\frac{Q_{t-1}^{3/2}}{Q_{0}^{3/2}}\,
\psi(w)
\end{aligned}
\]
\end{proof}

\subsection{Estimating the Third-Order Lipschitz Constant}\label{sec:L_3}
In this section, we address the remaining key limitation of the \textsc{AGD-Until-Guilty} framework: its dependence on prior knowledge of the third-order Lipschitz constant $L_3$. The algorithm \textsc{Guarded-Non-Convex-AGD} by Carmon et al.\ requires prior knowledge of $L_3$ in two places: to set the regularization parameter $\alpha = 2L_3^{1/3}\epsilon^{2/3}$ and the negative-curvature step size $\eta = \sqrt{2\alpha/L_3}$. We justify the changes in \textsc{PF-AGD} which allow us to replace the use of $L_3$ with a running estimate $M_k$.

The section is organized as follows. Subsection~\ref{sec:proxies} defines the proxy conditions and justifies that these conditions correctly identify when $M_k$ is insufficient. Subsections~\ref{sec:modified-lemma7} and~\ref{sec:outer-bound} then verify that the progress bound of Carmon et al. {\cite[Lemma 7]{pmlr-v70-carmon17a}} and the outer iteration count $K$ both carry over unchanged up to constants.

Since we are returning two extra pairs $(y_j,w_t), (w_t,y_j)$ from \citet{pmlr-v70-carmon17a}, we give a straightforward justification to show that the negative curvature exploitation in Carmon et al. {\cite[Lemmas 3 and 6]{pmlr-v70-carmon17a}} still holds.

For Carmon et al. {\cite[Lemma 3]{pmlr-v70-carmon17a}}, we must show that $\|y_j - w_t\| \le 4\tau$. Without loss of generality let $u=w_t$ and $v=y_j$ as the other case is symmetric. In Carmon et al. {\cite[Lemma 3]{pmlr-v70-carmon17a}}, we always have that $\|u - y_0\| \le \tau$. But since $y_j$ is one of the iterates, we also have $\|y_j - y_0\| \le \tau$. By the triangle inequality, we have $\|y_j-w_t\|\leq 2\tau\leq 4\tau$.\\

For Carmon et al. {\cite[Lemma 6]{pmlr-v70-carmon17a}}, we must show that $f(v)\leq f(y_0)+14\alpha\tau^2$ for $v \in \{y_j,w_t\}$. We first consider $v=y_j$. By Corollary~\ref{cor:modified_cor1_main}, $\hat{f}(y_j) \le \hat{f}(y_0) = f(y_0)$ and since the penalty term is non-negative, $f(y_j) \le \hat{f}(y_j)$. Thus, $f(y_j) \le f(y_0)$ holds. For the case $v=w_t$: if $w_t=y_0$, then we are done. If $w_t=w^{\min}_t$, then the inequality follows from the $\zeta_t$ backtracking.

\subsubsection{Proxy Conditions for $M_k$}\label{sec:proxies}
Fix an outer iteration $k$ and a current estimate $M_k>0$. In the inner \texttt{while} loop we set
\begin{equation*}
\alpha(M_k) \coloneqq 2M_k^{1/3}\epsilon^{2/3},\qquad
\tau(M_k) \coloneqq \sqrt{\frac{\alpha(M_k)}{32M_k}},\qquad
\eta(M_k) \coloneqq \sqrt{\frac{2\alpha(M_k)}{M_k}}.
\end{equation*}

We run \textsc{Modified-AGD} on $\hat f$ to obtain $(x_0^{t},y_0^{t},u,v)$. When $u,v\neq\textsc{Null}$,
\[
b^{(1)} \gets \textcolor{algoblue}{\textsc{Find-Best-Iterate}_3}(f,y_0^{t},u,v),\qquad
b^{(2)} \gets \textcolor{algoblue}{\textsc{Exploit-NC-Pair}_3}(f,u,v,\eta(M_k)).
\]

The algorithm accepts the current $M_k$ in the non-convex branch only if either
\begin{equation}\label{eq:case1-test}
f(b^{(1)})\ \le\ f(y_0)\ -\ \alpha(M_k)\tau(M_k)^2
\end{equation}
or (in the $f(b^{(1)}) > f(y_0) - \alpha\tau(M_k)^2$ case) if both the following proxies for [\citet{pmlr-v70-carmon17a}, Lemmas 5 and 6] hold:
\begin{align}
\label{eq:cert-lemma6}
f(v)\ &\le\ f(y_0)\ +\ 14\,\alpha(M_k)\tau(M_k)^2,\\
\label{eq:cert-lemma5}
f(b^{(2)})\ &\le\ \max\Big\{ f(v)-\frac{\alpha(M_k)}{4}\eta(M_k)^2,\ \ f(u)-\frac{\alpha(M_k)}{12}\eta(M_k)^2\Big\}.
\end{align}

The key observation is that the conclusions of Carmon et al. {\cite[Lemmas~5 and~6]{pmlr-v70-carmon17a}} can be verified directly on the computed iterates, without knowing $L_3$ explicitly. Whenever \eqref{eq:cert-lemma6} or \eqref{eq:cert-lemma5} fail, we can conclude that $M_k < L_3$ and the algorithm updates $M_k\leftarrow \gamma M_k$ for some fixed $\gamma>1$; once $M_k \geq L_3$, the conditions are guaranteed to hold and the estimate is never increased again. This gives a backtracking scheme where the total number of $L_3$ increases is bounded by $\mathcal{O}(\log_{\min(\gamma, 1/\rho)}(L_3/M_0))$.

\subsubsection{Progress Bound with Adaptive $M_k$}\label{sec:modified-lemma7}
The following result gives us the equivalent of Carmon et al. {\cite[Lemma 7]{pmlr-v70-carmon17a}} with $M_k$ replacing $L_3$.

\begin{lemma}
Let $f : \mathbb{R}^d \to \mathbb{R}$ be $L_1$-smooth and have $L_3$-Lipschitz continuous third-order derivatives, let $\epsilon, \alpha(M_k) > 0$ and $p_0 \in \mathbb{R}^d$. If $p_0^K$ is the sequence of iterates produced by \textcolor{algoblue}{\textsc{PF-AGD}}$(f, p_0, L_1^{(0)}, M_0, \gamma, \epsilon)$, then for every $1 \leq k < K$,

\begin{equation*}
f(p_k) \leq f(p_{k-1}) - \min \left\{ \frac{\epsilon^2}{5\alpha(M_k)}, \frac{\alpha(M_k)^2}{32M_k} \right\}.
\end{equation*}
\end{lemma}
\begin{proof}
We shall review the proof of Carmon et al. {\cite[Lemma 7]{pmlr-v70-carmon17a}} one argument at a time. The first case is when $u,v=\textsc{Null}$. Although our regularized objective $\hat f$ potentially changes from iteration to iteration due to $\alpha(M_k)$ being increased, the original proof of this case from Carmon et al. {\cite[Lemma 2]{pmlr-v70-carmon17a}} only relied on the property that $\hat f(p_k)=\hat f(y_t)\leq \hat f(y_0) =f(p_{k-1})$, which is still valid despite the fact that the regularization coefficient has changed.

Also, the next case where $f(b^{(1)}) \le f(y_0)-\alpha(M_k)\tau(M_k)^2$, we are also done with $L_3$ replaced by $M_k$. In our algorithm we simply return and break the line search in this case.

Now given $f(b^{(1)}) > f(y_0)-\alpha(M_k)\tau(M_k)^2$, by Carmon et al. {\cite[Lemma 3]{pmlr-v70-carmon17a}} (unchanged, works for arbitrary $\tau(M)$) we have that
\[
\|u-v\|\leq4\tau(M_k)\leq\sqrt{\frac{\alpha(M_k)}{2M_k}}=\frac{\eta(M_k)}{2}
\]

Therefore, we can apply inequality \eqref{eq:cert-lemma5} (with $\eta(M_k)$ as defined above) to show that
\begin{equation}\label{eq: fb2}
f(b^{(2)})\ \le\ \max\Big\{ f(v)-\frac{\alpha(M_k)^2}{2M_k},\ \ f(u)-\frac{\alpha(M_k)^2}{6M_k}\Big\}.
\end{equation}

Carmon et al. {\cite[Corollary~1 (9)]{pmlr-v70-carmon17a}} holds as usual, so $f(u)\leq \hat f(u) \leq \hat f(y_0) = f(p_{k-1})$ (see Corollary~\ref{cor:modified_cor1_main} for more details). Moreover, since $f(b^{(1)}) \geq f(y_0)-\alpha(M_k)\tau(M_k)^2$ and $\tau(M_k) = \sqrt{\frac{\alpha(M_k)}{32M_k}}$, we may apply \eqref{eq:cert-lemma6} to obtain
\[
f(v) \leq f(y_0)+14\alpha(M_k) \tau(M_k)^2 \leq f(p_{k-1})+ \frac{7\alpha(M_k)^2}{16M_k}
\]

Combining this with \eqref{eq: fb2}, we find that
\[
f(p_k) \leq f(b^{(2)}) \leq f(p_{k-1}) - \min \left\{ \frac{\alpha(M_k)^2}{2M_k} - \frac{7\alpha(M_k)^2}{16M_k}, \frac{\alpha(M_k)^2}{6M_k} \right\} = f(p_{k-1}) - \frac{\alpha(M_k)^2}{16M_k},
\]

which concludes the case $v, u \neq \textsc{Null}$ under third-order smoothness.
\end{proof}

\subsubsection{Bounding the Outer Iterations}\label{sec:outer-bound}
We wish to obtain a bound on the total number of outer iterations in terms of our original problem parameter $L_3$ and $M_0$. We begin by showing that our estimate always remains bounded. 

\begin{lemma}[Boundedness of $M_k$]\label{Mbounded}
Let $\{M_k\}_{k \geq 0}$ be the sequence of estimates produced by 
\textcolor{algoblue}{\textsc{PF-AGD}}$(f, p_0, L_1^{(0)}, M_0, \gamma, \epsilon)$, with initial estimate $M_0 > 0$ 
and multiplicative update factor $\gamma > 1$. Then for all $k \geq 0$,
\[
    M_k \leq \bar{M} \coloneqq \max\{M_0,\, \gamma L_3\}.
\]
In particular, the estimate $M_k$ is increased at most 
$\mathcal{O}(\log_\gamma(L_3 / M_0))$ times in total.
\end{lemma}
\begin{proof}
If $M_0 \geq L_3$, we know the check in line~11 of \textsc{PF-AGD} will always be \texttt{False} so the estimate will never be incremented. Otherwise, $M_0 <L_3$. Suppose after some point $M_k\geq L_3$, by the same reasoning as above, we know that $M_k$ will not be incremented further. Hence, $M_k< \gamma L_3$, and at most $\mathcal{O}(\log_\gamma(L_3/M_0))$ increments occur. Combining both cases we have $M_k\leq \max\{M_0,\gamma L_3\}$.
\end{proof}

With the boundedness of $M_k$ in mind, we split our analysis into cases. If $M_0 \leq \gamma L_3$, we have
\begin{align}
\nonumber f(p_k) \leq f(p_{k-1}) - \min \left\{ \frac{\epsilon^2}{5\alpha(M_k)}, \frac{\alpha(M_k)^2}{32M_k} \right\} &= f(p_{k-1}) - \min \left\{ \frac{\epsilon^{4/3}}{10M_k^{1/3}}, \frac{\epsilon^{4/3}}{8M_k^{1/3}} \right\} \\
&\nonumber \leq f(p_{k-1}) - \min \left\{ \frac{\epsilon^2}{5\alpha(\gamma L_3)}, \frac{\alpha(\gamma L_3)^2}{32\gamma L_3} \right\}\\
&= f(p_{k-1}) - \min \left\{ \frac{\epsilon^2}{5\gamma^{1/3}\alpha(L_3)}, \frac{\alpha(L_3)^2}{32\gamma^{1/3} L_3} \right\} \label{new-lemma7}
\end{align}

Otherwise, $M_k=M_0$ and we have

\begin{equation}
    f(p_k) \leq f(p_{k-1}) - \min \left\{ \frac{\epsilon^2}{5\alpha(M_0)}, \frac{\alpha(M_0)^2}{32M_0} \right\} = f(p_{k-1}) - \frac{\epsilon^{4/3}}{10M_0^{1/3}}\label{new-lemma7(2)}
\end{equation}

With these new progress bounds we can derive the upper bound $K$ of the number of iterations of \textsc{PF-AGD} by telescoping \eqref{new-lemma7}
\begin{align*}
\Delta_f \ge f(p_0) - f(p_{K-1}) = \sum_{k=1}^{K-1} \left( f(p_{k-1}) - f(p_k) \right) &\ge (K-1) \gamma^{-1/3}\cdot \min \left\{ \frac{\epsilon^2}{5\alpha(L_3)}, \frac{\alpha(L_3)^2}{32L_3} \right\}\\ & \ge (K-1) \frac{\epsilon^{4/3}}{10\gamma^{1/3}L_3^{1/3}}.
\end{align*}

In the case of \eqref{new-lemma7(2)}
\begin{align*}
\Delta_f \ge f(p_0) - f(p_{K-1}) = \sum_{k=1}^{K-1} \left( f(p_{k-1}) - f(p_k) \right)  \ge (K-1) \frac{\epsilon^{4/3}}{10M_0^{1/3}}.
\end{align*}

We therefore conclude that 
\begin{equation}\label{boundonK}
K\leq 1+10\epsilon^{-4/3}\Delta_f\bar M^{1/3},
\end{equation}
where $\bar{M} := \max\{\gamma L_3,M_0\}$.

\subsection{Global Convergence Rate for PF-AGD}\label{bound:proof}
With the Lyapunov analysis of Section~\ref{sec:Lyapunov_analysis} and the adaptive $L_3$ estimation of Section~\ref{sec:L_3} in hand, we are almost ready to assemble the main complexity bound. The argument proceeds in three stages. We first justify in Section~\ref{sec:partial-momentum} that the partial-momentum restarts introduced in Algorithm~\ref{alg:restart-handler} do not break acceleration. Section~\ref{sec:inner_iters} then translates these progress guarantees into a bound on the inner iteration count $T$, verifies that a non-convexity certificate $(u,v) \neq \textsc{Null}$ implies violation of a pathwise strong convexity condition, and confirms the iterates satisfy the boundedness property of Carmon et al. {\cite[Corollary~1 (9)]{pmlr-v70-carmon17a}} throughout. Finally, Section~\ref{sec:main_bound} combines the outer iteration bound $K$ from Section~\ref{sec:L_3} with the inner bound $T$ to obtain the main result: \textsc{PF-AGD} finds an $\epsilon$-stationary point in $\widetilde{\mathcal{O}}(\epsilon^{-5/3})$ gradient evaluations.

\subsubsection{Restarts Preserve Acceleration}
\label{sec:partial-momentum}
 
We show that the momentum correction introduced in Algorithm~\ref{alg:restart-handler} preserves the accelerated rates. Recall the Lyapunov potential from \eqref{lyapunov_function}:
\begin{equation}\label{eq:lyapunov-def}
  V^{(w)}_t(y_k,x_k)
  \;:=\;
  f(y_k) - f(w)
  + \frac{\sigma}{2}\,\bigl\|z_t(y_k,x_k) - w\bigr\|^2,
  \quad
  z_t(y_k,x_k) \;:=\; x_k + \sqrt{Q_{t-1}}\,(x_k - y_k),
\end{equation}
where $w\in\mathbb{R}^d$ is a reference point and $Q_t := L_1^{(t)}/\sigma$. The subscript on~$V$ determines which condition number enters the auxiliary point: $V_{t+1}$ uses~$Q_t$ via $z_{t+1}(y_k,x_k) = x_k + \sqrt{Q_t}\,(x_k - y_k)$. Additionally, $\hat{V}_{t}$ denotes the Lyapunov potential with auxiliary point evaluated at $Q_{t-1}^{\mathrm{agd}}$, i.e., $\hat{z}_{t}(y,x)=x+\sqrt{Q_{t-1}^{\mathrm{agd}}}(x-y)$.

At each iteration~$t$ of Algorithm~\ref{alg:agd-until-guilty}, \textsc{Modified-AGD}, the routine \textsc{AGD-Step} (Algorithm~\ref{alg:agd-step}) computes the iterates $(x_t,y_t)$ from $(x_{t-1},y_{t-1})$ using the condition number $Q_t = L_1^{(t)}/\sigma$. On the other hand, \citet{cavalcanti2025adaptive} produce $(y_{t+1},x_{t+1})$ using $Q_t$ as step~$t$. Given this disparity in the iteration index, Lemma~\ref{lem:app-AGD-Lyap-1} is rewritten as
\begin{equation}\label{eq:lemma1}
  \bigl(1+\delta_t\bigr)\;V^{(w)}_{t+1}(y_t,\,x_t)
  \;\le\;
  V^{(w)}_{t+1}(y_{t-1},\,x_{t-1}),
  \qquad
  \delta_t = \frac{1}{\sqrt{Q_t}-1}\,,
\end{equation}
where both sides use $V_{t+1}$, which evaluates the auxiliary point with~$Q_t$: $z_{t+1} = x + \sqrt{Q_t}\,(x-y)$.

If \textsc{Certify-Progress} detects $f(y_t) > f(y_0)$, a restart is triggered and \textsc{Restart-Handler} overwrites~$(x_t,y_t)$. We write $(x_t^{\mathrm{agd}},\,y_t^{\mathrm{agd}},\,Q_t^{\mathrm{agd}})$ for the iterates and condition number produced by \textsc{AGD-Step} \emph{before} the restart, and $(x_t,\,y_t,\,Q_t)$ for the \emph{final} values stored at iteration~$t$ (after the restart, if one occurred).  When no restart occurs, $(x_t,y_t,Q_t) = (x_t^{\mathrm{agd}},y_t^{\mathrm{agd}},Q_t^{\mathrm{agd}})$. We remark here that although the restart overwrites the previous iterate, it is still enough to just bound the total number of iterations $T$. The total number of gradient evaluations in one epoch of \textsc{Modified-AGD} with~$T$ iterations and~$b$ restarts is at most
\[
  G \;\le\; 2T + 2b
  + 2\bigl\lceil\log_\gamma(\bar{L}_1/L_1^{(0)})\bigr\rceil.
\]
Since $b\le T$, this gives $G\le 4T + O(\log(\bar{L}_1/L_1^{(0)}))$.

The following lemma shows that the auxiliary-point correction strictly reduces the Lyapunov function, since the corrected iterates share the same auxiliary point $z$ and satisfy $f(y_t) < f(y_t^{\mathrm{agd}})$.
\begin{lemma}
\label{lem:V-decrease}
Let $f:\mathbb{R}^d\to\mathbb{R}$ be $L_1$-smooth and let $w\in\mathbb{R}^d$. 
At iteration~$t$:
\begin{enumerate}[label=\textnormal{(\alph*)}]
  \item The Lyapunov auxiliary point is computed at
    line~6 of \textcolor{algoblue}{\textsc{Modified-AGD}}:
    \begin{equation}\label{eq:zprev-def}
      z \;=\; x_t^{\mathrm{agd}} + \sqrt{Q_t^{\mathrm{agd}}}\,
             \bigl(x_t^{\mathrm{agd}} - y_t^{\mathrm{agd}}\bigr).
    \end{equation}
    In the notation above, this is $\hat z_{t+1}(y^{\mathrm{agd}}_t, x^{\mathrm{agd}}_t)$
  \item \textcolor{algoblue}{\textsc{Certify-Progress}} detects
    $f\!\bigl(y_t^{\mathrm{agd}}\bigr)>f(y_0)$ and triggers a restart.
  \item \textcolor{algoblue}{\textsc{Restart-Handler}} receives
    $z$ and $Q_t^{\mathrm{agd}}$ as inputs, and produces new iterates
    $(x_t,\,y_t)$ satisfying the partial-momentum invariant
    \begin{equation}\label{eq:pm-invariant}
      x_t + \sqrt{Q_t^{\mathrm{agd}}}\,(x_t-y_t) \;=\; z,
    \end{equation}
    together with $f(y_t)\le f(y_{t-1})$.
  \item Iteration $t-1$ did not trigger a restart:
    $f(y_{t-1})\le f(y_0)$.
\end{enumerate}
Then,
\begin{equation}\label{eq:V-decrease}
  \hat{V}^{(w)}_{t+1}(y_t,\,x_t)
  \;<\;
  \hat{V}^{(w)}_{t+1}\!\bigl(y_t^{\mathrm{agd}},\,x_t^{\mathrm{agd}}\bigr).
\end{equation}
\end{lemma}
 
\begin{proof}
We show that both evaluations of~$\hat V_{t+1}$ share the same auxiliary point, so that their difference reduces to a comparison of function values.
 
\emph{Auxiliary points coincide.}\;
Since $\hat V_{t+1}$ uses $Q_t^{\mathrm{agd}}$, the auxiliary point at the new iterates is
\[
  \hat z_{t+1}(y_t,x_t)
  \;=\; x_t + \sqrt{Q_t^{\mathrm{agd}}}\,(x_t - y_t)
  \;\stackrel{\eqref{eq:pm-invariant}}{=}\; z.
\]
Pre-restart, $\hat z_{t+1}\!\bigl(y_t^{\mathrm{agd}},x_t^{\mathrm{agd}}\bigr) = z$ by~\eqref{eq:zprev-def}. Therefore both auxiliary points equal~$z$, and
\begin{equation}\label{eq:V-diff-f-diff}
  \hat V^{(w)}_{t+1}(y_t,x_t)
  \;-\;
  \hat V^{(w)}_{t+1}\!\bigl(y_t^{\mathrm{agd}},x_t^{\mathrm{agd}}\bigr)
  \;=\;
  f(y_t) - f\!\bigl(y_t^{\mathrm{agd}}\bigr).
\end{equation}
 
\emph{Function values decrease.}\;
Lines~1--4 of \textsc{Restart-Handler} perform a steepest-descent step from~$y_{t-1}$ with backtracking, guaranteeing
\begin{equation}\label{eq:restart-descent}
  f(y_t) \;\le\; f(y_{t-1})
  - \frac{1}{2L_1^{(t)}}\,\bigl\|\nabla f(y_{t-1})\bigr\|^2
  \;\le\; f(y_{t-1})
  \;\le\; f(y_0),
\end{equation}
where the last inequality uses assumption~(d). Assumption~(b) gives $f\!\bigl(y_t^{\mathrm{agd}}\bigr) > f(y_0)$, so
\[
  f(y_t) \;\le\; f(y_0) \;<\; f\!\bigl(y_t^{\mathrm{agd}}\bigr).
\]
Substituting into~\eqref{eq:V-diff-f-diff} yields the strict inequality~\eqref{eq:V-decrease}.
\end{proof}

Now that we can compare potential functions before and after a restart we can rederive the recursion at a restarted iteration. We use Lemmas~\ref{lemma-5} and~\ref{lemma-6} to convert $\hat V_t$ back to $V_t$. 

\begin{lemma}
\label{lem:combined-restart}
Under the hypotheses of Lemma~\ref{lem:V-decrease}, let $Q_t\ge Q_t^{\mathrm{agd}}$ be the final condition number after \textsc{Restart-Handler}. Assume that the pathwise strong convexity conditions in \eqref{eq:pathwise-sc} hold. Write $(x_t,y_t)$ for the post-restart iterates produced by Algorithm~\ref{alg:restart-handler}.
 
\medskip\noindent
\textbf{(a) General~$w$.}
\begin{equation}\label{eq:combined-general}
  V^{(w)}_{t+1}(y_t,x_t)
  \;\le\;
  \frac{1}{1+\hat\delta_t}\;
  \frac{Q_t^{3/2}}{Q_{t-1}^{3/2}}\;
  V^{(w)}_t(y_{t-1},x_{t-1})
  \;+\;
  \frac{2Q_t^{3/2}}{\sigma\,Q_{t-1}^{3/2}}\;
  \|\nabla f(w)\|^2,
\end{equation}
where $\hat\delta_t = 1/(\sqrt{Q_t^{\mathrm{agd}}}-1)$.
 
\medskip\noindent
\textbf{(b) Descent reference point.}\;
If additionally $w$ satisfies $f(w)\le f(y_k) - \frac{1}{2L_1^{(k)}}\|\nabla f(y_k)\|^2$ for $k=0,\ldots,t$, then
\begin{equation}\label{eq:combined-descent}
  V^{(w)}_{t+1}(y_t,x_t)
  \;\le\;
  \frac{1}{1+\hat\delta_t}\;
  (1+2\hat\Delta_t)(1+2\Delta'_t)\;
  \frac{Q_t^{3/2}}{Q_{t-1}^{3/2}}\;
  V^{(w)}_t(y_{t-1},x_{t-1}),
\end{equation}
where $\hat\Delta_t := Q_t - Q_t^{\mathrm{agd}} \ge 0$ and $\Delta'_t := Q_t^{\mathrm{agd}} - Q_{t-1} \ge 0$.
\end{lemma}
 
\begin{proof}
We compose three bounds.
 
\medskip
\noindent\emph{Step~1 (Lemma~\ref{lem:app-AGD-Lyap-1}).}\;
The AGD step uses $Q_t^{\mathrm{agd}}$. Applying~\eqref{eq:lemma1} with $Q_t\leftarrow Q_t^{\mathrm{agd}}$:
\begin{equation}\label{eq:chain-step1}
  (1+\hat\delta_t)\;
  \hat{V}^{(w)}_{t+1}(y_t^{\mathrm{agd}},x_t^{\mathrm{agd}})
  \;\le\;
  \hat{V}^{(w)}_{t+1}(y_{t-1},x_{t-1}),
  \qquad
  \hat\delta_t = \frac{1}{\sqrt{Q_t^{\mathrm{agd}}}-1}\,.
\end{equation}
 
\medskip
\noindent\emph{Step~2 (Lemma~\ref{lem:V-decrease}).}\;
The restart preserves $\hat{z}_{t+1}$:
\begin{equation}\label{eq:chain-step2}
  \hat{V}^{(w)}_{t+1}(y_t,x_t)
  \;\le\;
  \hat{V}^{(w)}_{t+1}(y_t^{\mathrm{agd}},x_t^{\mathrm{agd}}).
\end{equation}
Combining~\eqref{eq:chain-step1} and~\eqref{eq:chain-step2}:
\begin{equation}\label{eq:chain-12}
  \hat{V}^{(w)}_{t+1}(y_t,x_t)
  \;\le\;
  \frac{1}{1+\hat\delta_t}\;
  \hat{V}^{(w)}_{t+1}(y_{t-1},x_{t-1}).
\end{equation}
 
\medskip
\noindent\emph{Step~3 (Lemma~\ref{lemma-5}: absorb $Q$-changes).}\;
We convert $\hat{V}_{t+1}$ (using $Q_t^{\mathrm{agd}}$) to the standard Lyapunov functions $V_{t+1}$ (using $Q_t$) and $V_t$ (using $Q_{t-1}$).
 
\emph{At the new iterates $(y_t,x_t)$.}\;
Since $Q_t\ge Q_t^{\mathrm{agd}}$, Lemma~\ref{lemma-5} gives
\begin{equation}\label{eq:chain-new1}
  V^{(w)}_{t+1}(y_t,x_t)
  \;\le\;
  \frac{Q_t^{3/2}}{(Q_t^{\mathrm{agd}})^{3/2}}\;
  \hat{V}^{(w)}_{t+1}(y_t,x_t)
  \;+\;
  \frac{Q_t^{3/2}}{\sigma\,(Q_t^{\mathrm{agd}})^{3/2}}\;
  \|\nabla f(w)\|^2.
\end{equation}
 
\emph{At the old iterates $(y_{t-1},x_{t-1})$.}\;
Since $Q_t^{\mathrm{agd}}\ge Q_{t-1}$, Lemma~\ref{lemma-5} gives
\begin{equation}\label{eq:chain-old1}
  \hat{V}^{(w)}_{t+1}(y_{t-1},x_{t-1})
  \;\le\;
  \frac{(Q_t^{\mathrm{agd}})^{3/2}}{Q_{t-1}^{3/2}}\;
  V^{(w)}_t(y_{t-1},x_{t-1})
  \;+\;
  \frac{(Q_t^{\mathrm{agd}})^{3/2}}{\sigma\,Q_{t-1}^{3/2}}\;
  \|\nabla f(w)\|^2.
\end{equation}
 
\medskip
\noindent\emph{Combining.}\;
Substitute~\eqref{eq:chain-12}
into~\eqref{eq:chain-new1}:
\[
  V^{(w)}_{t+1}(y_t,x_t)
  \;\le\;
  \frac{1}{1+\hat\delta_t}\;
  \frac{Q_t^{3/2}}{(Q_t^{\mathrm{agd}})^{3/2}}\;
  \hat{V}^{(w)}_{t+1}(y_{t-1},x_{t-1})
  \;+\;
  \frac{Q_t^{3/2}}{\sigma\,(Q_t^{\mathrm{agd}})^{3/2}}\;
  \|\nabla f(w)\|^2.
\]
Now substitute~\eqref{eq:chain-old1}. The $Q^{\mathrm{agd}}_t$-ratios cancel:
\[
  \frac{Q_t^{3/2}}{(Q_t^{\mathrm{agd}})^{3/2}}\cdot
  \frac{(Q_t^{\mathrm{agd}})^{3/2}}{Q_{t-1}^{3/2}}
  \;=\;\frac{Q_t^{3/2}}{Q_{t-1}^{3/2}}\,,
\]
yielding
\[
  V^{(w)}_{t+1}(y_t,x_t)
  \;\le\;
  \frac{1}{1+\hat\delta_t}\;
  \frac{Q_t^{3/2}}{Q_{t-1}^{3/2}}\;
  V^{(w)}_t(y_{t-1},x_{t-1})
  \;+\;
  C_t\,\|\nabla f(w)\|^2,
\]
where the gradient-slack coefficient is
\[
  C_t
  = \frac{1}{1+\hat\delta_t}\;
    \frac{Q_t^{3/2}}{\sigma\,Q_{t-1}^{3/2}}
  + \frac{Q_t^{3/2}}{\sigma\,(Q_t^{\mathrm{agd}})^{3/2}}
  \;\le\;
  \frac{Q_t^{3/2}}{\sigma\,Q_{t-1}^{3/2}}
  + \frac{Q_t^{3/2}}{\sigma\,Q_{t-1}^{3/2}}
  = \frac{2\,Q_t^{3/2}}{\sigma\,Q_{t-1}^{3/2}}\,,
\]
using $1/(1+\hat\delta_t)\le 1$ and $Q_t^{\mathrm{agd}}\ge Q_{t-1}$. This establishes~\eqref{eq:combined-general}.
 
For~(b), replace Lemma~\ref{lemma-5} by Lemma~\ref{lemma-6} at each sub-step (eliminating the gradient terms and introducing the factor $1+2\Delta_t$ with $\Delta_t = Q_t-Q_{t-1}$).
 
\emph{At the new iterates $(y_t,x_t)$.}\;
Since $Q_t\ge Q_t^{\mathrm{agd}}$, Lemma~\ref{lemma-6} gives
\begin{equation}\label{eq:chain-new}
  V^{(w)}_{t+1}(y_t,x_t)
  \;\le\;
  (1+2\hat\Delta_t)\frac{Q_t^{3/2}}{(Q_t^{\mathrm{agd}})^{3/2}}\;
  \hat{V}^{(w)}_{t+1}(y_t,x_t),
\end{equation}
where $\hat\Delta_t:= Q_t-Q^{\mathrm{agd}}_t\geq0$.
 
\emph{At the old iterates $(y_{t-1},x_{t-1})$.}\;
Since $Q_t^{\mathrm{agd}}\ge Q_{t-1}$, Lemma~\ref{lemma-6} gives
\begin{equation}\label{eq:chain-old}
  \hat{V}^{(w)}_{t+1}(y_{t-1},x_{t-1})
  \;\le\;
  (1+2\Delta'_t)\frac{(Q_t^{\mathrm{agd}})^{3/2}}{Q_{t-1}^{3/2}}\;
  V^{(w)}_t(y_{t-1},x_{t-1}),
\end{equation}
where $\Delta'_t:= Q^{\mathrm{agd}}_t - Q_{t-1} \geq0$.
 
\medskip
\noindent\emph{Combining.}\;
Substitute~\eqref{eq:chain-12} into~\eqref{eq:chain-new}:
\[
  V^{(w)}_{t+1}(y_t,x_t)
  \;\le\;
  \frac{1}{1+\hat\delta_t}(1+2\hat{\Delta}_t)\;
  \frac{Q_t^{3/2}}{(Q_t^{\mathrm{agd}})^{3/2}}\;
  \hat{V}^{(w)}_{t+1}(y_{t-1},x_{t-1}).
\]
Now substitute~\eqref{eq:chain-old}. The $Q^{\mathrm{agd}}_t$-ratios cancel:
\[
  \frac{Q_t^{3/2}}{(Q_t^{\mathrm{agd}})^{3/2}}\cdot
  \frac{(Q_t^{\mathrm{agd}})^{3/2}}{Q_{t-1}^{3/2}}
  \;=\;\frac{Q_t^{3/2}}{Q_{t-1}^{3/2}}\,,
\]
yielding
\[
  V^{(w)}_{t+1}(y_t,x_t)
  \;\le\;
  \frac{1}{1+\hat\delta_t}(1+2\hat{\Delta}_t)(1+2\Delta'_t)\;
  \frac{Q_t^{3/2}}{Q_{t-1}^{3/2}}\;
  V^{(w)}_t(y_{t-1},x_{t-1}).
\]
This establishes~\eqref{eq:combined-descent}.
\end{proof} 

Now we can rebuild the recursions for arbitrary and descent reference points given by the two propositions below.
\begin{proposition}
\label{prop:unified-general}
Let $f:\mathbb{R}^d\to\mathbb{R}$ be $L_1$-smooth. Fix $w\in\mathbb{R}^d$ and assume pathwise strong convexity \eqref{eq:pathwise-sc} holds for $s=0,\ldots,t-1$. Given $x_0=y_0$ and $L_1^{(0)}>\sigma$, with $c\in[1/2,1)$ and restarts handled by Algorithm~\ref{alg:restart-handler}:
\begin{equation}\label{eq:unified-general}
  f(y_t)-f(w)
  \;\le\;
  e^{-t/\sqrt{Q_t}}\;Q^{3/2}_t\;\psi(w)
  \;+\;
  \frac{2Q_{t}^2}{\sigma}\|\nabla f(w)\|^2,
\end{equation}
where $Q_t$ is the final condition number at iteration~$t$ and $\psi(w) = f(y_0)-f(w)+\frac{\sigma}{2}\|w-y_0\|^2$.
\end{proposition}
 
\begin{proof}
By Lemma~\ref{lem:combined-restart}(a), every iteration $\forall k<t$ (restarted or not) satisfies
\[
  V^{(w)}_{k+2}(y_{k+1},x_{k+1})
  \;\le\;
  \frac{1}{1+\hat\delta_{k+1}}\;
  \frac{Q_{k+1}^{3/2}}{Q_k^{3/2}}\;
  V^{(w)}_{k+1}(y_k,x_k)
  \;+\;
  \frac{2Q_{k+1}^{3/2}}{\sigma\,Q_k^{3/2}}\;
  \|\nabla f(w)\|^2,
\]
where $\hat\delta_{k+1}\ge 1/(\sqrt{Q_t}-1)$. Using $1/(1+\hat\delta_{k+1}) \le 1-1/\sqrt{Q_t}$ (valid since $Q_{k+1}\le Q_t$ for all $k$) and telescoping the $Q$-ratio product $\prod Q_{k+1}^{3/2}/Q_k^{3/2} = Q_t^{3/2}/Q_0^{3/2}$, by induction on $t$, we obtain that $\forall t>0$,

\begin{equation*}
\begin{aligned}
V^{(w)}_{t+1}(y_t,x_t)
\;\le\;&
\Bigl(1-\frac1{\sqrt{Q_{t}}}\Bigr)^{t}\frac{Q_{t}^{3/2}}{Q_{0}^{3/2}}\,
V^{(w)}_{0}(y_0,x_0)
\;+\;\frac{2\|\nabla f(w)\|^2}{\sigma}\;
\sum_{i=0}^{t}
\Bigl(1-\frac1{\sqrt{Q_{t}}}\Bigr)^{t-i}\frac{Q_{t}^{3/2}}{Q_{i-1}^{3/2}}\\
\;\le\;&
\Bigl(1-\frac1{\sqrt{Q_{t}}}\Bigr)^{t}\frac{Q_{t}^{3/2}}{Q_{0}^{3/2}}\,
V^{(w)}_{0}(y_0,x_0)
\;+\;\frac{2\|\nabla f(w)\|^2}{\sigma}\;
\frac{Q_{t}^{3/2}}{Q_{0}^{3/2}}\sum_{i=0}^{t}
\Bigl(1-\frac1{\sqrt{Q_{t}}}\Bigr)^{t-i}\\
\;\le\;&
\Bigl(1-\frac1{\sqrt{Q_{t}}}\Bigr)^{t}\frac{Q_{t}^{3/2}}{Q_{0}^{3/2}}\,
V^{(w)}_{0}(y_0,x_0)
\;+\;\frac{2\|\nabla f(w)\|^2}{\sigma}\;
\frac{Q_{t}^{3/2}}{Q_{0}^{3/2}}
\sqrt{Q_{t}}\Bigl(1-\Bigl(1-\frac1{\sqrt{Q_{t}}}\Bigr)^{t+1}\Bigr)\\
\;\le\;&
\Bigl(1-\frac1{\sqrt{Q_{t}}}\Bigr)^{t}Q_{t}^{3/2}\,
\psi(w)
\;+\;\frac{2\|\nabla f(w)\|^2}{\sigma}\;
Q_{t}^{2}\Bigl(1-\Bigl(1-\frac1{\sqrt{Q_{t}}}\Bigr)^{t+1}\Bigr)\\
\;\le\;&
e^{-t/\sqrt{Q_t}}Q_{t}^{3/2}\,
\psi(w)
\;+\;\frac{2Q_{t}^{2}}{\sigma}\;\|\nabla f(w)\|^2,
\end{aligned}
\end{equation*}
where we have used in the penultimate inequality that $Q_0\ge 1$ and in the final inequality that $1-(1-1/\sqrt{Q_t})^t\le 1$ and $(1-\frac{1}{x})^t \leq e^{t/x}$, $\forall t \geq0,x>1$. Finally, using the fact that $f(y_t)-f(w)\le V_{t+1}^{(w)}(y_t,x_t)$ yields~\eqref{eq:unified-general}.
\end{proof}

\begin{proposition}
\label{prop:unified-descent}
In addition to Proposition~\ref{prop:unified-general}, let $w$ satisfy $f(w)\le f(y_k)-\frac{1}{2L_1^{(k)}}\|\nabla f(y_k)\|^2$ for $k=0,\ldots,t-1$. Then:
\begin{equation}\label{eq:unified-descent}
  f(y_t)-f(w)
  \;\le\;
  e^{-t/\sqrt{Q_t}}\;(3Q_t)^m\;
  \frac{Q_t^{3/2}}{Q_0^{3/2}}\;\psi(w),
\end{equation}
where $m$ counts the total number of condition-number increases with $m\le\lfloor\log_{\min(\gamma, 1/\rho)}(L_1/L_1^{(0)})\rfloor+1$.
\end{proposition}
 
\begin{proof}
By Lemma~\ref{lem:combined-restart}(b), every iteration satisfies
\[
  V^{(w)}_{k+2}(y_{k+1},x_{k+1})
  \;\le\;
  \frac{1}{1+\hat\delta_{k+1}}\;(1+2\hat\Delta_{k+1})(1+2\Delta'_{k+1})\;
  \frac{Q_{k+1}^{3/2}}{Q_k^{3/2}}\;
  V^{(w)}_{k+1}(y_k,x_k),
\]
where $\hat\Delta_k := Q_k - Q_k^{\mathrm{agd}} \ge 0$ and $\Delta'_k := Q_k^{\mathrm{agd}} - Q_{k-1} \ge 0$. By induction:
\[
  V^{(w)}_{t+1}(y_t,x_t)
  \;\le\;
  \Bigl(1-\frac{1}{\sqrt{Q_t}}\Bigr)^{\!t}
  \Bigl[\prod_{k:\hat\Delta_k>0}(1+2\hat\Delta_k)\Bigr]
  \Bigl[\prod_{k:\Delta'_k>0}(1+2\Delta'_k)\Bigr]\;
  \frac{Q_t^{3/2}}{Q_0^{3/2}}\;\psi(w).
\]
Since $1+2\hat\Delta_k\le 3Q_t$, $1+2\Delta'_k\le 3Q_t$ and both products have at most $m$ terms combined, $\Bigl[\prod_{k:\hat\Delta_k>0}(1+2\hat\Delta_k)\Bigr]\Bigl[\prod_{k:\Delta'_k>0}(1+2\Delta'_k)\Bigr]\le(3Q_t)^m$. The bound on~$m$ follows from monotonicity of~$L_1^{(t)}$ with a $\gamma$ (constant) or $\geq 1/\rho$ factor of expansion.
\end{proof}

\subsubsection{Bounding the Inner Iterations}\label{sec:inner_iters}
Given the two restarted progress bounds above we are now ready to bound the total number of \textsc{Modified-AGD} iterations.

\begin{corollary}\label{cor:modified_cor1}
Let $f : \mathbb{R}^d \to \mathbb{R}$ be $L_1$-smooth, let $y_0 \in \R^d, \epsilon>0$ and $0<\sigma\leq L_1$. Let $(x_0^{t},y_0^{t},u,v) = \textcolor{algoblue}{\textsc{Modified-AGD}}(f,y_0,\varepsilon,L^{(0)}_1,\sigma, \gamma)$, and define
\[
  Q_t := L_1^{(t)}/\sigma, \qquad
  \bar L_1
  \;:=\;
  \max\Bigl\{L_1^{(0)},\frac{L_1}{2(1-c)\rho}\Bigr\},
  \qquad
  \bar Q := \frac{\bar L_1}{\sigma}.
\]

\begin{enumerate}
\item[(1)]
The number of AGD steps $t$ satisfies
\begin{align*}
t &\;\le\;
1+\max\!\left\{0,\;
\sqrt{Q_{t-1}}\,
\log\!\left(
\frac{2L_1^{(t-1)}\,Q_{t-1}^{3/2}(3Q_{t-1})^m\,\psi(w^{\min}_{t-1})}
{\varepsilon^2}
\right)
\right\}
\\&\;\le\;
1+\max\!\left\{0,\;
\sqrt{\frac{\bar L_1}{\sigma}}\,
\log\!\left(
\frac{2L_1^{(t-1)}\,Q_{t-1}^{3/2}(3Q_{t-1})^m\,\psi(w^{\min}_{t-1})}
{\varepsilon^2}
\right)
\right\},
\end{align*}
where $\psi(w)=f(y_0)-f(w)+\frac{\sigma}{2}\|w-y_0\|^2$ is as in line 3 of \textcolor{algoblue}{\textsc{Certify-Progress}} and $m\le\lfloor\log_{\min(\gamma, 1/\rho)}(\bar{L}_1/L_1^{(0)})\rfloor+1$. If $u,v\neq \textsc{Null}$ (non-convexity was detected), then

\item[(2)]
\begin{equation}
\label{algo sc check}
f(u)
<
f(v) + \langle\nabla f(v),u-v\rangle
+ \frac{\sigma}{2}\|u-v\|^2,
\end{equation}
where for some $0 \le j < t$, the pair $(u,v)$ is one of
\[
(y_j,x_j),\quad (w_t,x_j),\quad (y_j,w_t),\quad (w_t,y_j),
\]
with $w_t$ defined on line 8 of \textcolor{algoblue}{\textsc{Modified-AGD}}. Moreover,
\item[(3)]
\begin{equation*}
\max\{f(y_1),\ldots,f(y_{t-1}),f(u)\}
\;\le\;
f(y_0).
\end{equation*}
\end{enumerate}
\end{corollary}

\begin{proof}
We prove each claim in turn.

\paragraph{Proof of (1).}
The bound (1) is clear for $t = 1$. For $t > 1$, the algorithm has not terminated at iteration $t-1$, and so we know that neither the condition in line 16 of \textsc{Modified-AGD} nor the condition in line 4 of \textsc{Certify-Progress} held at iteration $t-1$. Thus
\begin{equation*}
\varepsilon^2 < \|\nabla f(y_{t-1})\|^2\leq 2L^{(t-1)}_1Q_{t-1}^{3/2}{(3Q_{t-1})^m}\,\psi({w^{\min}_{t-1}})\,e^{-(t-1)/\sqrt{Q_{t-1}}},
\end{equation*}
which gives the bound (1) when rearranged.

\paragraph{Proof of (2).} Now we consider the returned vectors $x^{t}_0,y^{t}_0,u$ and $v$ from \textsc{Modified-AGD}. Note that $u,v \neq \textsc{Null}$ only if $w_t \neq \textsc{Null}$. Suppose that $w_t = y_0$. Then, by line 1 of \textsc{Certify-Progress}, we have
\begin{equation*}
f(y_t) - f(w_t)>\frac{2Q_{t}^2}{\sigma}\|\nabla f(w_t)\|^2= C\,
\psi(w_t) + \frac{2Q_{t}^2}{\sigma}\|\nabla f(w_t)\|^2,
\end{equation*}
where $\psi(w_t) = \psi(y_0) = 0$ and $C:=e^{-t/\sqrt{Q_{t}}}Q_{t}^{3/2}$. Since this contradicts the progress bound in Proposition~\ref{prop:unified-general}, we obtain the certificate of non-convexity by the contrapositive of \eqref{eq:pathwise-sc}: one of the conditions must not hold for some $0 \leq s <t$, implying \textsc{Find-Witness}  will return for some $j \leq s$.

Similarly, in the case where $w_t = w^{\min}_t$, the inequality $f(w^{\min}_t) \leq f(\zeta_t) \leq f(y_t) - \frac{1}{2L^{(t)}_1} \|\nabla f(y_t)\|^2$ holds. Consequently, line 4 of \textsc{Certify-Progress} implies that:
\begin{equation*}
\frac{1}{2L^{(t)}_1}\|\nabla f(y_t)\|^2>Q_t^{3/2}{(3Q_{t})^m}\,\psi({w^{\min}_t})\,e^{-t/\sqrt{Q_t}}.
\end{equation*}
By the progress guarantee we have
\begin{equation*}
f(y_t)-f(w^{\min}_t)\geq \frac{1}{2L^{(t)}_1}\|\nabla f(y_t)\|^2 > Q_t^{3/2}{(3Q_{t})^m}\,\psi({w^{\min}_t})\,e^{-t/\sqrt{Q_t}},
\end{equation*}
contradicting inequality \eqref{eq:unified-descent} in Proposition~\ref{prop:unified-descent}.

\paragraph{Proof of (3).}
To see that the bound (3) holds, note that $f(y_s)\leq f(y_0)$ for $s = 0,...,t-1$ since the condition in line 2 of \textsc{Certify-Progress} did not hold. If $u = y_j$ for some $0 \leq j < t$ then $f(u) \leq f(y_0)$ holds trivially. 

Alternatively, if $u = w_t = w^{\min}_t$, then by the algorithm's initialization, $w^{\min}_0 \gets y_0$. Since the sequence $f(w^{\min}_t)$ is monotonically decreasing, it follows that $f(u) = f(w^{\min}_t) \leq f(y_0)$.
\end{proof}

\textbf{Order of lines 2 and 4 in} \textsc{Certify-Progress.}
We briefly justify that swapping lines 2 and 4 in \textsc{Certify-Progress} does not affect the validity of Corollary~\ref{cor:modified_cor1}. The only way the ordering could matter is if the conditions in lines~2 and~4 were both satisfied in the same iteration.
\begin{itemize}
  \item (1): This part relies on the line~4 inequality at iteration \(t-1\), and in that iteration the line~4 return is assumed not to trigger. Additionally, all inequalities used are robust to restarts.  By the fact above, the ordering doesn't matter.
  \item (2): The argument only uses the bounds in lines~1 and~4, which do not depend on $f(y_t)\le f(y_0)$ and are robust to restarts, so swapping lines~2 and~4 has no effect.
  \item (3): In the proof, we only used that the condition in line~2 fails for iterations up to \(t-1\). Since there is no return from line~4 prior to $t$, by the fact above, the order is irrelevant. When \(u = w_t^{\min}\), we rely only on the initialization, so the ordering again plays no role.
\end{itemize}

To bound the number of steps $T$ of \textsc{Modified-AGD}, note that for every $w \in \mathbb{R}^d$
\[
\psi(w) = \hat{f}(y_0) - \hat{f}(w) + \frac{\alpha(M_k)}{2}\|w - y_0\|^2
= f(y_0) - f(w) - \frac{\alpha(M_k)}{2}\|w - y_0\|^2
\le \Delta_f .
\]

Let $L$ denote the $L_1$-Lipschitz constant of the \textbf{regularized objective} $\hat{f}(x) := f(x) + \alpha \|x - p_{k-1}\|^2$ which we pass to \textsc{Modified-AGD}; with the appropriate inputs to $\alpha$, we can define $\bar L$ and $\bar Q$ as uniform bounds. In practice, we take fixed backtracking constants $c=0.5,\rho = 0.8$, and $\gamma=2$ for Algorithm \ref{alg:ABLS}, satisfying $\frac{1}{2(1-c)\rho} \leq \gamma$, which gives
\[
  \bar L
  \;=\;
  \max\Bigl\{L^{(0)},\frac{L_1+2\alpha(\bar M)}{2(1-c)\rho},
  \gamma(L_1+2\alpha(\bar M)) \Bigr\}\;=\;
  \max\Bigl\{L^{(0)},2(L_1+2\alpha(\bar M))\Bigr\},
\]
\[
  Q_t := \frac{L^{(t)}}{\alpha(M_k)},
  \qquad
  \bar Q := \frac{\bar L}{\alpha(M_0)}.
\]

Therefore, substituting $\varepsilon = \epsilon/10$ and $\sigma = \alpha(M_k) = 2M_k^{1/3}\epsilon^{2/3}$ into the guarantee \eqref{eq: t_bound} of Corollary~\ref{cor:modified_cor1_main} we obtain,
\begin{align*}
T &\leq 1+\sqrt{\max\Bigl\{Q_0,4 + \frac{L_1}{M_k^{1/3}\epsilon^{2/3}}\Bigr\}} \log_+ \left( \frac{200\bar L\,\bar Q^{3/2}(3\bar Q)^m\Delta_f}{\epsilon^2} \right)\\
&\leq 1+\sqrt{\max\Bigl\{Q_0,4 + \frac{L_1}{M_0^{1/3}\epsilon^{2/3}}\Bigr\}} \log_+ \left( \frac{200\bar L\,\bar Q^{3/2}(3\bar Q)^m\Delta_f}{\epsilon^2} \right), 
\end{align*}
where $\log_+(\cdot)$ is shorthand for $\max\{0, \log(\cdot)\}$ and $\bar{M} := \max\{\gamma L_3,M_0\}$; $M_0\leq M_k \leq \bar M$, $m \leq \lfloor \log_{\min(\gamma, 1/\rho)} ((L_1+2\alpha(\bar M)) / L^{(0)}) \rfloor + 1$.

\subsubsection{Main Complexity Bound}\label{sec:main_bound}
\begin{theorem}\label{Thm2} Let $f : \mathbb{R}^d \to \mathbb{R}$ be $L_1$-smooth and have $L_3$-Lipschitz continuous third-order derivatives. Let $p_0 \in \mathbb{R}^d$, $\Delta_f = f(p_0) - \inf_{z \in \mathbb{R}^d} f(z)$ and $0 < \epsilon^{2/3} \le \min \{ \Delta_f^{1/2} \bar M^{1/6}, L_1 / (8 \bar M^{1/3}) \}$, where $\bar{M} := \max\{\gamma L_3,M_0\}$. If we set
\begin{equation*}
\sigma = \alpha(M_k) = 2 M_k^{1/3} \epsilon^{2/3}, \quad \gamma = 2
\end{equation*}
\textcolor{algoblue}{\textsc{PF-AGD}}$(f, p_0, L_1^{(0)}, M_0, \gamma, \epsilon)$ finds a point $p_K$ such that $\|\nabla f(p_K)\| \le \epsilon$ and requires at most
\begin{equation*}
27\frac{\Delta_fL^{1/2}_1\bar M^{1/3}}{M_0^{1/6}\epsilon^{5/3}}\log \left( \frac{1473L^{5/2+m}_1 (9/2)^m\,\Delta_f}{M_0^{1/2+m/3}\epsilon^{3+2m/3}} \right) when \quad L^{(0)}\leq 2(L_1+2\alpha(\bar M))
\end{equation*}

\begin{equation*}
16\frac{\Delta_f L^{(0)^{1/2}}\bar M^{1/3}}{M_0^{1/6}\epsilon^{5/3}}\log \left( \frac{193L^{(0)^{5/2}}(3Q_0)^m\,\Delta_f}{M_0^{1/2}\epsilon^{3}} \right) when \quad L^{(0)}\geq 2(L_1+2\alpha(\bar M))
\end{equation*}
gradient evaluations, where $m \leq \lfloor \log_{\min(\gamma, 1/\rho)} ((L_1+2\alpha(\bar M)) / L^{(0)}) \rfloor + 1$.
\end{theorem}

\begin{proof}
The number of gradient evaluations is at most $2KT + \mathcal{O}(R)$, where $K$ is the number of iterations of \textsc{PF-AGD}, $T$ is the maximum number of accepted steps performed in any call to \textsc{Modified-AGD} and $R$ is the number of rejected or restarted steps.

Following the derivation in Section~\ref{sec:L_3}, the upper bound for $K$ is given by \eqref{boundonK}.
\begin{equation} 
\label{bound_on_K}
K\leq 1+10\epsilon^{-4/3}\Delta_f\bar M^{1/3}.
\end{equation}

From above we have also derived that
\begin{equation}
\label{bound_on_T}
T \leq 1+\sqrt{\max\Bigl\{Q_0,4 + \frac{L_1}{M_0^{1/3}\epsilon^{2/3}}\Bigr\}} \log_+ \left( \frac{200\bar L\,\bar Q^{3/2}(3\bar Q)^m\Delta_f}{\epsilon^2} \right). 
\end{equation}

Finally, we use $\epsilon^{2/3} \le \min\{\Delta_f^{1/2} \bar M^{1/6}, L_1/(8\bar M^{1/3})\}$ to simplify the bounds on $K$ and $T$. Using $1 \le \epsilon^{-4/3} \Delta_f \bar M^{1/3}$ reduces \eqref{bound_on_K} to $K \le 11 \epsilon^{-4/3}\Delta_f \bar M^{1/3}$.

We wish to bound the $\max$ term uniformly. Observe that we can bound the elements individually as
\[
Q_0= \frac{L^{(0)}}{\alpha(M_k)}\leq \frac{L^{(0)}}{\alpha(M_0)} \leq \frac{\bar L}{\alpha(M_0)} = \bar Q,
\]
and
\[
4 + \frac{L_1}{M_0^{1/3}\epsilon^{2/3}} = \frac{4\alpha(M_0)+2L_1}{\alpha(M_0)} \leq \frac{\bar L}{\alpha(M_0)} = \bar Q.
\]

Applying this to \eqref{bound_on_T} gives
\begin{equation*}
T \leq 1+\sqrt{\bar Q} \log \left( \frac{200\bar L\,\bar Q^{3/2}(3\bar Q)^m\Delta_f}{\epsilon^2} \right), 
\end{equation*}
where $\Delta_fL_1\epsilon^{-2}\geq8 \implies \Delta_f\bar L\epsilon^{-2}\geq16$ allows us to drop the subscript from the log. We can split into cases for a cleaner bound.

Assume first that $\bar L = L^{(0)}$, i.e., $L^{(0)}\geq 2(L_1+2\alpha(\bar M))$, so that $\bar Q = \frac{L^{(0)}}{\alpha(M_0)}$. Then
\begin{align*}
T & \leq 1+\sqrt{Q_0}\log \left(\frac{200L^{(0)^{5/2}}(3Q_0)^m\,\Delta_f}{\sigma^{3/2}\epsilon^2} \right)\\
& \leq Q_0^{1/2}\log \left( \frac{193L^{(0)^{5/2}}(3Q_0)^m\,\Delta_f}{M_0^{1/2}\epsilon^{3}} \right)\\
& \leq \frac{L^{(0)^{1/2}}}{\sqrt{2}M_0^{1/6}\epsilon^{1/3}}\log \left( \frac{193L^{(0)^{5/2}}(3Q_0)^m\,\Delta_f}{M_0^{1/2}\epsilon^{3}} \right), 
\end{align*}
where the second inequality follows from $Q_0 \geq1$ and $1+\log(x)=\log(e\cdot x)\quad \forall x>0$.

On the other hand, if $\bar L = 2(L_1+2\alpha(\bar M))\leq 3L_1$, i.e., $L^{(0)}\leq 2(L_1+2\alpha(\bar M))$, then
\begin{align*}
T &\leq 1+ \sqrt{\frac{2(L_1+2\alpha(\bar M))}{\alpha(M_0)}} \log \left( \frac{400(L_1+4 \bar M^{1/3} \epsilon^{2/3})\,\bar Q^{3/2}(3\bar Q)^m\Delta_f}{\epsilon^2} \right)\\
&\leq 1+ \sqrt{\frac{3L_1}{\alpha(M_0)}} \log \left( \frac{400(L_1+L_1/2)\,\bar Q^{3/2}(3\bar Q)^m\Delta_f}{\epsilon^2} \right)\\
& \leq 1+\sqrt{\frac{3}{2}}\frac{L^{1/2}_1}{M_0^{1/6}\epsilon^{1/3}} \log \left( \frac{600\cdot 2^{3/2}L_1\,(L_1+2\alpha(\bar M))^{3/2}(9L_1/\sigma)^m\Delta_f}{\alpha(M_0)^{3/2}\epsilon^2} \right)\\
& \leq 1+\sqrt{\frac{3}{2}}\frac{L^{1/2}_1}{M_0^{1/6}\epsilon^{1/3}} \log \left( \frac{1103L^{5/2}_19^m(L_1/(2 M_0^{1/3} \epsilon^{2/3}))^{m}\,\Delta_f}{M_0^{1/2}\epsilon^{3}} \right)\\
& \leq 1+\sqrt{\frac{3}{2}}\frac{L^{1/2}_1}{M_0^{1/6}\epsilon^{1/3}} \log \left( \frac{1103L^{5/2+m}_1 (9/2)^m\,\Delta_f}{M_0^{1/2+m/3}\epsilon^{3+2m/3}} \right).
\end{align*}

Finally, using $1\leq L_1/(8\bar M^{1/3}\epsilon^{2/3})$ so $S:=\sqrt{\frac{3}{2}}\frac{L^{1/2}_1}{M_0^{1/6}\epsilon^{1/3}}\geq \sqrt{\frac{3}{2}}\frac{L^{1/2}_1}{\bar M^{1/6}\epsilon^{1/3}}\geq \sqrt{12}$, we can multiply inside the log by $e^{1/S}\leq e^{1/\sqrt{12}}$ to absorb the additive constant $1103e^{1/\sqrt{12}} \leq 1473$.
\begin{equation*}
T \leq \sqrt{\frac{3}{2}}\frac{L^{1/2}_1}{M_0^{1/6}\epsilon^{1/3}} \log \left( \frac{1473L^{5/2+m}_1 (9/2)^m\,\Delta_f}{M_0^{1/2+m/3}\epsilon^{3+2m/3}} \right), 
\end{equation*}

Multiplying the two bounds for $T$ and $K$ with the factor $2$ gives the stated complexity bound.
\end{proof}

\paragraph{Discussion of the Complexity Bound.}
Theorem~\ref{Thm2} provides an explicit upper bound on the total number of gradient evaluations, which depends on the initialization of the $L_1$-Lipschitz constant estimate of the regularized objective $L^{(0)}$. In particular, the method requires at most
\[
27\frac{\Delta_f L_1^{1/2}\bar M^{1/3}}{M_0^{1/6}\epsilon^{5/3}}
\log \left( \frac{1473L^{5/2+m}_1 (9/2)^m\,\Delta_f}{M_0^{1/2+m/3}\epsilon^{3+2m/3}} \right)
\]
when $L^{(0)}\leq 2(L_1+2\alpha(\bar M))$, and
\[
16\frac{\Delta_f L^{(0)^{1/2}}\bar M^{1/3}}{M_0^{1/6}\epsilon^{5/3}}
\log \left( \frac{193L^{(0)^{5/2}}(3Q_0)^m\,\Delta_f}{M_0^{1/2}\epsilon^{3}} \right)
\]
when $L^{(0)}\geq 2(L_1+2\alpha(\bar M))$. Ignoring quantities independent of $\epsilon$ ($\Delta_f, L_1, L^{(0)}, M_0, \bar M, m$), the method achieves a convergence rate of $\tilde{\mathcal{O}}(\epsilon^{-5/3})$ in both cases. The two-case structure arises because when $L^{(0)}$ is sufficiently large, it already serves as an adequate upper bound on the Lipschitz constant of the true objective, and the bound depends directly on $L^{(0)}$; for smaller initializations, it reduces to a bound expressed in terms of the true smoothness parameter $L_1$.

For comparison, Carmon et al. {\cite[Theorem 2]{pmlr-v70-carmon17a}} establish the bound
\[
20 \cdot \frac{\Delta_f L_1^{1/2}L_3^{1/6}}{\epsilon^{5/3}}
\log\left(\frac{500L_1\Delta_f}{\epsilon^2}\right).
\]

\paragraph{Remark (constants).}
The larger constants in Theorem~\ref{Thm2} relative to Carmon et al. {\cite[Theorem 2]{pmlr-v70-carmon17a}} arise from three sources: (i) the factor $L_3^{1/6}$ is replaced by the adaptive quantity $\bar M^{1/3}/M_0^{1/6}$; (ii) the additional $(9/2)^m$ and $(3Q_t)^m$ terms in the inner bound, along with dependence on $m$ in the exponents, account for condition number increments during backtracking; and (iii) the adaptive condition number introduces an additional $\bar{Q}^{3/2}$ factor inside the logarithm of the inner iteration bound, inflating the logarithmic constant from 500 to 1473. While the parameter-free approach may incur larger constants and logarithmic overhead, it retains the same $\epsilon^{-5/3}$ scaling without requiring prior knowledge of $L_1$ or $L_3$.

\newpage
\section{Implementation Details and Further Experiments}\label{app:implementation}
\subsection{Implementation Details} 
All algorithms are implemented in Python with $\epsilon = 10^{-4}$. Oracle calls count only evaluations of $\nabla f$. Unless otherwise stated, single-iterate methods are initialized at $x_0 = \mathbf{0}$ and AGD-style methods at $p_0 = \mathbf{0}$ (outer loop iterates), with $N = 1{,}000$ independent trials per experiment. For reproducibility, each single seed experiment uses a random seed of $0$. The index \(k\) denotes iterates for single-iterate methods; for AGD-style methods, \(k\) and \(t\) denote outer and inner iterates, respectively.

\paragraph{Simple gradient descent (GD).} 
Simple GD uses iterates of the form
\begin{equation}
\label{gd_algo}
x_{k} = x_{k-1} - \frac{1}{L_1}\nabla f(x_{k-1}),
\end{equation}
where $L_1$ is the true Lipschitz constant of the gradient. 

\paragraph{Armijo steepest descent (SD).}
In practice, $L_1$ is often unknown and non-uniform, and therefore needs to be estimated adaptively. A standard approach is Armijo backtracking line search, initializing $L^{(0)}_1 = 1$. Given a current estimate $L^{(k)}_1$, we try the Simple GD step \eqref{armijo_algo} with $L_1:=L^{(k)}_1$ and accept it if the Armijo sufficient decrease condition holds:
\begin{equation}
\label{armijo_algo}
f(x_{k})\ \le\ f(x_{k-1})-\,\frac{c}{L^{(k)}_1}\|\nabla f(x_{k-1})\|^2,
\end{equation}
where $c\in(0,1]$; in our experiments \(c = 10^{-4}\). If \eqref{armijo_algo} fails, we reject the step and increase the curvature estimate by doubling $L^{(k)}_1\leftarrow 2L^{(k)}_1$, and retry. When a step is accepted, we keep the resulting $L^{(k)}_1$ as the starting estimate for the next iteration.

\paragraph{Semi-adaptive \textsc{AGD-Until-Guilty}.} \emph{We remark here that within the experiments section of \citet{pmlr-v70-carmon17a}, they do not use the algorithm for which they proved theoretical convergence but instead switch to a practical variant.} As $L_1$ and $L_3$ are unknown, \citet{pmlr-v70-carmon17a} estimate $L_1$ via line search as above with $c = 1$, and double the estimate $L_1^{(t)}$ within an inner iteration whenever \eqref{armijo_algo} fails. Following \citet{pmlr-v70-carmon17a}, we want the algorithm implementation to be independent of the final desired accuracy $\epsilon$ and to avoid using $L_3$. Hence, we set $\epsilon' = \|\nabla f(p_{k-1})\|/10$ and use $\alpha = \sigma = C_1 \|\nabla f(p_{k-1})\|^{2/3}$, where $C_1=0.01$ is a hyperparameter. To avoid computing $\|\nabla f(y_t)\|$ at every iteration, we move the $\|\nabla f(y_t)\|<\epsilon$ check into \textsc{Certify-Progress} and perform it only once every 5 iterations. In \textsc{Certify-Progress}, they also check $\hat{f}(x_t) + \nabla \hat{f}(x_t)^T (y_t - x_t) > \hat{f}(y_t)$ and find that this substantially increases the method’s capability of detecting negative curvature; most of the non-convexity detection is due to this check. In addition to what Carmon et al. have implemented, for AGD-style methods, we initialize $L^{(0)}_1$ using a local finite-difference estimate of gradient variation at $x_0$: for random unit vectors $u$, we compute $\frac{\|\nabla f(x_{0}+hu)-\nabla f(x_{0})\|}{h}$ with $h\ll1$ and take the maximum.

\paragraph{PF-AGD.}
As in \textsc{AGD-Until-Guilty}, we wish to remove the dependence on $\epsilon$. Hence, we set $\epsilon' = \|\nabla f(p_{k-1})\|/2$ with $\sigma = \alpha(M_{k-1}) = C_1 M_{k-1}^{1/3} \|\nabla f(p_{k-1})\|^{2/3}$, where $C_1 = 0.01$ is a hyperparameter. Instead of checking \textsc{Certify-Progress} every iteration, we invoke it every 5 iterations. For backtracking, we take $c=0.5,\rho = 0.8$ with $\gamma=2$ and $M_0=10^{-5}$.

\paragraph{Nonlinear Conjugate Gradient (CG).}
The method is given by the following recursion \cite{PolakRibiere1969}:
\[
\delta_k
= -\nabla f(x_k)
+ \max\left\{
\frac{\nabla f(x_k)^{\top}\bigl(\nabla f(x_k)-\nabla f(x_{k-1})\bigr)}
{\left\lVert \nabla f(x_{k-1})\right\rVert^{2}}
,\; 0
\right\}\delta_{k-1}
,\qquad
x_{k+1}=x_k+\eta_k\delta_k,
\]
where $\delta_0=0$ and $\eta_k$ is found via backtracking line search, as follows. If $\delta_k^{\top}\nabla f(x_k)\ge 0$ we set $\delta_k=-\nabla f(x_k)$ (truncating the recursion). We set $\eta_k=2\eta_{k-1}$ and then check whether
\[
f(x_k+\eta_k\delta_k)\le f(x_k)+\frac{\eta_k\,\delta_k^{\top}\nabla f(x_k)}{2}
\]
holds. If it does, we keep the value of $\eta_k$; otherwise, we set $\eta_k=\eta_k/2$ and repeat. The key difference from the semi-adaptive scheme used for the rest of the methods is the initialization $\eta_k=2\eta_{k-1}$, which allows the step size to grow. Performing line search is crucial for conjugate gradient to succeed, as otherwise it cannot produce approximately conjugate directions. If we use the semi-adaptive step size above, performance becomes very similar to that of gradient descent.

\subsection{Further Experiments}\label{app:further_exp}
Table~\ref{tab:pfagd_empirical_summary} summarizes how \textsc{PF-AGD} compares qualitatively to the three methods introduced above on each of the following problem families.

\begin{table}[h]
\centering
\caption{Qualitative empirical comparison of \textsc{PF-AGD} against the main baselines.}
\label{tab:pfagd_empirical_summary}
\begin{tabular}{l ccc}
\toprule
\textbf{Problem family}
& \textsc{AGD-Until-Guilty}
& Marumo et al. \cite{doi:10.1137/22M1540934}
& Nonlinear CG \\
\midrule
\multicolumn{4}{@{}l}{\emph{ML tasks}}\\
\quad Robust linear regression \cite{Beaton1974}     & $\uparrow$ & $\uparrow$ & Competitive \\
\quad MNIST neural network & $\uparrow$ & $\uparrow$ & Competitive    \\
\midrule
\multicolumn{4}{@{}l}{\emph{Convex / quadratic benchmarks}}\\
\quad Ill-conditioned quadratics    & $\uparrow$ & $\uparrow$ & Competitive       \\
\quad Regularised quadratics        & $\uparrow$ & $\uparrow$     & $\downarrow$ \\
\midrule
\multicolumn{4}{@{}l}{\emph{Non-convex test functions}}\\
\quad Qing \cite{Qing2006}                          & $\uparrow$ & $\uparrow$ & Competitive    \\
\quad Rosenbrock \cite{Rosenbrock1960AnAM}                   & Competitive  & $\uparrow$ & Competitive    \\
\quad Ackley \cite{Ackley1987}                        & $\uparrow$ & $\uparrow$     & $\downarrow$ \\
\quad Dixon--Price \cite{Dixon1989}                  & $\uparrow$ & $\uparrow$ & Competitive    \\
\quad Powell \cite{Powell1962}                       & Competitive  & $\uparrow$ & $\downarrow$ \\
\quad \textsc{SCosine}                       & Competitive  & $\uparrow$ & Competitive    \\
\bottomrule
\end{tabular}

\vspace{0.5em}
\small
\(\uparrow\): \textsc{PF-AGD} generally performs better; 
Competitive: \textsc{PF-AGD} performs comparably; 
\(\downarrow\): \textsc{PF-AGD} generally performs worse.
\end{table}

\begin{figure}[htbp]
  \centering
  \includegraphics[width=0.95\linewidth]{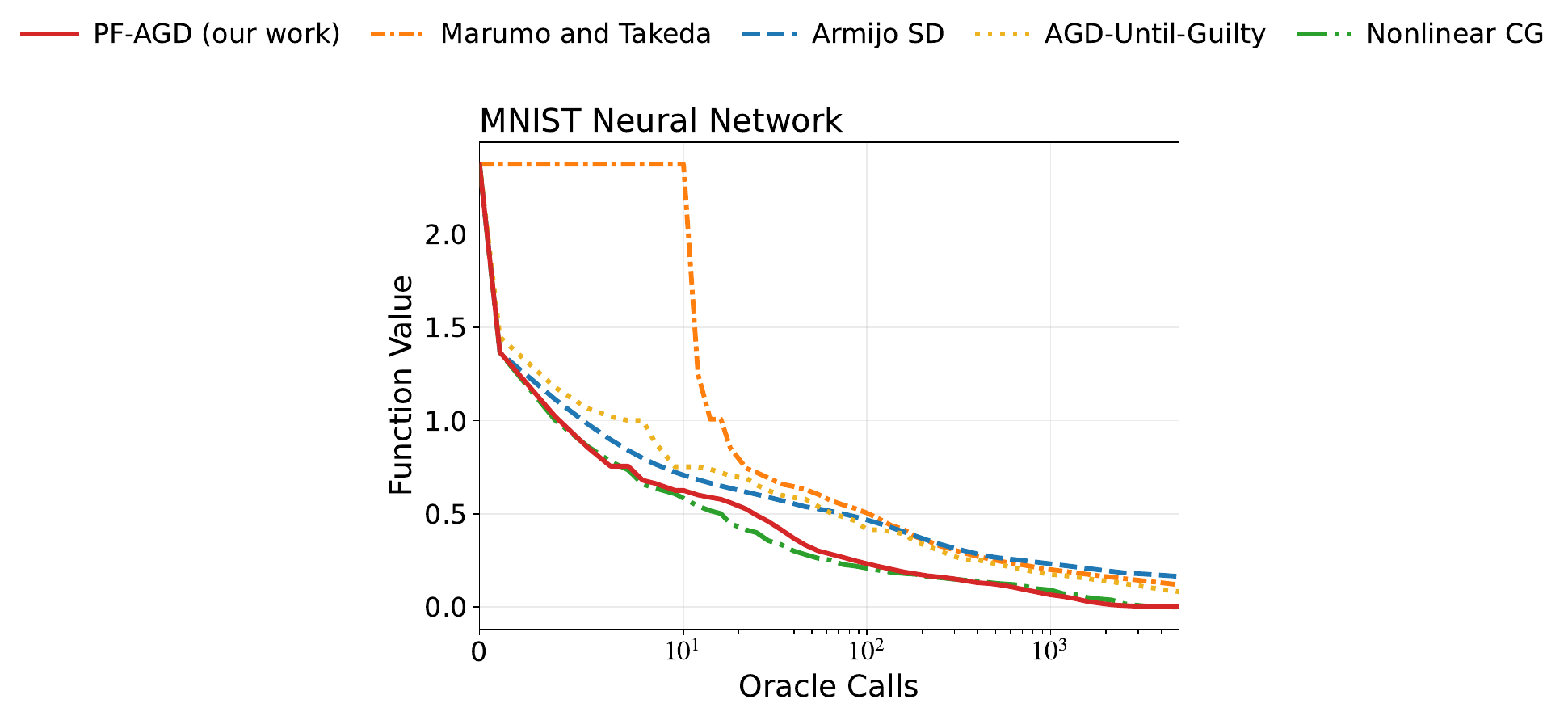}
  \caption{Performance of neural network training on MNIST (Section~\ref{section:experiments}).}
  \label{fig:mnist_nn}
\end{figure}

\subsubsection{Regularized Quadratic Functions}
Instead of $H$ positive definite, we induce negative curvature via the regularized quadratic form:

\[
f(x) = \frac{1}{2} x^{\top} Hx + b^{\top} x + \mu \|x\|_2^4, \qquad x \in \mathbb{R}^{d},
\]
where $b = [1, 1, \dots, 1]^{\top}$ and $\mu> 0$. We sample an indefinite Hessian $H\in\mathbb{R}^{d\times d}$ as $H=\bm{Q}D\bm{Q}^{\top}$, with $\bm{Q}$ Haar-distributed orthogonal as above and $D$ diagonal, containing both positive, negative, and zero eigenvalues. Given the negative eigenvalues of $H$, the quartic regularization term $\mu \|x\|_2^4$ ensures that the objective is bounded below. We choose $\mu=1$ for the objective to be sufficiently regularized so the algorithms can make good progress. The triple $(p,n,0)$ indicates the minimum number of positive, negative, and zero eigenvalues respectively, and $\mathcal{U}[\lambda_{min},\lambda_{max}]$ indicates the distribution of the remaining eigenvalues.

\paragraph{Results.} As illustrated in Figure~\ref{fig:regularized_quadratics}, nonlinear CG remains the most efficient approach, followed by \textsc{PF-AGD}. Despite the presence of negative curvature, which should theoretically favor its design, \textsc{AGD-Until-Guilty} proves to be the least effective baseline, plateauing below 75\% convergence across all four configurations. Notably, while Armijo SD matches the performance of \textsc{PF-AGD} and \citet{doi:10.1137/22M1540934} in the $U[-10,10]$ setting, this gap widens significantly as the spectral range expands to $U[-100,100]$. The inclusion of a zero eigenvalue, denoted by the $(p, n, 0)$ configurations, appears to have a negligible impact on relative performance, suggesting that the zero eigenvalue (degenerate) direction is quickly identified and avoided by all methods.

\begin{figure}[htbp]
  \centering
  \includegraphics[width=0.84\linewidth]{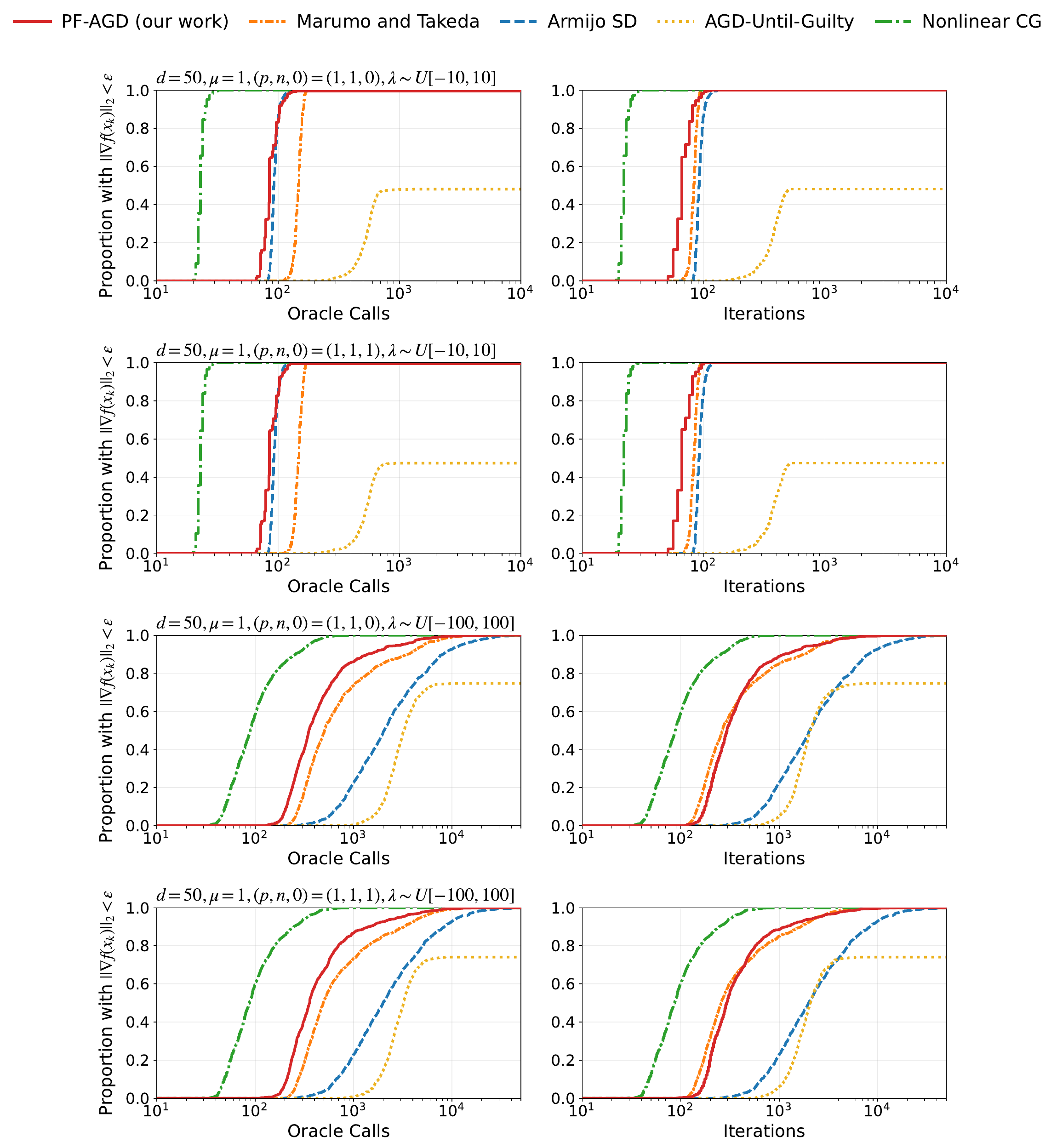}
  \caption{Performance on regularized quadratics with indefinite Hessians.}
  \label{fig:regularized_quadratics}
\end{figure}

\begin{figure}[t]
  \centering
  \includegraphics[width=0.95\linewidth]{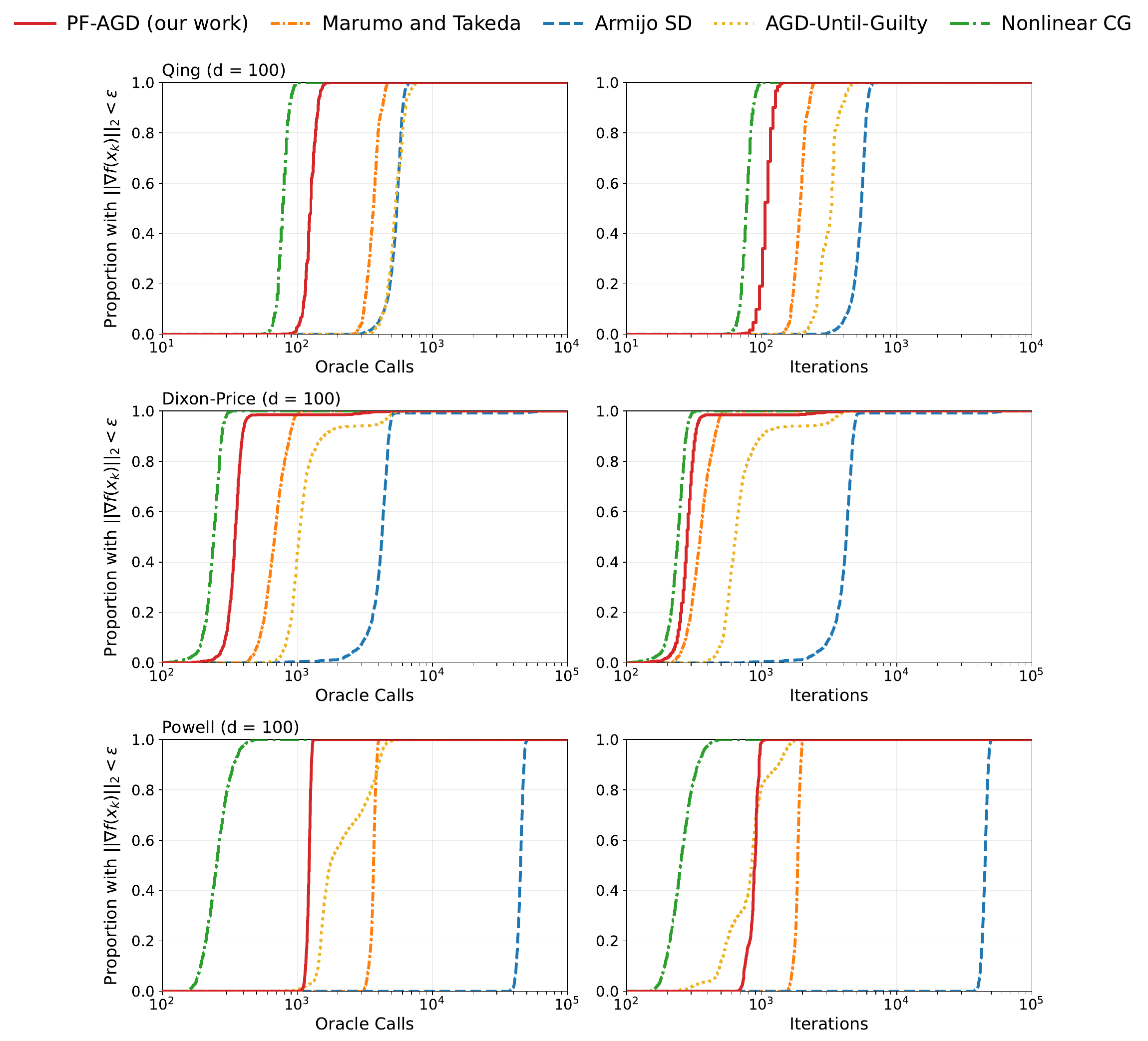}
  \caption{Empirical CDFs of oracle calls for the Qing (Appendix~\ref{subsec:qing}), Dixon--Price (Section~\ref{section:experiments}), and Powell (Appendix~\ref{subsec:powell}) functions.}
  \label{fig:dp-powell-qing-cdf}
\end{figure}
\clearpage

\subsubsection{Qing Function}
\label{subsec:qing}
For dimension $d \ge 1$, the Qing function \cite{Qing2006} is
\begin{equation*}
f(x) = \sum_{i=1}^{d} (x_i^2 - i)^2, \quad d \in \{4, 20, 60, 100\}.
\end{equation*}
The global minimum $f(x^*) = 0$ is attained at $x_i^* \in \{\pm\sqrt{i}\}$ for all $i \in \{1, \dots, d\}$, resulting in $2^d$ global minimizers. We set $x^* = [\sqrt{1}, \dots, \sqrt{d}]^\top$ and initialize $x_0, p_0 \sim \mathcal{N}(x^*, 10^{-1}I_d)$.

\paragraph{Results.} 
As shown in Figure~\ref{fig:dp-powell-qing-cdf}, all evaluated methods converged within 800 oracle calls. Nonlinear CG and \textsc{PF-AGD} proved most efficient, necessitating approximately 100 and 180 calls, respectively; whereas \citet{doi:10.1137/22M1540934} required 400 calls, and \textsc{AGD-Until-Guilty} and Armijo SD exhibited the slowest convergence with near-identical trajectories. Single-seed analysis, illustrated in Figure~\ref{fig:qing}, corroborates these findings. At $d=100$, the performance of \textsc{AGD-Until-Guilty} and Armijo SD remains closely aligned, which is consistent with the aggregate CDF results. However, at $d=4$, Armijo SD marginally outperforms the accelerated variants. This suggests that the overhead associated with acceleration may be counterproductive in low-dimensional regimes. As dimensionality increases, \textsc{PF-AGD} recovers its advantage, trailing nonlinear CG by a margin of at most 30 oracle calls. Upon termination, all methods reached similar objective values, successfully converging to the global minimizer at zero.

\begin{figure}[htbp]
  \centering
  \includegraphics[width=0.75\linewidth]{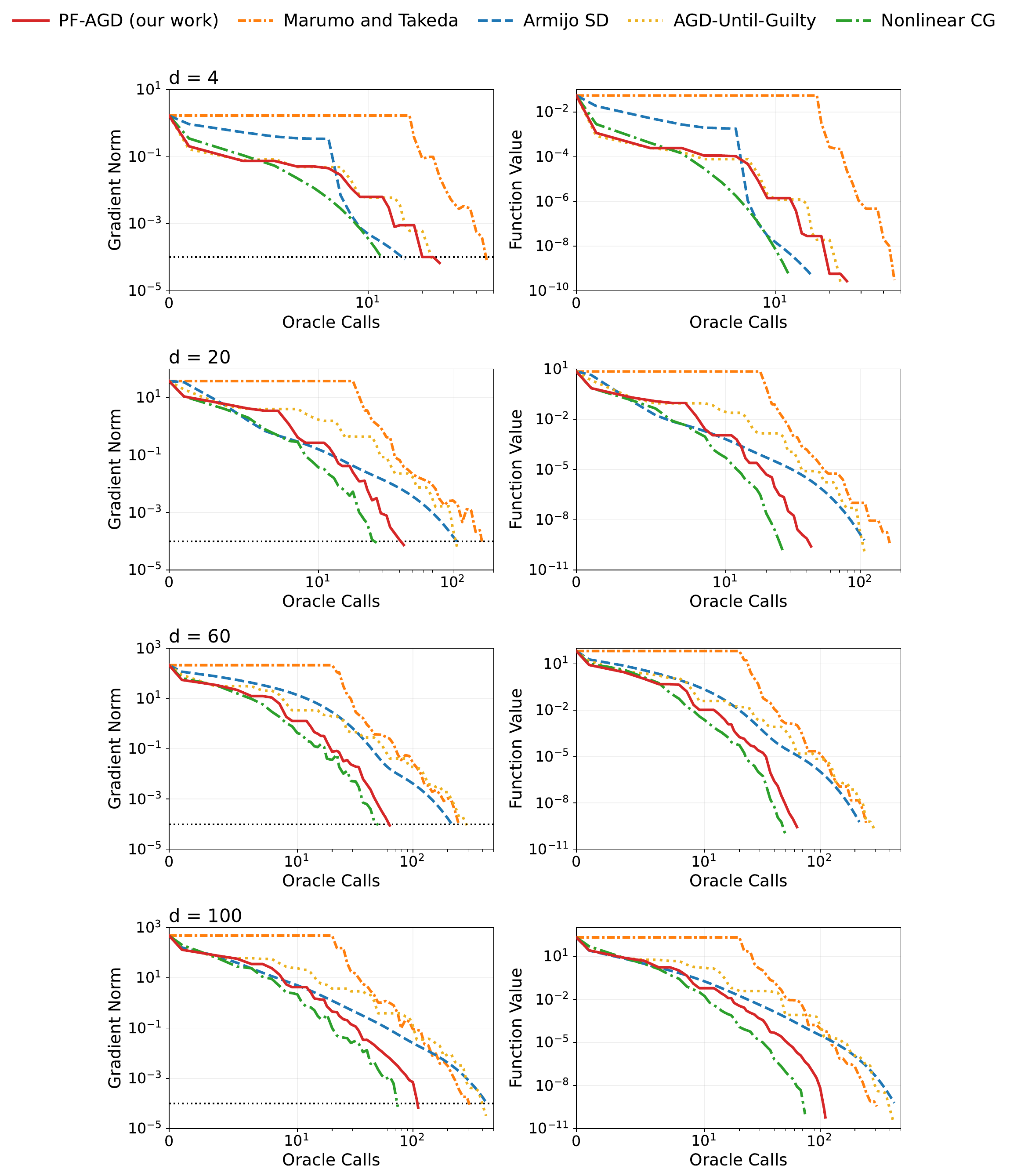}
  \caption{Performance on the Qing function.}
  \label{fig:qing}
\end{figure}

\subsubsection{Rosenbrock Test Function}
The Rosenbrock function \cite{Rosenbrock1960AnAM} is a non-convex function whose global minimum lies inside a narrow, parabolic-shaped (flat) valley. Finding the valley is straightforward but converging to the global minimum is difficult. Notably, the Rosenbrock function is not (globally) $L_1$-smooth. This makes the function a good performance test for the robustness of our $L_1$ estimates:
\[
f(x)=\sum_{i=1}^{d-1}\Bigl[100\bigl(x_{i+1}-x_i^2\bigr)^2+\bigl(1-x_i\bigr)^2\Bigr],
\qquad x=[x_1,\dots,x_d]^{\top}\in\mathbb{R}^d.
\]
The global minimizer is $x^*=[1, 1, \dots, 1]^{\top}$, where $f(x^*)=0$. For \(d=2\), we use the standard (challenging) initialization \(x_0,p_0 = [-1.2,1]^{\top}\); for \(d\in\{10,20,50\}\), we set \(x_0,p_0 = [-1.2,1,\ldots,-1.2,1]^{\top}\). We report both $f(x_t)$ and $\|\nabla f(x_t)\|$ as a function of the number of oracle calls, stopping once $\|\nabla f(x_t)\| \leq \epsilon$.

\paragraph{Results.} Figure~\ref{fig:rosenbrock} illustrates that nonlinear CG remains the most efficient method across all evaluated dimensions; however, the performance gap between it and the AGD-style methods narrows as the dimensionality increases. Overall, \textsc{PF-AGD} performs slightly worse than \textsc{AGD-Until-Guilty}, although it is notable that negative curvature is only detected and exploited in the $d=2$ case. In contrast, steepest descent proves to be the least efficient approach followed by \citet{doi:10.1137/22M1540934}, both exhibiting prolonged stagnation phases despite an Armijo SD having an initial rate of progress that matches nonlinear CG. Interestingly, the relative performance of the algorithms appears robust to dimensionality, with negligible impact on their respective convergence rates.

\begin{figure}[htbp]
  \centering
  \includegraphics[width=0.84\linewidth]{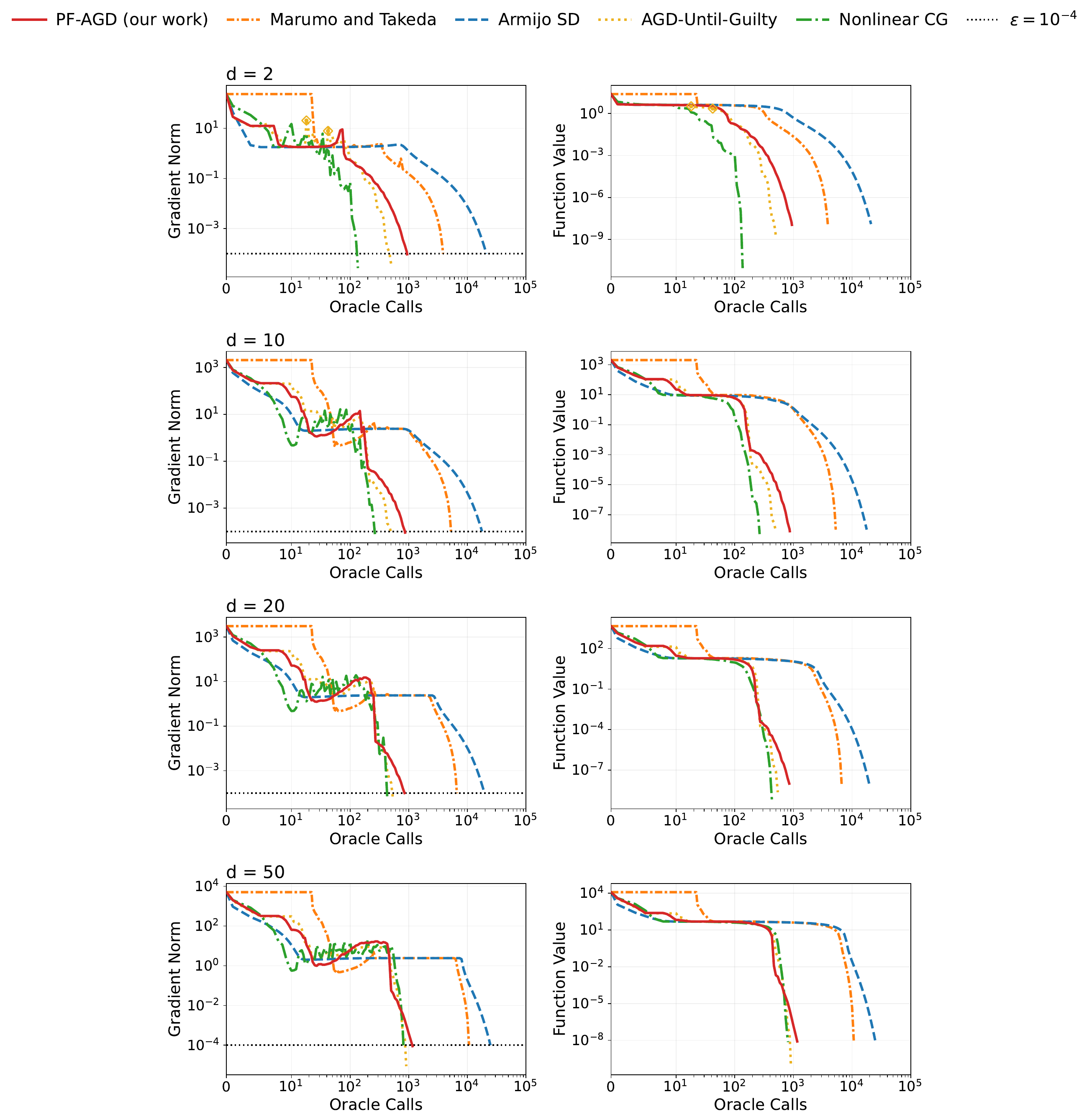}
  \caption{Performance on the Rosenbrock function.}
  \label{fig:rosenbrock}
\end{figure}

\subsubsection{Ackley Test Function}
We evaluate the algorithms on the Ackley function \cite{Ackley1987}:
\[
f(x) = -20 \exp\!\left(-0.2 \sqrt{\frac{1}{d}\sum_{i=1}^d x_i^2}\right)
      - \exp\!\left(\frac{1}{d}\sum_{i=1}^d \cos(2\pi x_i)\right)
      + e + 20, \quad x=[x_1,\dots,x_d]^{\top}\in\mathbb{R}^d.
\]
The origin is the global minimizer, with $f(x^*) = 0$, but the landscape around it features a near-flat outer region with high-frequency oscillations that induce many local optima where optimization methods may become trapped. As such, we initialize all methods near the global minimizer with a small perturbation to break the symmetry, setting $x_0,p_0 = [-1, -1, 0, \dots, 0]^{\top} + \mathcal{N}(\mathbf{0},10^{-4}I)$.

\paragraph{Results.} 
As shown in Figure~\ref{fig:ackley}, AGD-based methods achieve rapid initial reductions 
in gradient norm for $d \in \{2, 10, 20\}$; however, the high terminal function values 
suggest convergence to local minima rather than the global optimum. Among these, \textsc{PF-AGD} is the most oracle-efficient. Armijo SD fails to attain $\epsilon=10^{-4}$ accuracy for $d \in \{10, 20\}$, likely due to vanishing gradients in the flat regions characteristic of the Ackley landscape. At $d=50$, the additional degrees of freedom enable all methods to locate the global minimum. All algorithms encounter an initial plateau, though the AGD variants exhibit the most prolonged stagnation. Notably, \textsc{PF-AGD} and \citet{doi:10.1137/22M1540934} escapes this plateau after approximately $10^2$ oracle calls, substantially earlier than \textsc{AGD-Until-Guilty}, which requires around $3\times 10^4$ calls, indicating considerably greater robustness. The eventual descent of \textsc{AGD-Until-Guilty} coincides with several instances of negative curvature detection, as indicated by the markers.

\begin{figure}[htbp]
  \centering
  \includegraphics[width=0.8\linewidth]{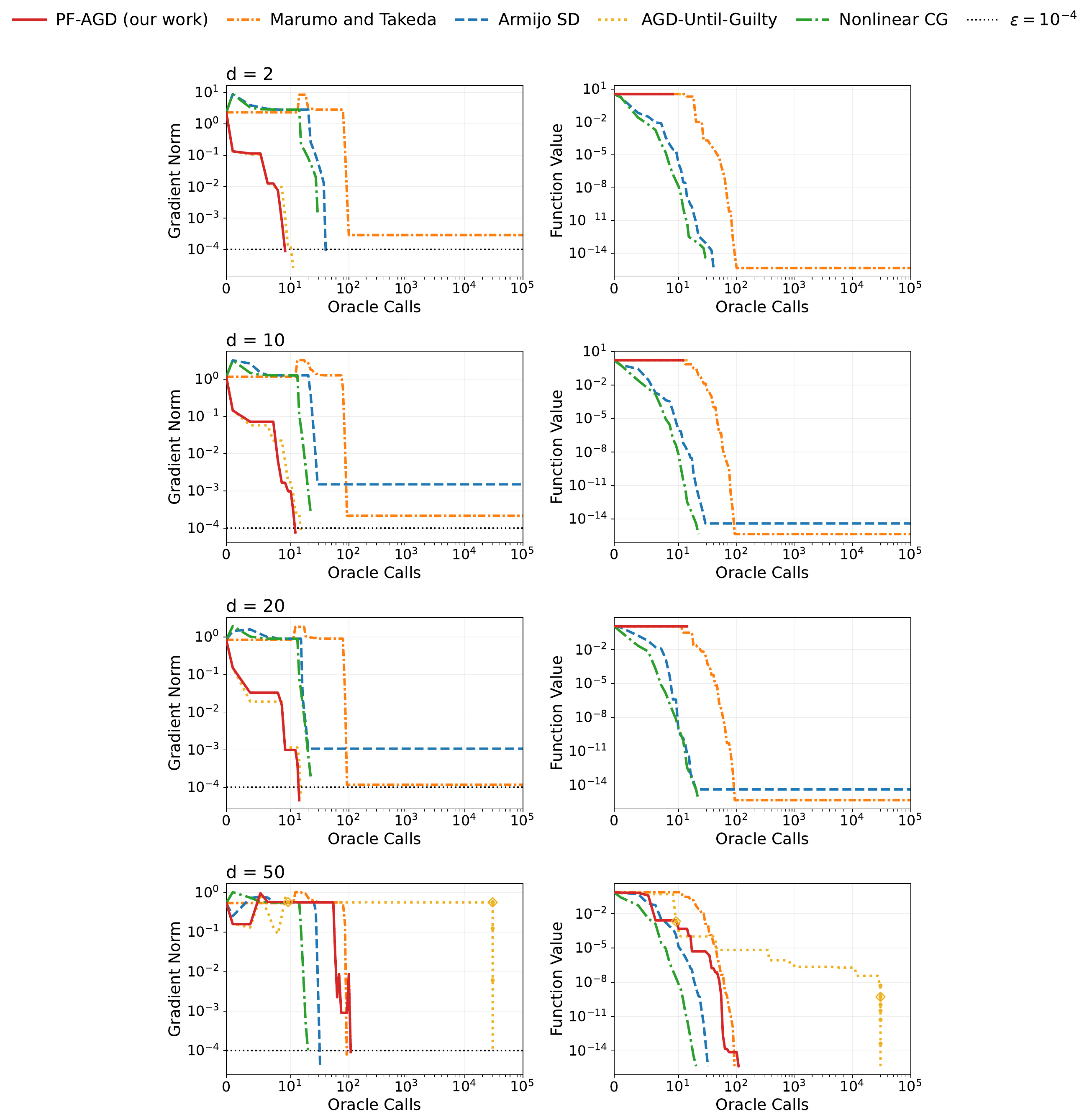}
  \caption{Ackley function: the dots correspond to negative curvature detection and the diamonds correspond to negative curvature exploitation (i.e., when $f(b^{(2)}) < f(b^{(1)})$).}
  \label{fig:ackley}
\end{figure}

\subsubsection{Powell Function}
\label{subsec:powell}
The Powell function \cite{Powell1962} is defined for dimensions $d$ that are multiples of $4$ (with $d\ge 4$). Partition variables into blocks
$(x_{4i-3},x_{4i-2},x_{4i-1},x_{4i})$, $i=1,\dots,d/4$, and set
\[
f(x) = \sum_{i=1}^{d/4}\Big[(x_{4i-3}+10x_{4i-2})^2 + 5(x_{4i-1}-x_{4i})^2 + (x_{4i-2}-2x_{4i-1})^4 + 10(x_{4i-3}-x_{4i})^4\Big].
\]
The global minimizer is located at the origin with $f(x^*) = 0$ and we initialize our algorithms at $x_0,p_0 \sim \mathcal{N}(x^*,10^{-1}I_d)$, testing with $d\in \{4,20,60,100\}$.

\paragraph{Results.} Figure~\ref{fig:dp-powell-qing-cdf} illustrates that at $d=100$, this problem proves particularly challenging for Armijo SD; indeed, almost no runs achieve convergence until a budget of $4 \times 10^4$ oracle calls. In contrast, nonlinear CG remains the most efficient optimizer, followed by \textsc{PF-AGD}, \textsc{AGD-Until-Guilty}, and the method proposed by \citet{doi:10.1137/22M1540934}. These trends are further corroborated by the single seed trace in Figure~\ref{fig:powell}. While \textsc{PF-AGD} and \textsc{AGD-Until-Guilty} exhibit nearly identical trajectories, Armijo SD requires approximately $5 \times 10^4$ calls to reach the target accuracy. Notably, the performance of \citet{doi:10.1137/22M1540934} is characterized by two distinct spikes in the gradient norm, and nonlinear CG displays high variance in its convergence.

\begin{figure}[htbp]
  \centering
  \includegraphics[width=0.84\linewidth]{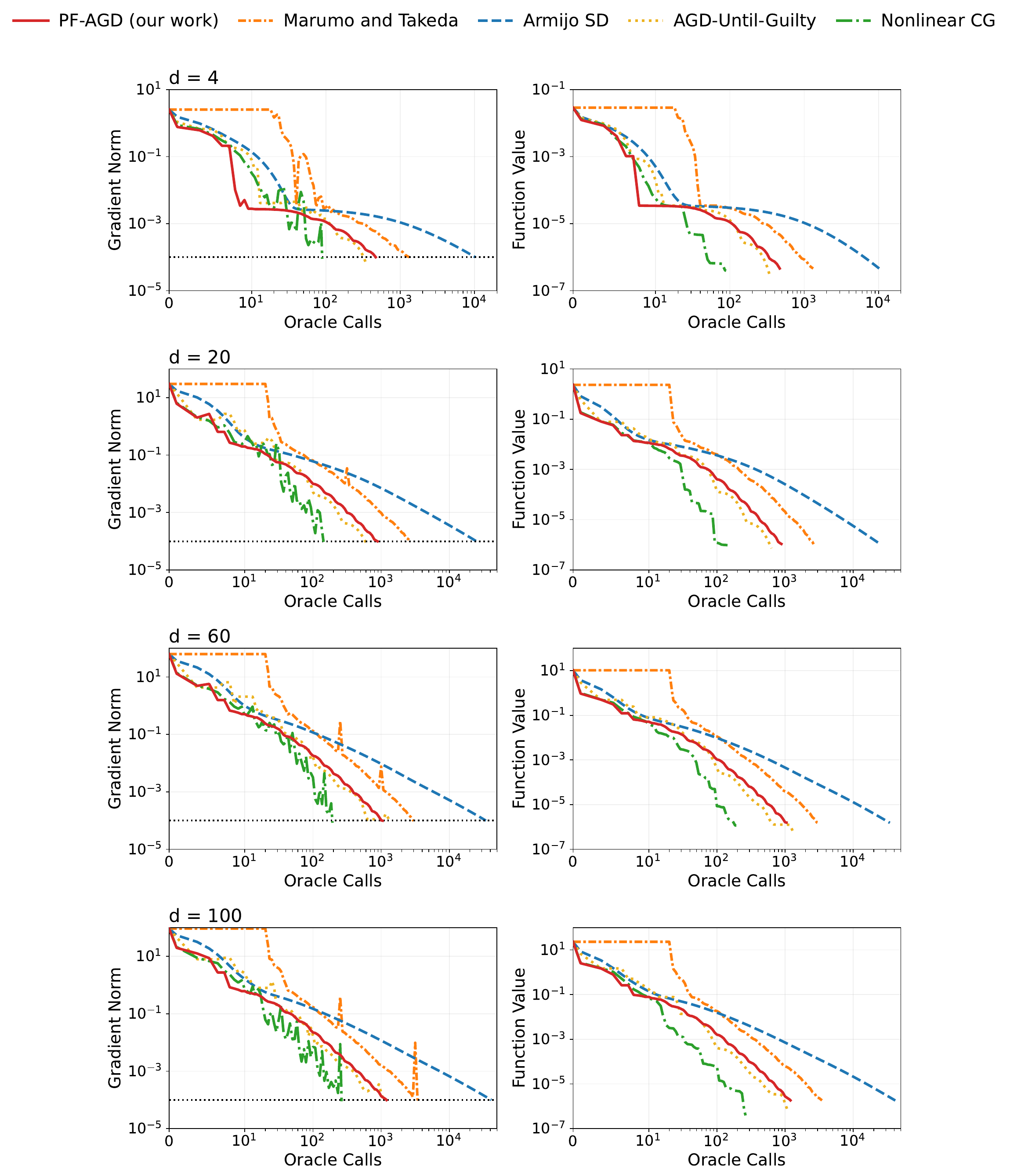}
  \caption{Performance on the Powell function.}
  \label{fig:powell}
\end{figure}

\subsubsection{\textsc{SCosine} Function}
For dimension $d \ge 1$, the \textsc{SCosine} function is
\[
f(x) = \sum_{i=1}^{d-1} \cos^{2}\left(x_i^{2} - \frac{x_{i+1}}{2}\right),
\quad x \in \mathbb{R}^d.
\]
There are infinitely many global minimizers, with $f(x^*) = 0$. One particular minimizer can be generated via the recursion, $x_1=0, x_{i+1} = 2x^2_i-\pi$, where each $x_i$ is a polynomial in $\pi$. We initialize the iterates at $x_0,p_0= [1,\dots,1]^\top$.

\paragraph{Results.}
Figure~\ref{fig:SCOSINE} displays the gradient norm and function value relative to oracle calls for dimensions $d \in \{10, 20, 50, 100\}$. The \textsc{SCosine} landscape is highly non-convex and characterized by abundant negative curvature. Armijo SD stalls across all tested dimensions, failing to progress towards the global minimizer; it requires approximately $2 \times 10^{3}$ oracle calls merely to approach $\epsilon$. Conversely, all other evaluated methods achieve function values below $10^{-8}$ within 40 oracle calls, with the gradient norm reaching the desired accuracy.

\begin{figure}[htbp]
  \centering
  \includegraphics[width=0.85\linewidth]{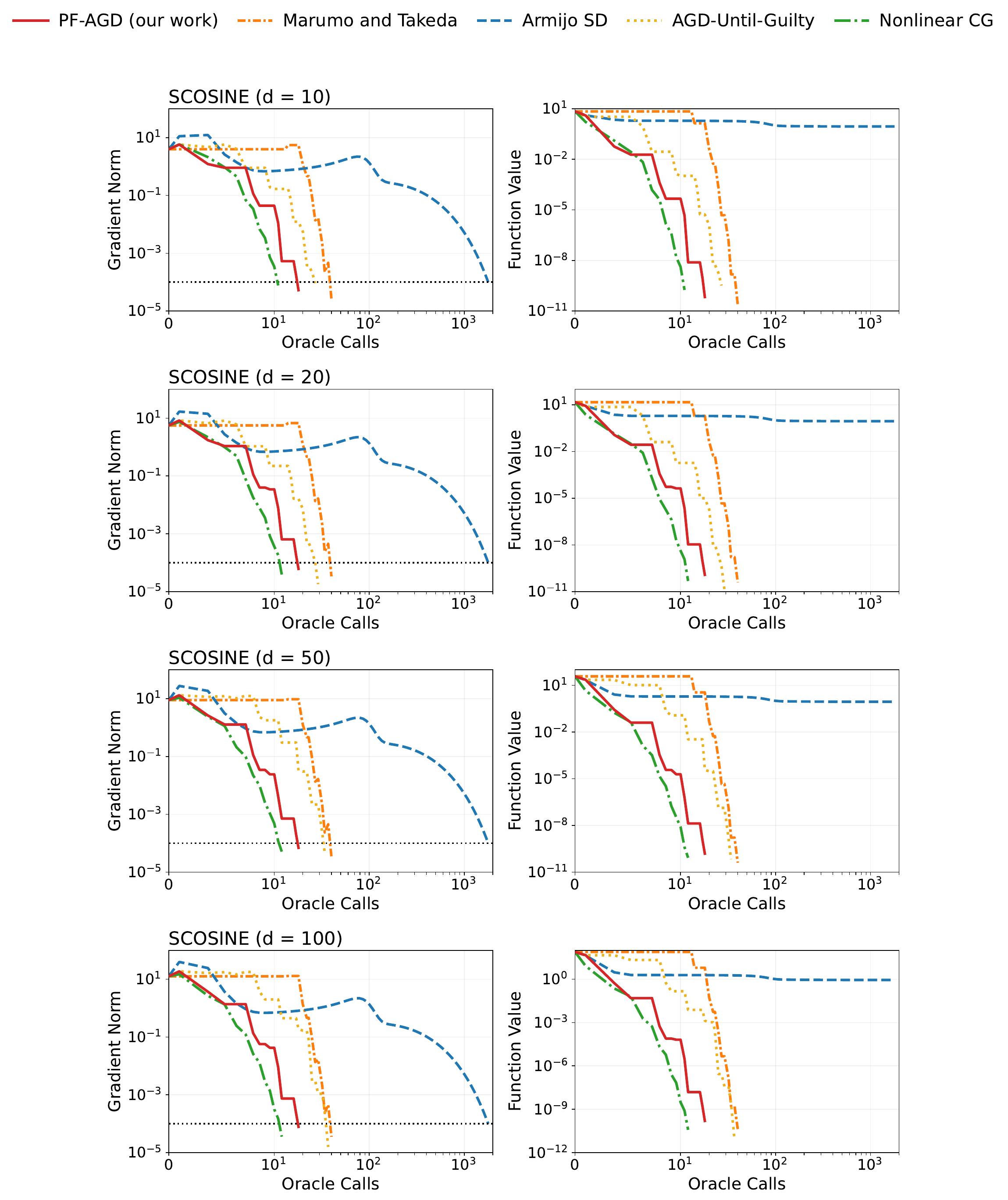}
  \caption{Performance on the \textsc{SCosine} function.}
  \label{fig:SCOSINE}
\end{figure}


\end{document}